\documentclass[12pt]{amsart}

\usepackage{amsmath}

\usepackage[T1]{fontenc}
\usepackage[utf8]{inputenc}
\usepackage[text={430.8pt, 576pt}, centering]{geometry}

\usepackage{lmodern}

\usepackage{bm}
\usepackage{amssymb, amsfonts}
\usepackage{paralist} 

\usepackage{comment}

\usepackage[english]{babel} 
\usepackage{csquotes} 
\usepackage[backend=bibtex, giveninits=true, style=alphabetic]{biblatex}

\usepackage{tabls}

\usepackage{tikz}
\usetikzlibrary{arrows}

\usepackage[breaklinks=true, colorlinks=true]{hyperref}

\newtheorem{theorem}{Theorem}[section]
\newtheorem{lemma}[theorem]{Lemma}
\newtheorem{corollary}[theorem]{Corollary}
\newtheorem{proposition}[theorem]{Proposition}

\theoremstyle{definition}
\newtheorem{definition}[theorem]{Definition}
\newtheorem{remark}[theorem]{Remark}
\newtheorem{question}[theorem]{Question}

\newtheorem{example}[theorem]{Example}

\newenvironment{statement}{\begin{quote}}{\end{quote}}

\newenvironment{verlong}{}{}

\excludecomment{verlong}
\includecomment{vershort}

\newcommand{\NN}{\mathbb{N}}
\newcommand{\ZZ}{\mathbb{Z}}

\newcommand{\with}{\,\colon\,}

\DeclareMathOperator{\lk}{lk}
\DeclareMathOperator{\dl}{dl}

\newcommand{\set}[1]{\left\{ #1 \right\}}

\newcommand{\tup}[1]{\left( #1 \right)}
\newcommand{\ive}[1]{\left[ #1 \right]}

\newcommand{\chitil}{\widetilde{\chi}}

\newcommand{\Pm}[1]{\mathcal{PM}\left(#1\right)}
\newcommand{\Pf}[1]{\mathcal{PF}\left(#1\right)}

\newcommand{\NS}{V^{\prime}}

\begin{document}

\title{The path-missing and path-free complexes of a directed graph}

\date{June 10, 2026}

\author{Darij Grinberg}

\address{Drexel University,
Korman Center, Room 263,
15 S 33rd Street,
Philadelphia, PA, 19104,
USA}
\email{darijgrinberg@gmail.com}

\author{Lukas Katth\"an}

\address{Goethe-University Frankfurt, Institute of Mathematics, Robert-Mayer-Str. 10, 60325 Frankfurt am Main, Germany}
\email{lukaskatthaen@gmx.de}

\author{Joel Brewster Lewis}

\address{George Washington University,
Phillips Hall, Room 707,
801 22nd St NW,
Washington, DC 20052 USA}
\email{jblewis@gwu.edu}


\thanks{DG thanks Galen Dorpalen-Barry and Richard Stanley for interesting conversations and the Mathematisches Forschungsinstitut Oberwolfach as well as the Institut Mittag--Leffler (Djursholm) for their hospitality.
The authors thank the referees for informative and helpful suggestions.}

\begin{abstract}
We study two simplicial complexes arising from a directed graph
$G = (V, E)$ with two chosen vertices $s$ and $t$:
the \emph{path-free complex},
consisting of all subsets $F \subseteq E$ that contain no path
from $s$ to $t$, and the \emph{path-missing complex},
its Alexander dual.
Using discrete Morse theory, we prove that both complexes have
well-behaved homotopy types -- either contractible or
homotopy-equivalent to spheres.
\end{abstract}

\maketitle

\section{Introduction}

In this paper, we introduce two simplicial complexes associated to a directed graph with specified source and sink nodes, which we call the \emph{path-free} and \emph{path-missing} complexes of the graph.
These complexes, whose definitions were inspired in a natural way by the theory of network flows and cuts, are not constructed \emph{a priori} to be well behaved.
Our main result will show that, nevertheless, they have a very nice topological/homotopical structure.
The two complexes add to a long history of graph-related complexes, which are often defined in terms of some form of connectivity or independence (cf. \cite{BDMRSX}, \cite[Section 5]{Forman-user}, \cite{Jonsson}, \cite{Kozlov-CAT}, \cite{EhrHet}, or the recently proved Kalai--Meshulam conjecture \cite{CSSS}, \cite{Engstrom20}, \cite{Kim22}, \cite{ZhangHu25}).
The path-free complex can be viewed as a ``local'' version of the ``complex of non-connected graphs'' from \cite[Section 5]{Forman-user}, in which the connectivity of the entire graph is replaced by the existence of paths between two given vertices. Our graphs here will be directed; this adds another twist on the existing theory, which is usually concerned with undirected graphs.

The rest of this introduction is devoted to the precise definitions of the path-free and path-missing complexes, and the statements of our main results.


\subsection{\label{sect.defs}General definitions}

In the following, all graphs are directed multigraphs, and self-loops are allowed.
Let $G = \tup{V,E, s,t}$ be a directed graph\footnote{We recall that a \emph{directed graph} consists of a set $V$ of vertices and a set $E$ of arcs. Each arc has a \emph{source} and a \emph{target}; if these are equal, the arc is called a \emph{self-loop}. An arc can carry additional information beyond its source and target; thus, parallel arcs are allowed.}
with vertex set $V$, arc set $E$, and two distinguished vertices $s,t \in V$.
A \emph{walk} (of $G$) means a sequence $\tup{v_0, e_1, v_1, e_2, v_2, \ldots, e_n, v_n}$, where $n$ is a nonnegative integer and where $v_0, v_1, \ldots, v_n \in V$ and $e_1, e_2, \ldots, e_n \in E$ have the property that each $e_i$ is an arc from $v_{i-1}$ to $v_i$.
This walk is said to have the vertices $v_0, v_1, \ldots, v_n$ and the arcs $e_1, e_2, \ldots, e_n$, and it is furthermore said to be a walk from $v_0$ to $v_n$.
A \emph{path} means a walk that does not visit any vertex more than once.
If $u, v \in V$, then a \emph{$u-v$-walk} means a walk from $u$ to $v$.
A \emph{$u-v$-path} means a $u-v$-walk that is a path.

We shall use the language of (abstract) simplicial complexes.
For an introduction to this language, see \cite[Section 9]{Bjo} (but unlike \cite{Bjo}, we shall not exclude the empty set in our simplicial complexes).
We recall the most fundamental notions of this language:
A \emph{simplicial complex} means a pair $\tup{W, \Delta}$ consisting of a finite set $W$ and a subset $\Delta$ of the powerset $2^W$ of $W$ such that the following axiom holds:\footnote{In terms of posets, this axiom says that $\Delta$ is an order ideal (i.e., a down-set) of the Boolean lattice $2^W$ (ordered by inclusion).}
\begin{statement}
If $A \in \Delta$ and $B \subseteq A$, then $B \in \Delta$.
\end{statement}
(Thus, $\varnothing \in \Delta$ if $\Delta$ itself is not empty; but $\Delta$ is allowed to be empty.)
Unlike various authors, we do not require $\set{w} \in \Delta$ for all $w \in W$.
We shall often denote the simplicial complex $\tup{W, \Delta}$ by $\Delta$; that is, we omit $W$ from the notation.
We will refer to the set $W$ as the \emph{ground set} of the complex $\tup{W, \Delta}$, and say that $\tup{W, \Delta}$ is a complex \emph{on the ground set $W$}.
The elements of $\Delta$ (which are, themselves, subsets of $W$) are called the \emph{faces} of the complex $\Delta$.

Given a graph $G$, we say that a subset $F \subseteq E$ \emph{contains} a path $p$ if $F$ contains each arc of $p$.  Let us now define the two simplicial complexes we are going to study.

\begin{definition} \label{def.PfPm}
    The \emph{path-free complex} $\Pf{G}$ and the \emph{path-missing complex} $\Pm{G}$ are the following simplicial complexes on the ground set $E$:
    \begin{align*}
        \Pf{G} &:= \set{ F \subseteq E \with F \text{ contains no $s-t$-path}} ; \\
        \Pm{G} &:= \set{ F \subseteq E \with E\setminus F \text{ contains an $s-t$-path}} .
    \end{align*}
\end{definition}

\begin{example} \label{exa.PfPm1}
Let $G$ be the following directed graph:
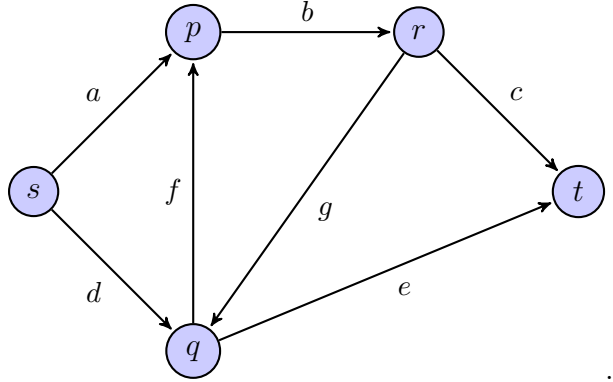
\[
\begin{tikzpicture}[->,>=stealth',shorten >=1pt,auto,node distance=3cm, thick,main node/.style={circle,fill=blue!20,draw}]
\begin{scope}[every edge/.style={draw=black,very thick}]
 \node[main node] (1) {$s$};
 \node[main node] (2) [above right of=1] {$p$};
 \node[main node] (3) [below right of=1] {$q$};
 \node[main node] (4) [right of=2] {$r$};
 \node[main node] (5) [below right of=4] {$t$};
 \path (1) edge node {$a$} (2) (2) edge node {$b$} (4) (4) edge node {$c$} (5) (1) edge node[below left] {$d$} (3) (3) edge node[below right] {$e$} (5) (3) edge node {$f$} (2) (4) edge node {$g$} (3);
\end{scope}
\end{tikzpicture}.
\]
Then the faces of the simplicial complex $\Pf{G}$ are the sets
\[
\set{b,c,e,f,g},\ \set{a,c,e,f,g},\ \set{b,c,d,g},\ \set{a,c,d,f,g},\ \set{a,b,e,f},\ \set{a,b,d,f,g}
\]
as well as all their subsets.
Meanwhile, the faces of the simplicial complex $\Pm{G}$ are the sets
\[
\set{d,e,f,g},\ \set{c,d,f},\ \set{a,b,c,f,g},\ \set{a,e,g}
\]
as well as all their subsets.
\end{example}

\subsection{\label{sect.main}Statements of the main results}

\subsubsection{\label{subsect.main.euler}The Euler characteristic}


The \emph{Euler characteristic} $\chi\tup{\Delta}$ of a simplicial complex $\Delta$ is defined to be the alternating sum $f_0 - f_1 + f_2 - f_3 \pm \cdots = \sum\limits_{i \geq 0} \tup{-1}^i f_i$, where $f_i$ denotes the number of faces $I \in \Delta$ having size $\# I = i+1$. (The faces of $\Delta$ are themselves finite sets, so they have well-defined sizes. In combinatorial topology, the \emph{dimension} of a face $I$ is defined to be its size minus $1$, so that $f_i$ counts the faces $I \in \Delta$ of dimension $i$.)

The \emph{reduced Euler characteristic} $\chitil\tup{\Delta}$ of a simplicial complex $\Delta$ is defined to be the alternating sum
\begin{align}
- f_{-1} + f_0 - f_1 + f_2 - f_3 \pm \cdots
= \sum\limits_{i \geq -1} \tup{-1}^i f_i
= \sum_{I \in \Delta} \tup{-1}^{\# I - 1} ,
\label{eq.chitil.def}
\end{align}
where $f_i$ denotes the number of faces $I \in \Delta$ having size $\# I = i+1$. If $\Delta \neq \varnothing$, then the two characteristics are related by the equality $\chitil\tup{\Delta} = \chi\tup{\Delta} -1$, since $f_{-1} = 1$.

A walk of $G$ is said to be
\begin{itemize}
\item \emph{nontrivial} if it contains at least one arc;
\item \emph{closed} if it starts and ends at one and the same vertex\footnote{Here, of course, a walk is said to \emph{start} at $u$ and \emph{end} at $v$ if it is an $u-v$-walk.};
\item a \emph{cycle} if it is nontrivial and closed and has the property that all but the last of its vertices are distinct (i.e., if $v_0, v_1, \ldots, v_m$ are its vertices, then $v_0, v_1, \ldots, v_{m-1}$ are distinct).
\end{itemize}
For example, the graph $G$ in Example~\ref{exa.PfPm1} has a cycle, which contains the vertices $p, r, q, p$ and the arcs $b, g, f$. 

A vertex $v$ of $G$ is said to be a \emph{nonsink} if there exists at least one arc of $G$ whose source is $v$.
We denote by $\NS$ the set of all nonsinks of $G$.

A \textit{useless arc} of $G$ will mean an arc that belongs to no $s-t$-path. In particular, any self-loop is a useless arc.
For example, the graph $G$ in Example~\ref{exa.PfPm1} has no useless arcs, but if we remove its arc $d$, then the resulting graph has $f$ as a useless arc.

We can now describe the reduced Euler characteristics of $\Pm{G}$ and $\Pf{G}$.

\begin{theorem}
\label{thm.PM.chitil} \hfill
\begin{enumerate}
    \item[(a)] If $G$ contains a useless arc or a cycle or satisfies $\tup{E = \varnothing \text{ and } s \neq t}$, then $\chitil\tup{\Pm{G}} = 0$.
    \item[(b)] Otherwise, $\chitil\tup{\Pm{G}} = \tup{-1}^{\# E - \# \NS + 1}$.
\end{enumerate}
\end{theorem}

\begin{theorem}
\label{thm.PF.chitil}
If $E \neq \varnothing$, then:
\begin{enumerate}
    \item[(a)] If $G$ contains a useless arc or a cycle, then $\chitil\tup{\Pf{G}} = 0$.
    \item[(b)] Otherwise, $\chitil\tup{\Pf{G}} = \tup{-1}^{\# \NS}$.
\end{enumerate}
If $E = \varnothing$, then $\chitil\tup{\Pf{G}}$ equals $0$ if $s = t$ and $-1$ otherwise.
\end{theorem}

We shall prove Theorem~\ref{thm.PM.chitil} and Theorem~\ref{thm.PF.chitil} in Section~\ref{sect.euler-proof}, after some preparations.

As a particularly simple-sounding corollary of the above two theorems, we can find the parities of the numbers of faces of $\Pf{G}$ and of $\Pm{G}$ (see the end of Section~\ref{sect.euler-proof} for detailed proofs).

\begin{corollary}
\label{cor.PM.parity} \hfill
\begin{enumerate}
    \item[(a)] If $G$ contains a useless arc or a cycle or satisfies $\tup{E = \varnothing \text{ and } s \neq t}$, then $\# \tup{\Pm{G}}$ is even.
    \item[(b)] Otherwise, $\# \tup{\Pm{G}}$ is odd.
\end{enumerate}
\end{corollary}

\begin{corollary}
\label{cor.PF.parity} \hfill
\begin{enumerate}
    \item[(a)] If $G$ contains a useless arc or a cycle or satisfies $\tup{E = \varnothing \text{ and } s = t}$, then $\# \tup{\Pf{G}}$ is even.
    \item[(b)] Otherwise, $\# \tup{\Pf{G}}$ is odd.
\end{enumerate}
\end{corollary}

\begin{remark}
Theorem~\ref{thm.PM.chitil} and Theorem~\ref{thm.PF.chitil}
show that the reduced Euler characteristics of $\Pf{G}$ and
$\Pm{G}$ always belong to $\set{-1, 0, 1}$.
Not all combinatorially defined simplicial complexes are
this well-behaved.
For example, if we replaced directed graphs by
undirected graphs throughout our definitions,
then the same graph $G$ considered in Example~\ref{exa.PfPm1}
(but without the arrows on the arcs)
would satisfy $\chitil\tup{\Pf{G}} = 3$ and
$\chitil\tup{\Pm{G}} = -3$.
\end{remark}

\subsubsection{\label{subsect.main.fpol}The f-polynomials of $\Pf{G}$ and $\Pm{G}$}

If $\Delta$ is any simplicial complex on a ground set $W$, then we define the \emph{f-polynomial} of $\Delta$ to be the polynomial
\[
f_\Delta\tup{x} := \sum_{I \in \Delta} x^{\# I} \in \ZZ\ive{x} .
\]
An easy comparison of the definitions of $f_\Delta$ and $\chitil\tup{\Delta}$ shows that the value of this polynomial at $-1$ is
\begin{align}
f_\Delta\tup{-1} = - \chitil\tup{\Delta} .
\label{eq.fDelta.-1}
\end{align}

A \textit{quasi-cycle} of $G$ will mean a subset $F$ of $E$ such that $F$ is either
\begin{itemize}
\item the set of arcs of a cycle of $G$, or
\item a $1$-element set consisting of a single useless arc of $G$.
\end{itemize}

Theorem~\ref{thm.PM.chitil} (a) can be restated as claiming that if $G$ has a quasi-cycle, then the polynomial $f_{\Pm{G}}$ is divisible by $1+x$ (since this is tantamount to saying that $f_{\Pm{G}}\tup{-1} = - \chitil\tup{\Pm{G}} = 0$).
Theorem~\ref{thm.PF.chitil} (a) says the same about $\Pf{G}$ instead of $\Pm{G}$.
The next result is a natural strengthening of these corollaries.

\begin{theorem} \label{thm.fpols.divis}
Let $k$ be a nonnegative integer.
Assume that $G$ has at least $k$ disjoint quasi-cycles.
Then the polynomials $f_{\Pm{G}}\tup{x}$ and $f_{\Pf{G}}\tup{x}$ are divisible by $\tup{1+x}^k$.
\end{theorem}

We shall prove this theorem in Section~\ref{sect.fpol-proof1}.

\subsubsection{\label{subsect.main.homtypes}Homotopy types of $\Pf{G}$ and $\Pm{G}$}

The \emph{homotopy type} of a simplicial complex is defined to be its equivalence class with respect to homotopy-equivalence.
Our next goal is to describe the homotopy types of the complexes $\Pf{G}$ and $\Pm{G}$ in terms of the structure of $G$.

First we consider $\Pm{G}$.

\begin{theorem} \label{thm.PM.homotopy}
	The path-missing complex $\Pm{G}$ has one of the following two homotopy types:
	\begin{enumerate}
		\item[(a)] If $G$ contains a useless arc or a cycle or satisfies $\tup{E = \varnothing \text{ and } s \neq t}$, then $\Pm{G}$ is contractible.
		\item[(b)] Otherwise, $\Pm{G}$ is homotopy-equivalent to a sphere of dimension $\#E - \# \NS - 1$.
	\end{enumerate}
\end{theorem}

Here and in the following, we say that a simplicial complex $\tup{W, \Delta}$ is ``homotopy-equivalent to a sphere of dimension $-1$'' if $\Delta = \set{\varnothing}$. (Such a complex is often called an \emph{irrelevant complex}, and its geometric realization is the empty space.)
Moreover, we agree to consider the empty complex $\tup{W, \varnothing}$ to be contractible.

Now we move on to $\Pf{G}$. If $s = t$, then $\Pf{G} = \varnothing$, since any subset of $E$ contains the trivial $s-t$-path in this case. If $E = \varnothing$ and $s \neq t$, then $\Pf{G} = \set{\varnothing}$. In all other cases, the homotopy type of $\Pf{G}$ is determined by the following theorem.

\begin{theorem} \label{thm.PF.homotopy}
    Assume that $E \neq \varnothing$. Then, the path-free complex $\Pf{G}$ has one of the following two homotopy types:
    \begin{enumerate}
        \item[(a)] If $G$ contains a useless arc or a cycle, then $\Pf{G}$ is contractible.
        \item[(b)] Otherwise, $\Pf{G}$ is homotopy-equivalent to a sphere of dimension $\#\NS - 2$.
    \end{enumerate}
\end{theorem}

Both Theorems \ref{thm.PM.homotopy} and \ref{thm.PF.homotopy} will be proved in Section~\ref{sect.homtype-proof}.
The proofs will rely on \emph{discrete Morse theory}, a purely combinatorial approach to the homotopy of complexes, which yields more than just the homotopy types.
See Section~\ref{sect.dmt} for a brief introduction to discrete Morse theory, and \cite{Kozlov20} and \cite{Forman-user} for deeper background.  We end the paper in Section~\ref{sect.further} with several open questions.

\section{Simple properties}

Before we get to the proofs of the above theorems, let us explore some of the most elementary properties of the complexes $\Pf{G}$ and $\Pm{G}$.

If $\Delta$ is a simplicial complex on the ground set $W$, then the \emph{Alexander dual} $\Delta^\vee$ of $\Delta$ is defined by
\[ \Delta^\vee = \set{ F \subseteq W \with W \setminus F \notin \Delta} . \]
This is again a simplicial complex on the ground set $W$.

\begin{lemma} \label{lem.duals}
    The simplicial complexes $\Pf{G}$ and $\Pm{G}$ are Alexander duals of each other.
\end{lemma}
\begin{proof}
    This is immediate from the definitions.
\end{proof}


Given a simplicial complex $\Delta$ with ground set $W$ and an element $w \in W$, we say that $\Delta$ is a \emph{cone with apex $w$} if it has the following property:
\begin{statement}
    For any face $I$ of $\Delta$, we have $I \cup \set{w} \in \Delta$.
\end{statement}
Note that the ``dual'' property (i.e., that $I \setminus \set{w} \in \Delta$ for any $I \in \Delta$) automatically holds for any simplicial complex $\Delta$.
Thus, $\Delta$ is a cone with apex $w$ if and only if adding $w$ to any subset of $W$ does not change whether this subset is a face of $\Delta$.

\begin{lemma}\label{lem:dumbarc}
    Assume that $G$ has a useless arc $e$.
    Then, both $\Pf{G}$ and $\Pm{G}$ are cones with apex $e$.
\end{lemma}
\begin{proof}
    The arc $e$ is useless, and thus is not contained in any $s-t$-path.
    Hence, if $I$ is a subset of $E$, then any $s-t$-path contained in
    $I \cup \set{e}$ must also be contained in $I$.
    Therefore, if a subset $I$ of $E$ contains no $s-t$-path, then the
    subset $I \cup \set{e}$ contains no $s-t$-path either.
    In other words, for any $I \in \Pf{G}$ it also holds that $I \cup\set{e} \in \Pf{G}$.
    Thus, $\Pf{G}$ is a cone with apex $e$.
    A similar argument applies to $\Pm{G}$.
    (Alternatively, using Lemma~\ref{lem.duals}, we can derive this from the following general fact:
    If a simplicial complex $\Delta$ is a cone with apex $w$, then $\Delta^\vee$ is a cone with apex $w$ as well.)
\end{proof}

Let us observe a few more properties of $\Pf{G}$ and $\Pm{G}$, which will not be used in the sequel.
If $\Delta$ is a simplicial complex on ground set $W$, then the subsets of $W$ that do not belong to $\Delta$ are called the \emph{nonfaces} of $\Delta$.
A \emph{minimal nonface} of $\Delta$ means a nonface of $\Delta$ that contains no smaller nonfaces of $\Delta$ as subsets.
The \emph{codimension} of a simplicial complex $\Delta$ on ground set $W$ is defined to be $\# W - \max_{F \in \Delta} \# F$.

\begin{proposition} \label{prop.min-nonfaces} \hfill
    \begin{enumerate}
        \item[(a)] The minimal nonfaces of $\Pf{G}$ are the $s-t$-paths of $G$.
        \item[(b)] The minimal nonfaces of $\Pm{G}$ are the minimal $s-t$-cut-sets of $G$, i.e., the minimal subsets of $E$ that intersect every $s-t$-path nontrivially.    
        \item[(c)] The codimension of $\Pf{G}$ is the size of a smallest $s-t$-cut-set of $G$.
        \item[(d)] The codimension of $\Pm{G}$ is the length of a shortest $s-t$-path.
    \end{enumerate}
\end{proposition}
\begin{proof}
    Parts (a) and (b) are immediate from the definitions.
    For parts (c) and (d), recall that the facets (i.e., maximal faces) of a complex are the complements of the minimal nonfaces of its Alexander dual.
    Hence the codimension of a complex is the minimum size of a minimal nonface of its Alexander dual.
\end{proof}

Proposition~\ref{prop.min-nonfaces} (b) has the consequence that the minimum size of a nonface of $\Pm{G}$ is the minimum size of an $s-t$-cut-set of $G$. This situates the complex $\Pm{G}$ in the context of combinatorial optimization, where minimum cuts have been a long-time topic of research \cite{PicQue}, \cite{ForFul62}.
It is well-known (and easy to prove) that any minimal $s-t$-cut-set of $G$ has the form
\[
\set{e \in E : e \text{ has source in } S \text{ and target in } V \setminus S}
\]
for some subset $S \subseteq V$ that contains $s$ but not $t$. But not every $s-t$-cut-set has this form.

\section{\label{sect.delcont}Deletion-contraction for $\Pf{G}$ and $\Pm{G}$}

In preparation for the proofs, we will investigate ways to recursively disassemble simplicial complexes and graphs.
To wit, we will introduce the deletion and link operations on simplicial complexes (Subsection~\ref{subsect.delcont.dl-lk}), as well as the deletion and contraction operations on graphs (Subsection~\ref{subsect.delcont.contract}).
We will then connect the former to the latter (Subsection~\ref{subsect.delcont.PF-PM-lk-dl}), and finally study the effects of the latter on cycles and useless arcs (Subsection~\ref{subsect.delcont.graph}).

\subsection{\label{subsect.delcont.dl-lk}Deletions and links}

Given any simplicial complex $\Delta$ on the ground set $W$, we make the following definitions:

\begin{itemize}


\item If $w \in W$, then
      the \emph{deletion of $w$ from $\Delta$} is the simplicial complex
      $\dl_\Delta(w)$ on the ground set $W \setminus \set{w}$ defined by
      \[
      \dl_\Delta(w) := \set{F \subseteq W \setminus \set{w} \with F \in \Delta } .
      \]


\item If $w \in W$, then
      the \emph{link of $w$ in $\Delta$} is the simplicial complex
      $\lk_\Delta(w)$ on the ground set $W \setminus \set{w}$ defined by
      \begin{align*}
      \lk_\Delta(w)
      &:= \set{F \subseteq W \setminus \set{w} \with F \cup \set{w} \in \Delta } . 
      \end{align*}
\end{itemize}

As an exercise in working with the definitions, the simplicially innocent reader may prove the following fundamental fact.

\begin{proposition} \label{prop.star.basics}
Let $\Delta$ be a simplicial complex on ground set $W$.
Let $w \in W$. Then:

\begin{enumerate}



\item[(a)] We have $\dl_\Delta(w) = \set{F \in \Delta \with w \notin F}$.

\item[(b)] We have
\begin{align*}
\lk_\Delta(w)
&= \set{F \in \dl_\Delta(w) \with F \cup \set{w} \in \Delta } \\
&= \set{I \setminus \set{w} \with I \in \Delta \text{ and } w \in I }.
\end{align*}

\item[(c)] The map $\set{I \in \Delta \with w \in I} \to \lk_\Delta(w),\ I \mapsto I \setminus \set{w}$ is a bijection.

\end{enumerate}

\end{proposition}

We also invite the reader to verify the following easy fact, which says that
taking the Alexander dual of a simplicial complex turns deletions into links
and vice versa:

\begin{proposition}
\label{prop.dl-lk-dual}
Let $\Delta$ be a simplicial complex on ground set $W$.
Let $w\in W$. Then:

\begin{enumerate}
\item[(a)] We have $\left(  \dl_{\Delta}\left(  w\right)  \right)  ^{\vee
}=\lk_{\Delta^{\vee}}\left(  w\right)  $.

\item[(b)] We have $\left(  \lk_{\Delta}\left(  w\right)  \right)  ^{\vee
}=\dl_{\Delta^{\vee}}\left(  w\right)  $.
\end{enumerate}
\end{proposition}

\begin{proof}
Without loss of generality, let $W=\left\{  1,2,\ldots,n\right\}  $ for some
integer $n$. Then, both claims are particular cases of \cite[Lemma
1.1]{Floystad03}. (Note that \cite{Floystad03} refers to our $\Delta^{\vee}$,
$\dl_{\Delta}\left(  w\right)  $ and $\lk_{\Delta}\left(  w\right)  $ as
$\Delta^{\ast}$, $\Delta_{W\setminus\left\{  w\right\}  }^{\varnothing
,\left\{  w\right\}  }$ and $\Delta_{W\setminus\left\{  w\right\}  }^{\left\{
w\right\}  ,\varnothing}$, respectively. Thus, \cite[Lemma 1.1]{Floystad03}
yields Proposition \ref{prop.dl-lk-dual} (a) by setting $R=W\setminus\left\{
w\right\}  $ and $S=\varnothing$ and $T=\left\{  w\right\}  $, and yields
Proposition \ref{prop.dl-lk-dual} (b) by setting $R=W\setminus\left\{
w\right\}  $ and $S=\left\{  w\right\}  $ and $T=\varnothing$.)
\end{proof}

\subsection{\label{subsect.delcont.contract}Deletion and contraction of arcs}

If $e$ is an arc in $E$, then we shall use the following notions (see Figure~\ref{fig.G/e} for an example):
\begin{itemize}

\item We let $G \backslash e$ denote the graph obtained from $G$ by deleting the arc $e \in E$ (that is, removing $e$ from the arc set of the graph). The vertices $s$ and $t$ remain the distinguished vertices of this new graph.

\item We let $G / e$ denote the graph obtained from $G$ by contracting the arc $e$ (that is, identifying the source of $e$ with the target of $e$, and removing the arc $e$).
If the two endpoints of $e$ are $u$ and $v$, then
the two vertices $u$ and $v$ become one single vertex of $G / e$, which we denote by $u \sim v$ but will still refer to as $u$ or as $v$ (by abuse of notation);
all other vertices of $G$ remain intact in $G / e$;
and each arc of $G$ distinct from $e$ remains an arc of $G / e$ (except that if $u$ or $v$ was its source or target, then that source or target now changes to $u \sim v$).
The distinguished vertices of $G / e$ are $s$ and $t$ again, or rather the vertices resulting from $s$ and $t$ upon the identification.

\end{itemize}

Thus, both graphs $G \backslash e$ and $G / e$ have arc set $E \setminus \set{e}$.
Hence, each walk of $G$ that does not use the arc $e$ is a walk of $G \backslash e$ and a walk of $G / e$.
(But the converse is not true: A walk of $G / e$ can get ``ripped apart'' when we undo the contraction of $e$, and thus fail to be a walk of $G$. Also, a path of $G$ that does not use $e$ might fail to be a path of $G / e$ if it contains both endpoints of $e$.)

\begin{figure}[h!]
\centering
\begin{tabular}[c]{|c|c|c|}\hline
& & \\
\quad
\begin{tikzpicture}[scale=2, >=stealth']
\begin{scope}[every node/.style={circle,fill=blue!20,draw}]
\node(u) at (0, 1) {$u$};
\node(v) at (1, 1) {$v$};
\node(z) at (1, 0) {$z$};
\node(y) at (0, 0) {$y$};
\node(x) at (0.5, 2) {$x$};
\end{scope}
\begin{scope}[every edge/.style={draw=black,very thick}]
\path[->] (u) edge node[above] {$e$} (v) (u) edge (x) (v) edge (x);
\path[->] (u) edge (y) (z) edge (y) (z) edge (v);
\end{scope}
\end{tikzpicture}
\quad\phantom{i}
&
\quad
\begin{tikzpicture}[scale=2, >=stealth']
\begin{scope}[every node/.style={circle,fill=blue!20,draw}]
\node(u) at (0, 1) {$u$};
\node(v) at (1, 1) {$v$};
\node(z) at (1, 0) {$z$};
\node(y) at (0, 0) {$y$};
\node(x) at (0.5, 2) {$x$};
\end{scope}
\begin{scope}[every edge/.style={draw=black,very thick}]
\path[->] (u) edge (x) (v) edge (x);
\path[->] (u) edge (y) (z) edge (y) (z) edge (v);
\end{scope}
\end{tikzpicture}
\quad\phantom{i}
&
\quad
\begin{tikzpicture}[scale=2, >=stealth']
\begin{scope}[every node/.style={circle,fill=blue!20,draw}]
\node(uv) at (0.5, 1) {$u\sim v$};
\node(z) at (1, 0) {$z$};
\node(y) at (0, 0) {$y$};
\node(x) at (0.5, 2) {$x$};
\end{scope}
\begin{scope}[every edge/.style={draw=black,very thick}]
\path[->] (uv) edge[bend left=20] (x) (uv) edge[bend right=20] (x);
\path[->] (uv) edge (y) (z) edge (y) (z) edge (uv);
\end{scope}
\end{tikzpicture}
\quad\phantom{i}
\\
& & \\\hline
$G$ & $G\backslash e$ & $G/e$ \\\hline
\end{tabular}
\caption{A graph $G$ with an arc $e$, and the two derived graphs $G\backslash e$ and $G/e$.}
\label{fig.G/e}
\end{figure}

We shall next prove several graph-theoretical properties of $G \backslash e$ and $G / e$ that will allow us to argue recursively in the ``deletion-contraction'' paradigm.

\begin{lemma} \label{lem.G/e.stpath}
    Let $e \in E$ be an arc whose source is $s$.
    Let $I$ be a subset of $E$ that contains $e$.
    Then, $I$ contains an $s-t$-path of $G$ if and only if $I \setminus \set{e}$ contains an $s-t$-path of $G / e$.
\end{lemma}

We notice that Lemma~\ref{lem.G/e.stpath} would be false if we did not require that the source of $e$ is $s$: a simple counterexample is the graph $s \overset{e}{\longleftarrow} t$ with $I = E$.

\begin{proof}[Proof of Lemma~\ref{lem.G/e.stpath}.]

$\Longrightarrow:$
Assume that $I$ contains an $s-t$-path $p$ of $G$.
We must prove that $I \setminus \set{e}$ contains an $s-t$-path of $G / e$.
It suffices to show that $I \setminus \set{e}$ contains an $s-t$-walk of $G / e$ (since we can then obtain an $s-t$-path from such a walk by removing cycles).
In other words, we must find an $s-t$-walk of $G / e$ contained in $I \setminus \set{e}$.
If $p$ does not contain $e$, then $p$ is already such an $s-t$-walk (since all arcs of $G$ except for $e$ are arcs of $G / e$), and so we are done.
If, however, $p$ does contain $e$, then we can obtain an $s-t$-walk of $G / e$
by removing $e$ from $p$,
and this $s-t$-walk is clearly contained in $I \setminus \set{e}$.
This proves the ``$\Longrightarrow$'' direction of Lemma~\ref{lem.G/e.stpath}.

$\Longleftarrow:$
Assume that $I \setminus \set{e}$ contains an $s-t$-path $q$ of $G / e$.
We must prove that $I$ contains an $s-t$-path of $G$.

Let $s'$ be the target of $e$. Note that the contraction of $e$ identifies the two vertices $s$ and $s'$.
Thus, the arcs of $G / e$ starting at $s$ are (1) the arcs of $G$ starting at $s$ except for the arc $e$, and (2) the arcs of $G$ starting at $s'$ (if $s' \neq s$).

Now, let us regard the arcs on the $s-t$-path $q$ as arcs of the original graph $G$.
In this way, it is no longer guaranteed that they still form an $s-t$-path,
since $s$ and $s'$ are not necessarily identical in $G$;
however, they still form either an $s-t$-path or an $s'-t$-path (because $s$ and $s'$ are the only two vertices that get identified in $G / e$, and the $s-t$-path $q$ never revisits its starting vertex $s$ after leaving it).
If they form an $s-t$-path, then we conclude that $I$ contains an $s-t$-path of $G$ (namely, the path that they form), and thus we are done.
If they don't, then they form an $s'-t$-path $q'$ that does not contain the vertex $s$, and we can again conclude that $I$ contains an $s-t$-path of $G$ (namely, the path obtained by prepending $e$ to $q'$).
In either case, we have shown that $I$ contains an $s-t$-path of $G$.
This proves the ``$\Longleftarrow$'' direction of Lemma~\ref{lem.G/e.stpath}.
\end{proof}

\begin{lemma} \label{lem.G/e.stpath-comp}
    Let $e \in E$ be an arc whose source is $s$.
    Let $F$ be a subset of $E \setminus \set{e}$.
    Then, $E \setminus F$ contains an $s-t$-path of $G$ if and only if $\tup{E \setminus \set{e}} \setminus F$ contains an $s-t$-path of $G / e$.
\end{lemma}

\begin{proof}
This follows from Lemma~\ref{lem.G/e.stpath} (applied to $I = E \setminus F$), because $e \in E \setminus F$ and $\tup{E \setminus F} \setminus \set{e} = \tup{E \setminus \set{e}} \setminus F$.
\end{proof}

\subsection{\label{subsect.delcont.PF-PM-lk-dl}Links and deletions of $\Pm{G}$ and $\Pf{G}$}

The link and the deletion of an arc $e$ in $\Pm{G}$ and $\Pf{G}$ can be described in terms of smaller graphs at least when the source of $e$ is $s$.

\begin{lemma}\label{lem:delcon}
    Let $e \in E$. Then:
    \begin{enumerate}
        \item[(a)] We always have $\lk_{\Pm{G}}(e) = \Pm{G \backslash e}$.
        \item[(b)] We have $\dl_{\Pm{G}}(e) = \Pm{G / e}$ if the source of $e$ is $s$.
    \end{enumerate}
\end{lemma}

\begin{proof}
    (a) An $s-t$-path of $G \backslash e$ is the same thing as an $s-t$-path of $G$ that does not contain $e$.
    Hence, for any subset $I$ of $E \setminus \set{e}$, we have the equivalence
    \begin{equation}
    \label{pf.lem:delcon.1}
                        \tup{I \text{ contains an $s-t$-path of $G$}} 
    \Longleftrightarrow  \tup{I \text{ contains an $s-t$-path of $G \backslash e$}}
    \end{equation}
    (since an $s-t$-path of $G$ contained in $I$ cannot use the arc $e$ anyway, and thus must be an $s-t$-path of $G \backslash e$).
    
    From the definition of $\lk_{\Pm{G}}(e)$, we have
    \begin{align*}
    \lk_{\Pm{G}}(e)
    &= \set{F \subseteq E \setminus \set{e} \with F \cup \set{e} \in \Pm{G}} \\
    &= \set{F \subseteq E \setminus \set{e} \with E \setminus \tup{F \cup \set{e}} \text{ contains an $s-t$-path of $G$}} \\
    & \qquad\qquad \tup{\text{by the definition of $\Pm{G}$}} \\
    &= \set{F \subseteq E \setminus \set{e} \with \tup{E \setminus \set{e}} \setminus F \text{ contains an $s-t$-path of $G$}} \\
    & \qquad\qquad \tup{\text{since $E \setminus \tup{F \cup \set{e}} = \tup{E \setminus \set{e}} \setminus F$ for any set $F$}} \\
    &= \set{F \subseteq E \setminus \set{e} \with \tup{E \setminus \set{e}} \setminus F \text{ contains an $s-t$-path of $G \backslash e$}} \\
    & \qquad\qquad \tup{\text{by the equivalence \eqref{pf.lem:delcon.1}, applied to $I = \tup{E \setminus \set{e}} \setminus F$}} \\
    &= \Pm{G \backslash e}
     \qquad \tup{\text{by the definition of $\Pm{G \backslash e}$}} .
    \end{align*}
    This proves Lemma~\ref{lem:delcon} (a).
    
    (b) Assume that the source of $e$ is $s$.
    From the definition of $\dl_{\Pm{G}}(e)$, we have
    \begin{align*}
    \dl_{\Pm{G}}(e)
    &= \set{F \subseteq E \setminus \set{e} \with F \in \Pm{G}} \\
    &= \set{F \subseteq E \setminus \set{e} \with E \setminus F \text{ contains an $s-t$-path of $G$}} \\
    & \qquad\qquad \tup{\text{by the definition of $\Pm{G}$}} \\
    &= \set{F \subseteq E \setminus \set{e} \with \tup{E \setminus \set{e}} \setminus F \text{ contains an $s-t$-path of $G / e$}} \\
    & \qquad\qquad \tup{\text{by Lemma~\ref{lem.G/e.stpath-comp}}} \\
    &= \Pm{G / e}
     \qquad \tup{\text{by the definition of $\Pm{G / e}$}} .
    \end{align*}
    This proves Lemma~\ref{lem:delcon} (b).
\end{proof}

The following is the $\Pf{G}$-analogue of Lemma \ref{lem:delcon}.

\begin{lemma}\label{lem:delcon_pf}
	Let $e \in E$. Then:
	\begin{enumerate}
		\item[(a)] We always have $\dl_{\Pf{G}}(e) = \Pf{G \backslash e}$.
		\item[(b)] We have $\lk_{\Pf{G}}(e) = \Pf{G / e}$ if the source of $e$ is $s$.
	\end{enumerate}
\end{lemma}

\begin{proof}[First proof of Lemma~\ref{lem:delcon_pf}.]
    (a) For any subset $I$ of $E \setminus \set{e}$, we have the equivalence
    \begin{equation}
    \label{pf.lem:delcon_pf.1}
                        \tup{I \text{ contains no $s-t$-path of $G$}} 
    \Longleftrightarrow  \tup{I \text{ contains no $s-t$-path of $G \backslash e$}} .
    \end{equation}
    (Indeed, this follows from the equivalence \eqref{pf.lem:delcon.1} by negating both sides.)
    
    From the definition of $\dl_{\Pf{G}}(e)$, we have
    \begin{align*}
    \dl_{\Pf{G}}(e)
    &= \set{F \subseteq E \setminus \set{e} \with F \in \Pf{G}} \\
    &= \set{F \subseteq E \setminus \set{e} \with F \text{ contains no $s-t$-path of $G$}} \\
    & \qquad\qquad \tup{\text{by the definition of $\Pf{G}$}} \\
    &= \set{F \subseteq E \setminus \set{e} \with F \text{ contains no $s-t$-path of $G \backslash e$}} \\
    & \qquad\qquad \tup{\text{by the equivalence \eqref{pf.lem:delcon_pf.1}, applied to $I=F$}} \\
    &= \Pf{G \backslash e}
     \qquad \tup{\text{by the definition of $\Pf{G \backslash e}$}} .
    \end{align*}
    This proves Lemma~\ref{lem:delcon_pf} (a).
    
    (b) Assume that the source of $e$ is $s$.
    Let $F$ be a subset of $E \setminus \set{e}$. Hence,
    $F \cup \set{e}$ is a subset of $E$ that contains $e$.
    Thus, Lemma~\ref{lem.G/e.stpath} (applied to $I = F \cup \set{e}$)
    shows that $F \cup \set{e}$ contains an $s-t$-path of $G$ if and
    only if $\tup{F \cup \set{e}} \setminus \set{e}$ contains an $s-t$-path of $G / e$.
    Thus, we have the equivalence
    \begin{align*}
    & \tup{F \cup \set{e} \text{ contains an $s-t$-path of $G$}} 
    \\
    &\Longleftrightarrow  \tup{\tup{F \cup \set{e}} \setminus \set{e} \text{ contains an $s-t$-path of $G / e$}} \\
    &\Longleftrightarrow  \tup{F \text{ contains an $s-t$-path of $G / e$}}
    \end{align*}
    (since $\tup{F \cup \set{e}} \setminus \set{e} = F$).
    Negating both sides of this, we obtain the equivalence
    \begin{align}
    &                   \tup{F \cup \set{e} \text{ contains no $s-t$-path of $G$}} \nonumber\\
    &    \Longleftrightarrow  \tup{F \text{ contains no $s-t$-path of $G / e$}} .
    \label{pf.lem:delcon_pf.2}
    \end{align}

    Forget that we fixed $F$.
    We thus have proved the equivalence \eqref{pf.lem:delcon_pf.2}
    for any subset $F$ of $E \setminus \set{e}$.

    From the definition of $\lk_{\Pf{G}}(e)$, we have
    \begin{align*}
    \lk_{\Pf{G}}(e)
    &= \set{F \subseteq E \setminus \set{e} \with F \cup \set{e} \in \Pf{G}} \\
    &= \set{F \subseteq E \setminus \set{e} \with F \cup \set{e} \text{ contains no $s-t$-path of $G$}} \\
    & \qquad\qquad \tup{\text{by the definition of $\Pf{G}$}} \\
    &= \set{F \subseteq E \setminus \set{e} \with F \text{ contains no $s-t$-path of $G / e$}} \\
    & \qquad\qquad \tup{\text{by the equivalence \eqref{pf.lem:delcon_pf.2}}} \\
    &= \Pf{G / e}
     \qquad \tup{\text{by the definition of $\Pf{G / e}$}} .
    \end{align*}
    This proves Lemma~\ref{lem:delcon_pf} (b).
\end{proof}

\begin{proof}[Second proof of Lemma~\ref{lem:delcon_pf}.]
Lemma~\ref{lem:delcon_pf} can also
be derived from Lemma \ref{lem:delcon} using Lemma \ref{lem.duals}
and Proposition \ref{prop.dl-lk-dual}. For example, let us prove part (a).
Lemma \ref{lem.duals} yields $\Pf{G}  =\left(
\Pm{G}  \right)  ^{\vee}$. Thus,
\begin{equation}
\dl_{\Pf{G}  }(e) = \dl_{\left(  \Pm{G} \right)  ^{\vee}}(e)
= \left(  \lk_{\Pm{G} }\left(  e\right)  \right)  ^{\vee},
\label{pf.lem:delcon_pf.2nd.1}
\end{equation}
since Proposition \ref{prop.dl-lk-dual} (b) yields
$\left(  \lk_{\Pm{G}  }\left(  e\right)  \right)  ^{\vee}
= \dl_{\left( \Pm{G}  \right)  ^{\vee}}(e)$.
But Lemma~\ref{lem:delcon} (a) yields $\lk_{\Pm{G} }(e)
= \Pm{G\backslash e}$.
Hence, \eqref{pf.lem:delcon_pf.2nd.1} rewrites as
\[
\dl_{\Pf{G}  }(e)
= \left(  \Pm{G\backslash e} \right)  ^{\vee}
= \Pf{G\backslash e}
\]
(by Lemma \ref{lem.duals}, applied to $G\backslash e$ instead of $G$). Thus,
Lemma \ref{lem:delcon_pf} (a) is proved. Similarly, we obtain part (b) from
Lemma \ref{lem:delcon} (b) using Proposition \ref{prop.dl-lk-dual} (a).
\end{proof}

\subsection{\label{subsect.delcont.graph}Graph-theoretical properties of deletion and contraction}

The next lemmas discuss how deletion and contraction of a single arc affect various properties of our graph, such as the existence of useless arcs and cycles and the number of nonsinks.

\begin{lemma}\label{lem:tgt-s-useless}
    Any arc of $G$ that has target $s$ is useless.
\end{lemma}

\begin{proof}
An arc with target $s$ cannot appear in an $s-t$-path.
Thus, it must be useless.
\end{proof}

\begin{lemma}\label{lem:induction_step_1}
	Let $e \in E$ be an arc whose source is $s$.
    Let $s'$ be the target of $e$.
    Assume that $e$ is not the only arc with target $s'$.
    Then, $G / e$ has a useless arc.
\end{lemma}

\begin{proof}
We assumed that $e$ is not the only arc with target $s'$.
Thus, there is another arc $f \neq e$ with target $s'$.
In the graph $G / e$, this arc $f$ has target $s$ (since its target $s'$ is identified with $s$ in $G / e$) and thus is useless (by Lemma~\ref{lem:tgt-s-useless}, applied to $G / e$ instead of $G$).
This proves Lemma~\ref{lem:induction_step_1}.
\end{proof}

\begin{lemma}\label{lem:induction_step_2a}
	Let $e \in E$ be an arc whose source is $s$.
    Assume that $e$ is not useless.
    Let $s'$ be the target of $e$.
	Assume further that $s' \neq t$, and that $e$ is the only arc with target $s'$.
    Then, $G \backslash e$ has a useless arc.
\end{lemma}

\begin{proof}
The arc $e$ is not useless; thus, there exists an $s-t$-path that contains $e$.
This $s-t$-path cannot end at $s'$ (since $s' \neq t$), and thus its last arc cannot be $e$.
Hence, there exists an arc that follows $e$ in this $s-t$-path.
Let $e'$ be this arc. Clearly, $e'$ has source $s'$.


We have assumed that $e$ is the only arc with target $s'$.
Hence, there is no arc of $G \backslash e$ that has target $s'$.

We shall show that $e'$ is useless in $G \backslash e$.
Indeed, assume the contrary. Thus, there exists some $s-t$-path of $G \backslash e$ that uses this arc $e'$.
This path must pass through the vertex $s'$ (since the arc $e'$ has source $s'$),
and thus must start at $s'$
(since there is no arc of $G \backslash e$ that has target $s'$).
Since it is an $s-t$-path, this means that $s' = s$.
Hence, the arc $e$ has target $s$ (since it has target $s'$),
and thus is useless in $G$ (by Lemma~\ref{lem:tgt-s-useless}).
But this contradicts the fact that $e$ is not useless in $G$.
This contradiction shows that our assumption was false.
Hence, $G \backslash e$ has a useless arc (namely, $e'$).
This proves Lemma~\ref{lem:induction_step_2a}.
\end{proof}

\begin{lemma}\label{lem:induction_step_2b}
	Let $e \in E$ be an arc whose source is $s$.
    Assume that $e$ is not useless.
    Let $s'$ be the target of $e$.
	Assume further that $\# E > 1$, and that $e$ is the only arc with target $s'$.
    Then, $G \backslash e$ has a useless arc.
\end{lemma}

\begin{proof}
If $s' \neq t$, then this follows from
Lemma~\ref{lem:induction_step_2a}.
Thus, for the rest of this proof, we WLOG assume that $s' = t$.

The arc $e$ is not useless; thus, it belongs to an $s-t$-path.
This $s-t$-path thus contains at least one arc.
Therefore, $s \neq t$.

But recall that $e$ is the only arc with target $s'$.
In other words, $e$ is the only arc with target $t$
(since $s' = t$).
Hence, the graph $G \backslash e$ has no arc with target $t$.
Therefore, the graph $G \backslash e$ has no $s-t$-path
(because $s \neq t$, so any $s-t$-path would have to end with an arc with target $t$).
Therefore, any arc of $G \backslash e$ is useless.

We have $\# E > 1$; hence, the set $E \setminus \set{e}$ is nonempty.
In other words, $G \backslash e$ has an arc (since
$E \setminus \set{e}$ is the arc set of $G \backslash e$).
Hence, $G \backslash e$ has a useless arc
(since any arc of $G \backslash e$ is useless).
Lemma~\ref{lem:induction_step_2b} is thus proved.
\end{proof}

\begin{lemma} \label{lem.G/e.basics}
    Let $e \in E$ be an arc whose source is $s$.
    Let $s'$ be the target of $e$.
    Assume that $G$ has no useless arcs.
    Then:
    
    \begin{enumerate}
    
    
    \item[(a)] The graph $G$ has no cycle containing $s$.
    
    \item[(b)] We have $s' \neq s$.
    

    \item[(c)] If $s' \neq t$, then there exists an arc
               of $G \backslash e$ whose source is $s'$.
    
    \item[(d)] If $s' \neq t$, then the graph $G / e$
               has exactly one fewer nonsink than $G$.
        

	
	\item[(e)] If $e$ is the only arc of $G$ with target $s'$,
               and if $\# E > 1$, then $s' \neq t$.
    
    \end{enumerate}
\end{lemma}

\begin{proof}

(a) Any cycle containing $s$ would contain an arc with target $s$.  By Lemma~\ref{lem:tgt-s-useless}, such an arc would be useless.
This contradicts the assumption that $G$ has no useless arcs.  Thus, $G$ has no cycle containing $s$.  This proves Lemma~\ref{lem.G/e.basics} (a).

(b) If we had $s' = s$, then the arc $e$ would be a self-loop and therefore useless (since every self-loop is useless). But this is impossible (since $G$ has no useless arcs).
Hence, Lemma~\ref{lem.G/e.basics} (b) is proven.  


(c) Assume that $s' \neq t$.
The arc $e$ of $G$ is not useless (since $G$ has no useless arcs),
and thus is contained in an $s-t$-path of $G$.
This path cannot end with $e$ (since $s' \neq t$);
thus, there exists an arc that follows $e$ in this path.
Let $f$ be this arc.
Then, the source of $f$ is $s'$,
which is distinct from $s$ (since Lemma~\ref{lem.G/e.basics} (b)
yields $s' \neq s$).
Hence, $f \neq e$ (since the source of $e$ is $s$).
Thus, $f$ is also an arc of $G \backslash e$.
Hence, there exists an arc
of $G \backslash e$ whose source is $s'$ (namely, $f$).
This proves Lemma~\ref{lem.G/e.basics} (c).

(d) Assume that $s' \neq t$.
Thus, Lemma~\ref{lem.G/e.basics} (c) shows that
there exists an arc of $G \backslash e$ whose source is $s'$.
Let $f$ be this arc.
Then, $f$ is an arc of $G \backslash e$, thus also an
arc of $G / e$ and of $G$.
Hence, $G$ has an arc with source $s'$ (namely, $f$).
In other words, $s'$ is a nonsink of $G$.
Also, $s$ is a nonsink of $G$ (since the arc $e$ has source $s$).
The two nonsinks $s$ and $s'$ are distinct (by
Lemma~\ref{lem.G/e.basics} (b)).
In the graph $G / e$, these two nonsinks become identical, but still
remain a nonsink of $G / e$ (since the arc $f$ has source $s'$ and
is an arc of $G / e$).
For any vertex $v \in V \setminus \set{s,s'}$, it is clear that
$v$ is a nonsink of $G / e$ if and only if $v$ is a nonsink of $G$
(since the arcs of $G / e$ with source $v$ are exactly the arcs of $G$
with source $v$).
Thus, when passing from $G$ to $G / e$, we only lose one nonsink (since the
two nonsinks $s$ and $s'$ become identical).
This proves Lemma~\ref{lem.G/e.basics} (d).

(e) Assume that $e$ is the only arc of $G$ with target $s'$,
and that $\# E > 1$. We must prove that $s' \neq t$.

Assume the contrary. Thus, $s' = t$.
Hence, $e$ is the only arc of $G$ with target $t$
(since $e$ is the only arc of $G$ with target $s'$).

Since $\# E > 1$, there exists at least one arc $f \in E$
distinct from $e$. This arc $f$ cannot be useless (since $G$
has no useless arcs).
In other words, $f$ is contained in an $s-t$-path $p$ of $G$.
Consider this path $p$.

This path $p$ has at least one arc (since it contains $f$).
Thus, it has a last arc. This last arc of $p$ must have
target $t$ (since $p$ is an $s-t$-path), and thus must be $e$
(since $e$ is the only arc of $G$ with target $t$).
Therefore, the second-to-last vertex of $p$ is the source of $e$.
In other words, the second-to-last vertex of $p$ is $s$
(since the source of $e$ is $s$).
Of course, the first vertex of $p$ is $s$, too (since $p$ is an $s-t$-path).

We know that both $e$ and $f$ are arcs of $p$.
Since $f$ is distinct from $e$, this entails that $p$ has at
least two arcs, thus at least three vertices.
Hence, the first vertex and the second-to-last vertex of $p$ are distinct
(since the vertices of a path are distinct).
But this contradicts the fact that both of these vertices are $s$.
This contradiction shows that our assumption was false.
Lemma~\ref{lem.G/e.basics} (e) is proved.
\end{proof}

\begin{lemma} \label{lem.G/e.B}
    Let $e \in E$ be an arc whose source is $s$.
    Assume that $G$ has a cycle but no useless arcs.
    Then, both graphs $G \backslash e$ and $G / e$ have cycles.
\end{lemma}

\begin{proof}
We have assumed that $G$ has a cycle.
This cycle cannot contain $s$ (by Lemma~\ref{lem.G/e.basics} (a)), and thus cannot contain the arc $e$.
Hence, this cycle is a cycle of $G \backslash e$ and also a nontrivial closed walk of $G / e$.
Thus, both graphs $G \backslash e$ and $G / e$ have nontrivial closed walks,
and therefore have cycles.
This proves Lemma~\ref{lem.G/e.B}.
\end{proof}

\begin{lemma} \label{lem.G-e.C}
    Let $e \in E$ be an arc whose source is $s$.
	Let $s'$ be the target of $e$.
	Assume that the graph $G \backslash e$ has no cycles,
	but has an $s-s'$-path.
	Let $f$ be a useless arc of $G \backslash e$.
	Then, $f$ is also a useless arc of $G$.
\end{lemma}

\begin{proof}
Assume the contrary. Thus, $f$ is not useless as an arc of $G$.
In other words, $G$ has an $s-t$-path $p$ that uses this arc $f$.
Consider this path $p$.

But $f$ is useless as an arc of $G \backslash e$.
Hence, $f$ belongs to no $s-t$-path of $G \backslash e$.
Note that $f$ is distinct from $e$, since $f$ is an arc of
$G \backslash e$.

The $s-t$-path $p$ uses the arc $f$, and thus also uses the
arc $e$ (since otherwise, $p$ would be an
$s-t$-path of $G \backslash e$, which would contradict the fact
that $f$ belongs to no $s-t$-path of $G \backslash e$).

We assumed that $G \backslash e$ has an $s-s'$-path.
Let us denote this path by $q$.
Replacing the arc $e$ by the path $q$ in the $s-t$-path $p$,
we obtain an $s-t$-walk of $G \backslash e$.
This walk must be a path (since otherwise, it would contain a cycle,
but $G \backslash e$ has no cycles), and thus is an $s-t$-path;
furthermore, it contains $f$ (since $p$ contains $f$,
and $f$ survives the replacement of $e$ by $q$ because $f$ is
distinct from $e$).
Thus, it is an $s-t$-path of $G \backslash e$ that contains $f$.
This contradicts the fact that
$f$ belongs to no $s-t$-path of $G \backslash e$.
This contradiction completes our proof.
\end{proof}

\begin{lemma} \label{lem.G/e.C}
    Let $e \in E$ be an arc whose source is $s$.
    Assume that $G$ has no cycles and no useless arcs.
    Assume further that $G \backslash e$ has a useless arc.
    Then, the graph $G / e$ has no cycles and no useless arcs.
\end{lemma}

\begin{proof}
Let $s'$ be the target of $e$.
Then, $s' \neq s$ (by Lemma~\ref{lem.G/e.basics} (b)).

The graph $G \backslash e$ is a subgraph of $G$, and thus has
no cycles (since $G$ has no cycles).

We have assumed that $G \backslash e$ has a useless arc.
Fix such an arc, and call it $f$.

First, we shall show that $G / e$ has no cycles.

Indeed, assume the contrary. Thus, $G / e$ has a cycle $c$.
Consider each arc of $c$ as an arc of $G$.
Then, $c$ is either a cycle of $G$, or an $s'-s$-path of $G$,
or an $s-s'$-path of $G$
(because the only way in which the cycle could break when
lifted from $G / e$ to $G$ is if the cycle passes through
the identified vertex $s \sim s'$).
The first of these three cases is impossible (since $G$ has no cycles),
and so is the second (by Lemma~\ref{lem:tgt-s-useless}, since $G$ has no useless arcs).
Thus, the third case must hold.
In other words, $c$ is an $s-s'$-path of $G$.
Since $c$ does not contain $e$ (because $c$ was originally a cycle of $G / e$),
this shows that $c$ is an $s-s'$-path of $G \backslash e$.
Thus, $G \backslash e$ has an $s-s'$-path.
Hence, Lemma~\ref{lem.G-e.C} shows that
$f$ is a useless arc of $G$.
But this contradicts the fact that $G$ has no useless arcs.
This contradiction shows that our assumption was false.
Hence, we have proven that $G / e$ has no cycles.

Next, we shall show that $G / e$ has no useless arcs.


Let $h$ be any arc of $G / e$; we will show that $h$ is not useless.  
Consider $h$ as an arc of $G$.
Then $h$ is distinct from $e$, and cannot be a useless arc of $G$ (since $G$ has no useless arcs).
Hence, there exists an $s-t$-path of $G$ containing $h$.
Let $r$ be such an $s-t$-path.
By removing the arc $e$ from $r$ (if it is contained in $r$),
we obtain an $s-t$-walk of $G / e$.
This latter $s-t$-walk must actually be an $s-t$-path (since $G / e$
has no cycles), and furthermore contains $h$ (because $r$ contains
$h$, and because $h$ is distinct from $e$ and thus could not
have been removed).
Hence, there exists an $s-t$-path of $G / e$ that contains $h$.
Thus $h$ is not useless.
Since $h$ was an arbitrary arc of $G / e$, we have proven that $G / e$ has no useless arcs.

Altogether, we now know that $G / e$ has no cycles and no useless arcs.
This proves Lemma~\ref{lem.G/e.C}.
\end{proof}

\begin{lemma} \label{lem.G/e.Dnew}
    Let $e \in E$ be an arc whose source is $s$.
	Assume that $\# E > 1$.
    Assume further that $G \backslash e$ has no useless arcs.
    
    
    Then, the graph $G \backslash e$ has the same nonsinks as the graph $G$.
    
    
\end{lemma}

\begin{proof}
The set of arcs of the graph $G \backslash e$ is
$E \setminus \set{e}$, and thus is nonempty
(since $\# E > 1$).
Thus, the graph $G \backslash e$ has at least one
arc $f$. This arc $f$ is not useless (since
$G \backslash e$ has no useless arcs), and thus
there exists an $s-t$-path of $G \backslash e$
that contains $f$. This $s-t$-path has at least
one arc (since it contains $f$), and thus must
have a first arc. This first arc has source $s$.

Hence, the graph $G \backslash e$ has an arc
with source $s$ (namely, this first arc).
In other words, $s$ is a nonsink of
$G \backslash e$. Consequently, $s$ is a nonsink
of $G$ as well.

For any vertex $v \in V \setminus \set{s}$, it is clear that
$v$ is a nonsink of $G \backslash e$ if and only if $v$ is a nonsink of $G$
(since $G \backslash e$ differs from $G$ only in the arc $e$,
whose source is $s$).
This also holds for $v = s$ (since $s$ is a nonsink of $G$ and of $G \backslash e$ both).
Thus, it holds for all $v \in V$.
In other words, the graph $G \backslash e$ has the same nonsinks as the graph $G$.
This proves Lemma~\ref{lem.G/e.Dnew}. 
\end{proof}

\section{\label{sect.fpol-proof1}Proof of the f-polynomial}

Let us first take aim at Theorem~\ref{thm.fpols.divis}.

We begin with a recurrence for f-polynomials in terms of
deletions and links
(see Subsection~\ref{subsect.main.fpol} and
Subsection~\ref{subsect.delcont.dl-lk} for the relevant definitions).

\begin{lemma} \label{lem.fpol.rec}
Let $\Delta$ be a simplicial complex on the ground set $W$.
Let $w \in W$. Then,
\[
f_{\Delta}\tup{x} = f_{\dl_\Delta(w)}\tup{x} + x f_{\lk_\Delta(w)}\tup{x} .
\]
\end{lemma}

\begin{proof}
This is a folklore result
(see \cite[the equality $\mathcal{F}_\Delta
= x \mathcal{F}_{\tup{\Delta:x}} + \mathcal{F}_{\tup{\Delta,x}}$
in the proof of
Lemma 2.4]{MariettiTesta08} for a generalization
in which $w$ is replaced by a monomial\footnote{This
generalization is stated in the language of monomial
ideals. To translate between the two languages,
identify each simplicial complex $\Delta$ on the ground
set $W$ with the monomial ideal $I$ of the polynomial ring
$\mathbf{k}\left[x_w \mid w \in W\right]$ that is
spanned by all monomials that do \textbf{not} have the
form $\prod_{w \in I} x_w$ for $I \subseteq \Delta$.}).
But for the reader's convenience,
let us give an easy elementary proof:

The definition of $f_{\Delta}\tup{x}$ yields
\begin{align}
f_{\Delta}\tup{x}
= \sum_{I \in \Delta} x^{\# I}
= \sum_{\substack{I \in \Delta; \\ w \notin I}} x^{\# I} + \sum_{\substack{I \in \Delta; \\ w \in I}} x^{\# I} .
\label{pf.lem.fpol.rec.1}
\end{align}
But the definition of $f_{\dl_\Delta(w)}\tup{x}$ shows that
\begin{align}
f_{\dl_\Delta(w)}\tup{x}
= \sum_{I \in \dl_\Delta(w)} x^{\# I}
= \sum_{\substack{I \in \Delta; \\ w \notin I}} x^{\# I}
\label{pf.lem.fpol.rec.2}
\end{align}
(since the $I \in \dl_\Delta(w)$ are precisely the $I \in \Delta$ satisfying $w \notin I$).
The definition of $f_{\lk_\Delta(w)}\tup{x}$ shows that
\begin{align*}
f_{\lk_\Delta(w)}\tup{x}
&= \sum_{I \in \lk_\Delta(w)} x^{\# I}
= \sum_{\substack{I \in \Delta; \\ w \in I}} x^{\# \tup{I \setminus \set{w}}} 
\end{align*}
(since the map $\set{I \in \Delta \with w \in I} \to \lk_\Delta(w),\ I \mapsto I \setminus \set{w}$ is a bijection).
Multiplying the latter equality by $x$, we obtain
\begin{align}
x f_{\lk_\Delta(w)}\tup{x}
&= \sum_{\substack{I \in \Delta; \\ w \in I}} x^{1 + \# \tup{I \setminus \set{w}}} 
= \sum_{\substack{I \in \Delta; \\ w \in I}} x^{ \# I }
\label{pf.lem.fpol.rec.3}
\end{align}
(since $1 + \# \tup{I \setminus \set{w}} = \# I$ for any $I \in \Delta$ satisfying $w \in I$).
Adding  \eqref{pf.lem.fpol.rec.2} and \eqref{pf.lem.fpol.rec.3} together,
and comparing the result with \eqref{pf.lem.fpol.rec.1},
we obtain Lemma~\ref{lem.fpol.rec}.
\end{proof}

As a consequence of Lemma~\ref{lem.fpol.rec}, we get the following lemma.

\begin{lemma} \label{lem.fpol.cone}
Let $\Delta$ be a simplicial complex on the ground set $W$. Let $w \in W$.
Assume that $\Delta$ is a cone with apex $w$. Then:

\begin{enumerate}

\item[(a)] We have $\dl_\Delta(w) = \lk_\Delta(w)$.

\item[(b)] We have $f_{\Delta}\tup{x} = \tup{1+x} f_{\lk_\Delta(w)}\tup{x}$.

\end{enumerate}
\end{lemma}

\begin{proof}
(a) Let $I \in \dl_\Delta(w)$. Then, $I \in \Delta$ and $w \notin I$
(by the definition of $\dl_\Delta(w)$).
But $\Delta$ is a cone with apex $w$; hence, from $I \in \Delta$,
we obtain $I \cup \set{w} \in \Delta$ (by the definition of a cone).
Thus, $I$ is an $F \in \dl_\Delta(w)$ satisfying $F \cup \set{w} \in \Delta$.
In other words, $I \in \lk_\Delta(w)$ (by the definition of $\lk_\Delta(w)$).

Forget that we fixed $I$.
We thus have shown that $I \in \lk_\Delta(w)$ for each $I \in \dl_\Delta(w)$.
In other words, $\dl_\Delta(w) \subseteq \lk_\Delta(w)$.
Combining this with the obvious inclusion $\lk_\Delta(w) \subseteq \dl_\Delta(w)$,
we obtain $\dl_\Delta(w) = \lk_\Delta(w)$.
This proves Lemma~\ref{lem.fpol.cone} (a).

(b) Lemma~\ref{lem.fpol.rec} yields
\begin{align*}
f_{\Delta}\tup{x}
&= f_{\dl_\Delta(w)}\tup{x} + x f_{\lk_\Delta(w)}\tup{x} \\
&= f_{\lk_\Delta(w)}\tup{x} + x f_{\lk_\Delta(w)}\tup{x}
\qquad \tup{\text{since part (a) yields } \dl_\Delta(w) = \lk_\Delta(w) } \\
&= \tup{1+x} f_{\lk_\Delta(w)}\tup{x} . \qedhere
\end{align*}
\end{proof}

\begin{verlong}

Lemma~\ref{lem.chitil.dual}, too, can be generalized to f-polynomials,
resulting in the following lemma.

\begin{lemma} \label{lem.fpol.dual}
Let $\Delta$ be a simplicial complex on a nonempty ground set $W$. Then,
\[
f_{\Delta^\vee}\tup{x} = \tup{1+x}^{\# W} - x^{\# W} f_{\Delta}\tup{1/x} .
\]
\end{lemma}

\begin{proof}
We have $\sum\limits_{I \in 2^W} x^{\# I}
= \sum\limits_{k=0}^{\# W} \dbinom{\# W}{k} x^k = \tup{1+x}^{\# W}$
(by the binomial theorem).
Substituting $1/x$ for $x$ in this equality, we find
\[
\sum_{I \in 2^W} \tup{1/x}^{\# I}
= \tup{\underbrace{1+1/x}_{= \tup{1/x} \tup{1+x}}}^{\# W}
= \tup{1/x}^{\# W} \tup{1+x}^{\# W} .
\]

But the definition of $\Delta^\vee$ shows that the faces of $\Delta^\vee$
are the complements (in $W$) of the subsets of $W$ that do not belong to $\Delta$.
Hence, the map $2^W \setminus \Delta \to \Delta^\vee, \  I \mapsto W \setminus I$
is a bijection.
Thus, we can substitute $W \setminus I$ for $I$ in the sum
$\sum\limits_{I \in \Delta^\vee} x^{\# I}$.
Hence, we obtain
\begin{align}
\sum_{I \in \Delta^\vee} x^{\# I}
&= \sum_{I \in 2^W \setminus \Delta} \underbrace{x^{\# \tup{W \setminus I}}}_{\substack{= x^{\# W - \# I} \\ = x^{\# W} \tup{1/x}^{\# I}}}
= x^{\# W} \underbrace{\sum_{I \in 2^W \setminus \Delta} \tup{1/x}^{\# I}}_{= \sum\limits_{I \in 2^W} \tup{1/x}^{\# I}
 - \sum\limits_{I \in \Delta} \tup{1/x}^{\# I}} \nonumber\\
&= x^{\# W} \underbrace{\sum_{I \in 2^W} \tup{1/x}^{\# I}}_{= \tup{1/x}^{\# W} \tup{1+x}^{\# W}}
 - x^{\# W} \sum_{I \in \Delta} \tup{1/x}^{\# I} \nonumber\\
&= \underbrace{x^{\# W}\tup{1/x}^{\# W}}_{= 1} \tup{1+x}^{\# W} - x^{\# W} \sum_{I \in \Delta} \tup{1/x}^{\# I}
\nonumber \\
& = \tup{1+x}^{\# W} - x^{\# W} \sum_{I \in \Delta} \tup{1/x}^{\# I} .
\label{pf.lem.fpol.dual.1}
\end{align}
But the definition of the f-polynomial of a simplicial complex shows that
\[
f_{\Delta}\tup{1/x}
= \sum_{I \in \Delta} \tup{1/x}^{\# I}
\qquad \text{ and } \qquad
f_{\Delta^\vee}\tup{x}
= \sum_{I \in \Delta^\vee} x^{\# I} .
\]
In light of this, the equality \eqref{pf.lem.fpol.dual.1} rewrites as $f_{\Delta^\vee}\tup{x} = \tup{1+x}^{\# W} - x^{\# W} f_{\Delta}\tup{1/x}$. This proves Lemma~\ref{lem.fpol.dual}.
\end{proof}

\end{verlong}

Next, we start analyzing quasi-cycles of deletions and contractions
(see Subsection~\ref{subsect.main.fpol} for the definition of a quasi-cycle).

\begin{lemma} \label{lem.G/e-cyc}
	Let $e \in E$ be any arc with source $s$.
	Then, any cycle of $G$ that does not contain $s$
	is also a cycle of $G \backslash e$ and a cycle of $G / e$.
\end{lemma}

\begin{proof}
Consider a cycle of $G$ that does not contain $s$.
Then, this cycle cannot use the arc $e$ (since $s$ is
the source of $e$), and thus is also a cycle of
$G \backslash e$.

It remains to show that it is also a cycle of $G / e$.
It is certainly a closed walk of $G / e$, since it does
not use the arc $e$.
It remains to show that its vertices remain distinct
when the arc $e$ is contracted.
But this is clear, since our cycle does not contain $s$.
\end{proof}

\begin{lemma} \label{lem.G/e.quasi1}
    Let $e \in E$ be a useless arc.
    Let $k$ be a positive integer such that $G$ has
    at least $k$ disjoint quasi-cycles.
    Then, $G \backslash e$ has at least $k-1$ disjoint quasi-cycles.
\end{lemma}

\begin{proof}
Any path of $G \backslash e$ is a path of $G$.
Thus, any useless arc of $G$ other than $e$ is also a useless arc of $G \backslash e$.
Furthermore, any cycle of $G$ that does not contain $e$ is also
a cycle of $G \backslash e$.
Combining these two statements, we conclude that any quasi-cycle
of $G$ that does not contain $e$ is also a quasi-cycle of $G \backslash e$.

But $G$ has at least $k$ disjoint quasi-cycles.
Clearly, at least $k-1$ of them do not contain $e$,
and thus are quasi-cycles of $G \backslash e$ as well.
Hence, $G \backslash e$ has at least $k-1$ disjoint quasi-cycles.
\end{proof}

\begin{lemma} \label{lem.G/e.quasiB}
    Let $e \in E$ be an arc whose source is $s$.
    Let $k$ be a nonnegative integer.
    Assume that $G$ has at least $k$ disjoint quasi-cycles but has no useless arcs.%
	\footnote{Thus, of course, a quasi-cycle of $G$ is just the set of the arcs of some cycle.}
    Then, each of the two graphs $G \backslash e$ and $G / e$
    has at least $k$ disjoint quasi-cycles.
\end{lemma}

\begin{proof}
Let $c$ be a quasi-cycle of $G$.
Then, $c$ must be the set of the arcs of a cycle of $G$ (since
$G$ has no useless arcs).
This cycle cannot contain $s$ (by Lemma~\ref{lem.G/e.basics} (a)),
and thus Lemma~\ref{lem.G/e-cyc} shows that this cycle
is also a cycle of $G \backslash e$ and a cycle of $G / e$.
Hence, the quasi-cycle $c$ is a quasi-cycle of $G \backslash e$
and a quasi-cycle of $G / e$.

Now, forget that we have fixed $c$.
We thus have shown that each quasi-cycle $c$ of $G$ is a quasi-cycle
of $G \backslash e$ and a quasi-cycle of $G / e$.
Therefore, since $G$ has at least $k$ disjoint quasi-cycles,
we conclude that each of the two graphs $G \backslash e$ and $G / e$
has at least $k$ disjoint quasi-cycles.
\end{proof}

We are now able to prove Theorem~\ref{thm.fpols.divis}, which we recall for convenience:

\begin{statement} \textbf{Theorem~\ref{thm.fpols.divis}:}
Let $k$ be a nonnegative integer.
Assume that $G$ has at least $k$ disjoint quasi-cycles.
Then the polynomials $f_{\Pm{G}}\tup{x}$ and $f_{\Pf{G}}\tup{x}$ are divisible by $\tup{1+x}^k$.
\end{statement}

\begin{proof}[Proof of Theorem~\ref{thm.fpols.divis}.]
We proceed in multiple steps.

\textit{Step 1:}
We claim that Theorem~\ref{thm.fpols.divis} is true when $k = 0$.

Indeed, if $k = 0$, then any polynomial is divisible by $\tup{1+x}^k$ (since $\tup{1+x}^k = \tup{1+x}^0 = 1$); thus, Theorem~\ref{thm.fpols.divis} is proven in this case.


\textit{Step 2:}
We claim that Theorem~\ref{thm.fpols.divis} is true when $\# E = 0$.

Indeed, if $\# E = 0$, then $G$ has no arcs and
therefore no quasi-cycles;
but this entails that $k = 0$
(since $G$ has at least $k$ disjoint quasi-cycles),
and therefore Theorem~\ref{thm.fpols.divis} is true (according to Step 1).
Hence, Theorem~\ref{thm.fpols.divis} is proven in the case when $\# E = 0$.

\textit{Step 3:}
We shall now prove Theorem~\ref{thm.fpols.divis} by induction on $\# E$.

The base case ($\# E = 0$) follows from Step 2.

Thus, we proceed to the induction step.
We fix a directed graph $G = \tup{V, E, s, t}$ with $\# E > 0$,
and we assume (as induction hypothesis) that Theorem~\ref{thm.fpols.divis} is true for all graphs with exactly $\# E - 1$ arcs.
Thus, in particular, Theorem~\ref{thm.fpols.divis} is true for $G \backslash e$ and for $G / e$ whenever $e$ is an arc of $G$.
We must now prove Theorem~\ref{thm.fpols.divis} for our graph $G$.

If $k = 0$, then Theorem~\ref{thm.fpols.divis} is true by Step 1; thus, we WLOG assume that $k$ is positive.
Hence, $k-1$ is a nonnegative integer.



We are in one of the following two cases:

\textit{Case 1:} The graph $G$ has no useless arcs.

\textit{Case 2:} The graph $G$ has a useless arc.

Let us first consider Case 1. In this case, the graph $G$ has no useless arcs.
But $G$ has at least one arc (because $\# E > 0$).
Let $f$ be such an arc.
The arc $f$ cannot be useless (since $G$ has no useless arcs),
and thus belongs to some $s-t$-path.
This $s-t$-path must have at least one arc (namely, $f$);
let $e$ be its first arc.
Thus, $e$ is an arc with source $s$.

Lemma~\ref{lem.fpol.rec} (applied to $W = E$, $\Delta = \Pm{G}$
and $w = e$) yields
\begin{align}
f_{\Pm{G}}\tup{x}
& = f_{\dl_{\Pm{G}}(e)}\tup{x} + x f_{\lk_{\Pm{G}}(e)}\tup{x} \nonumber\\
& = f_{\Pm{G/e}}\tup{x} + x f_{\Pm{G \backslash e}}\tup{x}
\label{pf.thm.fpols.divis.rec}
\end{align}
(since $\dl_{\Pm{G}}(e) = \Pm{G/e}$ by Lemma~\ref{lem:delcon} (b),
and since $\lk_{\Pm{G}}(e) = \Pm{G \backslash e}$ by Lemma~\ref{lem:delcon} (a)).

Lemma~\ref{lem.fpol.rec} (applied to $W = E$, $\Delta = \Pf{G}$
and $w = e$) yields
\begin{align}
f_{\Pf{G}}\tup{x}
& = f_{\dl_{\Pf{G}}(e)}\tup{x} + x f_{\lk_{\Pf{G}}(e)}\tup{x} \nonumber\\
& = f_{\Pf{G \backslash e}}\tup{x} + x f_{\Pf{G / e}}\tup{x}
\label{pf.thm.fpols.divis.rec-Pf}
\end{align}
(since $\dl_{\Pf{G}}(e) = \Pf{G \backslash e}$ by Lemma~\ref{lem:delcon_pf} (a),
and since $\lk_{\Pf{G}}(e) = \Pf{G/e}$ by Lemma~\ref{lem:delcon_pf} (b)).

But Lemma~\ref{lem.G/e.quasiB} shows that
each of the two graphs $G \backslash e$ and $G / e$
has at least $k$ disjoint quasi-cycles.
By the induction hypothesis, we can thus apply Theorem~\ref{thm.fpols.divis} to each of these two graphs,
and conclude that
the polynomials $f_{\Pm{G \backslash e}}\tup{x}$
and $f_{\Pf{G \backslash e}}\tup{x}$ as well as
the polynomials $f_{\Pm{G/e}}\tup{x}$
and $f_{\Pf{G/e}}\tup{x}$
are divisible by $\tup{1+x}^k$.
Hence, all four addends on the right hand sides of
\eqref{pf.thm.fpols.divis.rec} and
\eqref{pf.thm.fpols.divis.rec-Pf} are divisible by $\tup{1+x}^k$.
Hence, so are the left hand sides.
In other words, both
polynomials $f_{\Pm{G}}\tup{x}$ and $f_{\Pf{G}}\tup{x}$ are divisible by $\tup{1+x}^k$.
Hence, we have proven Theorem~\ref{thm.fpols.divis} for our graph $G$
in Case 1.

Let us next consider Case 2.
In this case, the graph $G$ has a useless arc.
Let $e$ be such an arc.
Lemma~\ref{lem:dumbarc} shows that $\Pm{G}$ is a cone with apex $e$.
Hence, Lemma~\ref{lem.fpol.cone} (b)
(applied to $W = E$, $\Delta = \Pm{G}$ and $w = e$) shows that
$f_{\Pm{G}}\tup{x} = \tup{1+x} f_{\lk_{\Pm{G}}(e)}\tup{x}$.
In view of Lemma~\ref{lem:delcon} (a), this rewrites as
$f_{\Pm{G}}\tup{x} = \tup{1+x} f_{\Pm{G \backslash e}}\tup{x}$.

But Lemma~\ref{lem.G/e.quasi1} shows that the graph
$G \backslash e$ has at least $k-1$ disjoint quasi-cycles.
By the induction hypothesis, we can thus apply Theorem~\ref{thm.fpols.divis}
to this graph with $k$ replaced by $k-1$,
and conclude that 
the polynomials $f_{\Pm{G \backslash e}}\tup{x}$ and $f_{\Pf{G \backslash e}}\tup{x}$
are divisible by $\tup{1+x}^{k-1}$.

Now, the equality
$f_{\Pm{G}}\tup{x} = \tup{1+x} f_{\Pm{G \backslash e}}\tup{x}$
shows that the polynomial $f_{\Pm{G}}\tup{x}$ is divisible
by $\tup{1+x}^k$ (since $f_{\Pm{G \backslash e}}\tup{x}$
is divisible by $\tup{1+x}^{k-1}$).

Furthermore,
Lemma~\ref{lem:dumbarc} shows that $\Pf{G}$ is a cone with apex $e$.
Hence, Lemma~\ref{lem.fpol.cone} (a)
(applied to $W = E$, $\Delta = \Pf{G}$ and $w = e$) shows that
$\dl_{\Pf{G}}(e) = \lk_{\Pf{G}}(e)$.
Also, Lemma~\ref{lem.fpol.cone} (b)
(applied to $W = E$, $\Delta = \Pf{G}$ and $w = e$) shows that
$f_{\Pf{G}}\tup{x} = \tup{1+x} f_{\lk_{\Pf{G}}(e)}\tup{x}$.
In view of $\dl_{\Pf{G}}(e) = \lk_{\Pf{G}}(e)$, this rewrites as
$f_{\Pf{G}}\tup{x} = \tup{1+x} f_{\dl_{\Pf{G}}(e)}\tup{x}$.
In view of Lemma~\ref{lem:delcon_pf} (a), this rewrites as
$f_{\Pf{G}}\tup{x} = \tup{1+x} f_{\Pf{G \backslash e}}\tup{x}$.
Thus, the polynomial $f_{\Pf{G}}\tup{x}$ is divisible
by $\tup{1+x}^k$ (since $f_{\Pf{G \backslash e}}\tup{x}$
is divisible by $\tup{1+x}^{k-1}$).

We have now shown that both
polynomials $f_{\Pm{G}}\tup{x}$ and $f_{\Pf{G}}\tup{x}$ are divisible by $\tup{1+x}^k$.
Hence, we have proven Theorem~\ref{thm.fpols.divis} for our graph $G$
in Case 2.

We have now proven Theorem~\ref{thm.fpols.divis}
in both Cases 1 and 2; hence, Theorem~\ref{thm.fpols.divis} always
holds for our graph $G$.
This completes the induction step, and with it
the proof of Theorem~\ref{thm.fpols.divis}.
\end{proof}

\section{\label{sect.cpol}The c-polynomial}

\subsection{Definition}

Our following analysis of $\Pm{G}$ and $\Pf{G}$ will rely on certain
simple invariants of $G$ that are most
conveniently recorded under the umbrella of a polynomial. To define it, we
need the following simple lemma.%
\footnote{Recall that $\NS$ denotes the set of all nonsinks of $G$.}

\begin{lemma}
\label{lem.cpol.wd}
Assume that $E\neq\varnothing$.
Then, $\#\NS -1$ and $\#E-\#\NS$ are nonnegative integers.
\end{lemma}


\begin{proof}
The source of any arc $e\in E$ is a nonsink of $G$ (since it is the source of
an arc), and thus belongs to $\NS$. Hence, the map%
\begin{align*}
E  &  \rightarrow \NS,\\
e  &  \mapsto\left(  \text{the source of }e\right)
\end{align*}
is well-defined. This map is furthermore surjective (since each $v\in
\NS$ is a nonsink of $G$, and thus (by definition) the source of some
arc $e\in E$). Hence, we have found a surjective map from $E$ to $\NS$.
Thus, $\#E\geq\#\NS$. Therefore, $\#E-\#\NS$ is a nonnegative
integer. It remains to prove that so is $\#\NS -1$.

There exists at least one arc $e\in E$ (since $E\neq\varnothing$). Pick such
an arc $e$. Its source must be a nonsink of $G$ (since it is the source of an
arc), i.e., an element of $\NS$. Hence, the set $\NS$ has at
least one element. In other words, $\#\NS\geq1$. Hence, $\#\NS -1$
is a nonnegative integer. This completes the proof of Lemma \ref{lem.cpol.wd}.
\end{proof}

\begin{definition}
\label{def.cpol.cpol}Assume that $E\neq\varnothing$.
Then, we define the \emph{c-polynomial} of $G$ to be the polynomial
\[
c_{G}\left(  x,y\right)  :=%
\begin{cases}
0, & \text{if }G\text{ has a useless arc or a cycle};\\
x^{\#\NS -1}y^{\#E-\#\NS}, & \text{otherwise}%
\end{cases}
\]
in $\mathbb{Z}\left[  x,y\right]  $. (This polynomial is well-defined, since
Lemma \ref{lem.cpol.wd} shows that both exponents $\#\NS -1$ and
$\#E-\#\NS$ are nonnegative integers.)
\end{definition}

Note that $c_{G}\left(  x,y\right)  $ depends not only on the underlying
digraph $\left(  V,E\right)  $ but also on the vertices $s$ and $t$.

\subsection{The recursion for c-polynomials}

The most useful feature of the c-polynomial (to us) will be the following
recursive formula.

\begin{lemma}
\label{lem.cpol.rec}Assume that $G$ has no useless arcs, and that $\#E>1$.
Let $e\in E$ be an arc whose source is $s$. Then,%
\[
c_{G}\left(  x,y\right)  =xc_{G/e}\left(  x,y\right)  +yc_{G\backslash
e}\left(  x,y\right)  .
\]

\end{lemma}

\begin{proof}
From $\#E>1$, we obtain 
$E\setminus\left\{  e\right\}  \neq\varnothing$.
Hence, $c_{G/e}\tup{x,y}$ and $c_{G\backslash e}\tup{x,y}$ are defined.

If $G$ has a cycle, then so do $G/e$ and $G\backslash e$ (by Lemma
\ref{lem.G/e.B}), and thus all three polynomials $c_{G}\left(  x,y\right)  $
and $c_{G/e}\left(  x,y\right)  $ and $c_{G\backslash e}\left(  x,y\right)  $
equal $0$ (by Definition \ref{def.cpol.cpol}). Hence, in this case, the claim
of Lemma \ref{lem.cpol.rec} boils down to $0=x0+y0$, which is obvious.

Thus, we WLOG assume that $G$ has no cycles. Hence, $G$ has no useless arcs
and no cycles. Definition \ref{def.cpol.cpol} thus yields%
\begin{equation}
c_{G}\left(  x,y\right)  =x^{\#\NS -1}y^{\#E-\#\NS}.
\label{pf.lem.cpol.rec.cG=}
\end{equation}

Let $s^{\prime}$ be the target of the arc $e$. We are in one of the following
two cases:

\textit{Case 1:} The arc $e$ is not the only arc of $G$ with target
$s^{\prime}$.

\textit{Case 2:} The arc $e$ is the only arc of $G$ with target $s^{\prime}$.

Let us first consider Case 1. In this case, the arc $e$ is not the only arc
of $G$ with target $s^{\prime}$. Hence, Lemma \ref{lem:induction_step_1} shows
that $G/e$ has a useless arc. Hence, Definition \ref{def.cpol.cpol} yields
\begin{align}
c_{G/e}\left(  x,y\right)  =0.
\label{pf.lem.cpol.rec.c1.0}
\end{align}
Furthermore, the graph $G\backslash e$ is a subgraph of $G$, and thus has no
cycles (since $G$ has no cycles). If $G\backslash e$ had a useless arc, then
Lemma \ref{lem.G/e.C} would yield that the graph $G/e$ has no cycles and no
useless arcs; but this would contradict the fact that $G/e$ has a useless
arc. Thus, $G\backslash e$ has no useless arcs.

Therefore, Lemma \ref{lem.G/e.Dnew} shows that the graph $G\backslash e$
has the same nonsinks as $G$. In other words, the set of nonsinks of
$G\backslash e$ is $\NS$ (since the set of nonsinks of $G$ is
$\NS$).

We now know that the graph $G\backslash e$ has no useless arcs and no cycles,
and its set of nonsinks is $\NS$, whereas its arc set is
$E\setminus\left\{  e\right\}  $ (by its definition). Thus, Definition
\ref{def.cpol.cpol} yields
\begin{align*}
c_{G\backslash e}\left(  x,y\right)   &  =x^{\#\NS -1}y^{\#\left(
E\setminus\left\{  e\right\}  \right)  -\#\NS}\\
&  =x^{\#\NS -1}y^{\#E-1-\#\NS}\ \ \ \ \ \ \ \ \ \ \left(
\text{since }\#\left(  E\setminus\left\{  e\right\}  \right)  =\#E-1\right)  .
\end{align*}
Using this equality and using \eqref{pf.lem.cpol.rec.c1.0}, we have
\begin{align*}
x c_{G/e}\left(  x,y\right) + y c_{G\backslash e} \tup{x,y}
&  =x0+yx^{\#\NS -1}y^{\#E-1-\#\NS} \\
& =yx^{\#\NS -1}y^{\#E-1-\#\NS}=x^{\#\NS -1}y^{\#E-\#\NS}.
\end{align*}
Comparing this with \eqref{pf.lem.cpol.rec.cG=}, we obtain $c_{G}\left(
x,y\right)  =xc_{G/e}\left(  x,y\right)  +yc_{G\backslash e}\left(
x,y\right)  $. Hence, Lemma \ref{lem.cpol.rec} is proved in Case 1.

Let us now consider Case 2. In this case, the arc $e$ is the only arc of $G$
with target $s^{\prime}$. Moreover, the arc $e$ of $G$ is not useless (since
$G$ has no useless arcs). Hence, Lemma \ref{lem:induction_step_2b} yields
that $G\backslash e$ has a useless arc. Therefore,
Definition \ref{def.cpol.cpol} yields
\begin{align}
c_{G\backslash e}\left(  x,y\right)  = 0 .
\label{pf.lem.cpol.rec.c2.0}
\end{align}

Furthermore, Lemma \ref{lem.G/e.C} shows that the graph $G/e$ has no cycles
and no useless arcs.

Let $V^{\prime\prime}$ be the set of all nonsinks of $G/e$. Lemma
\ref{lem.G/e.basics} (e) yields that $s^{\prime}\neq t$. Thus, Lemma
\ref{lem.G/e.basics} (d) shows that the graph $G/e$ has exactly one fewer
nonsink than $G$. In other words, $\#V^{\prime\prime}=\#\NS -1$.
Next, recall that $G/e$ has no cycles
and no useless arcs, and the set of all nonsinks of $G/e$ is $V^{\prime
\prime}$, whereas the arc set of $G/e$ is $E\setminus\left\{  e\right\}  $.
Thus, Definition \ref{def.cpol.cpol} yields%
\begin{align*}
c_{G/e}\left(  x,y\right)   &  =x^{\#V^{\prime\prime}-1}y^{\#\left(
E\setminus\left\{  e\right\}  \right)  -\#V^{\prime\prime}}\\
&  =x^{\left(  \#\NS -1\right)  -1}y^{\left(  \#E-1\right)  -\left(
\#\NS -1\right)  }\ \ \ \ \ \ \ \ \ \ \left(
\begin{array}
[c]{c}%
\text{since }\#V^{\prime\prime}=\#\NS -1\\
\text{and }\#\left(  E\setminus\left\{  e\right\}  \right)  =\#E-1
\end{array}
\right) \\
&  =x^{\#\NS -2}y^{\#E-\#\NS}.
\end{align*}
Using this equality and using \eqref{pf.lem.cpol.rec.c2.0}, we obtain
\begin{align*}
x c_{G/e}\left(  x,y\right)
+ y c_{G\backslash e}\left(  x,y\right)
&  =xx^{\#\NS -2}y^{\#E-\#\NS}+y0 \\
&= xx^{\#\NS -2} y^{\#E-\#\NS}
=x^{\#\NS -1}y^{\#E-\#\NS}.
\end{align*}
Comparing this with \eqref{pf.lem.cpol.rec.cG=}, we obtain $c_{G}\left(
x,y\right)  =xc_{G/e}\left(  x,y\right)  +yc_{G\backslash e}\left(
x,y\right)  $. Hence, Lemma \ref{lem.cpol.rec} is proved in Case 2.

We have now proved Lemma \ref{lem.cpol.rec} in all cases.
\end{proof}

\section{\label{sect.euler-proof}Proof of the Euler characteristic}

We shall now work towards computing the Euler characteristics of our complexes.

\subsection{General facts about Euler characteristics}

We begin with some general properties of reduced Euler characteristics.
Recall that these were defined in Subsection~\ref{subsect.main.euler}.

\begin{lemma} \label{lem.chitil.rec}
Let $\Delta$ be a simplicial complex on the ground set $W$. Let $w \in W$. Then,
\[
\chitil\tup{\Delta} = \chitil\tup{\dl_\Delta(w)} - \chitil\tup{\lk_\Delta(w)} .
\]
\end{lemma}

\begin{proof}
Lemma~\ref{lem.fpol.rec} yields
$f_{\Delta}\tup{x} = f_{\dl_\Delta(w)}\tup{x} + x f_{\lk_\Delta(w)}\tup{x}$.
Substituting $-1$ for $x$ on both sides of this equality,
we find
$f_{\Delta}\tup{-1} = f_{\dl_\Delta(w)}\tup{-1} - f_{\lk_\Delta(w)}\tup{-1}$.
In view of \eqref{eq.fDelta.-1}, this rewrites as
$-\chitil\tup{\Delta} = -\chitil\tup{\dl_\Delta(w)} - \tup{-\chitil\tup{\lk_\Delta(w)}}$.
Multiplying both sides of this equality by $-1$, we obtain precisely
the claim of Lemma~\ref{lem.chitil.rec}.
\end{proof}

\begin{lemma} \label{lem.chitil.dual}
Let $\Delta$ be a simplicial complex on a nonempty ground set $W$. Then,
\[
\chitil\tup{\Delta^\vee} = \tup{-1}^{\# W - 1} \chitil\tup{\Delta} .
\]
\end{lemma}

\begin{proof}
By definition, the faces of $\Delta^\vee$
are the complements (in $W$) of the subsets of $W$ that do not belong to $\Delta$.
Hence, the map $2^W \setminus \Delta \to \Delta^\vee, \  I \mapsto W \setminus I$
is a bijection.
Thus, 
\begin{align}
\chitil\tup{\Delta^\vee} &= \sum_{I \in \Delta^\vee} \tup{-1}^{\# I - 1} =  \sum_{I \in 2^W \setminus \Delta} \tup{-1}^{\# \tup{W \setminus I} - 1} \nonumber\\
&=  \tup{-1}^{\# W - 1} \sum_{I \in 2^W \setminus \Delta} \tup{-1}^{\# I }.
\label{pf.lem.chitil.dual.1}
\end{align}
Since $W$ is nonempty, we have $\sum_{I \in 2^W} \tup{-1}^{\# I} = 0$, and therefore 
\[
\sum_{I \in 2^W \setminus \Delta} \tup{-1}^{\# I } = - \sum_{I \in \Delta} \tup{-1}^{\# I } = \sum_{I \in \Delta} \tup{-1}^{\# I - 1 } = \chitil\tup{\Delta}.
\]
Substituting this into \eqref{pf.lem.chitil.dual.1}, we obtain Lemma~\ref{lem.chitil.dual}.
\end{proof}

\begin{verlong}
\begin{lemma} \label{lem.chitil.sphere}
Let $W$ be a nonempty finite set.
Let $\Delta$ be the simplicial complex on the ground set $W$ that consists of all proper subsets of $W$.
Then, $\chitil\tup{\Delta} = \tup{-1}^{\# W}$.
\end{lemma}

\begin{proof}
Let $\Gamma$ be the simplicial complex on the ground set $W$ whose only face is the empty set $\varnothing$.
Then, it is easy to see that $\chitil\tup{\Gamma} = -1$ and that $\Gamma = \Delta^\vee$. Hence, $\chitil\tup{\Gamma} = \chitil\tup{\Delta^\vee} = \tup{-1}^{\# W - 1} \chitil\tup{\Delta}$ (by Lemma~\ref{lem.chitil.dual}).
Comparing this with $\chitil\tup{\Gamma} = -1$, we find $\tup{-1}^{\# W - 1} \chitil\tup{\Delta} = -1$ and thus
$\chitil\tup{\Delta} = \tup{-1}^{\# W}$.
This proves Lemma~\ref{lem.chitil.sphere}.
\end{proof}

\begin{lemma} \label{lem.chitil.ball}
Let $W$ be a finite set.
Let $\Delta$ be the simplicial complex on the ground set $W$ that consists of all subsets of $W$.
Then, $\chitil\tup{\Delta} = 0$ if $W$ is nonempty, and $\chitil\tup{\Delta} = -1$ otherwise.
\end{lemma}

\begin{proof}
We WLOG assume that $W$ is nonempty (since the other case is easy).
Let $\Gamma$ be the simplicial complex on the ground set $W$ that has no faces at all.
Then, it is easy to see that $\chitil\tup{\Gamma} = 0$ and that $\Gamma = \Delta^\vee$. Hence, $\chitil\tup{\Gamma} = \chitil\tup{\Delta^\vee} = \tup{-1}^{\# W - 1} \chitil\tup{\Delta}$ (by Lemma~\ref{lem.chitil.dual}).
Comparing this with $\chitil\tup{\Gamma} = 0$, we find $\tup{-1}^{\# W - 1} \chitil\tup{\Delta} = 0$ and thus
$\chitil\tup{\Delta} = 0$.
This proves Lemma~\ref{lem.chitil.ball}.
\end{proof}
\end{verlong}

\begin{lemma} \label{lem.chitil.cone}
Let $\Delta$ be a simplicial complex on a ground set $W$. Let $w \in W$. Assume that $\Delta$ is a cone with apex $w$.
Then, $\chitil\tup{\Delta} = 0$.
\end{lemma}

\begin{proof}
Lemma~\ref{lem.fpol.cone} (b) yields
$f_{\Delta}\tup{x} = \tup{1+x} f_{\lk_\Delta(w)}\tup{x}$.
Substituting $-1$ for $x$ on both sides of this equality, we obtain
$f_{\Delta}\tup{-1} = \underbrace{\tup{1+\tup{-1}}}_{=0} f_{\lk_\Delta(w)}\tup{-1} = 0$.
In view of \eqref{eq.fDelta.-1}, this rewrites as
$- \chitil\tup{\Delta} = 0$. Thus,
$\chitil\tup{\Delta} = 0$.
This proves Lemma~\ref{lem.chitil.cone}.
\end{proof}

\subsection{The Euler characteristics of $\Pm{G}$ and $\Pf{G}$}

Recall the c-polynomial $c_{G}\left(  x,y\right)  $ introduced in Definition
\ref{def.cpol.cpol}. We shall now restate Theorem \ref{thm.PM.chitil} in the
form that is most convenient for our proof.

\begin{lemma}
\label{lem.PM.chitil-c}
Assume that $E \neq \varnothing$. Then,
\[
\widetilde{\chi}\left(  \Pm{G}\right)
= - c_G \left( 1,-1\right)  .
\]

\end{lemma}

This lemma will quickly yield the original form of Theorem \ref{thm.PM.chitil}
(once we compute $c_{G}\left(  1,-1\right)  $ and handle the $E=\varnothing$
case by hand).

\begin{proof}[Proof of Lemma \ref{lem.PM.chitil-c}.]
We proceed in two steps:

\textit{Step 1:} We claim that Lemma \ref{lem.PM.chitil-c} is true when $G$
has a useless arc.

Indeed, assume that $G$ has a useless arc $e$. Then, Lemma~\ref{lem:dumbarc}
shows that $\Pm{G}$ is a cone with apex $e$, and therefore
Lemma~\ref{lem.chitil.cone} (applied to $W=E$, $\Delta=\Pm{G}$ and
$w=e$) shows that $\widetilde{\chi}\left(  \Pm{G}\right)  =0$.
Meanwhile, Definition \ref{def.cpol.cpol} yields $c_{G}\left(  x,y\right)  =0$
(since $G$ has a useless arc) and therefore $c_{G}\left(  1,-1\right)  =0$.
Therefore, $-c_{G}\left(  1,-1\right)  =-0=0$. Comparing this with
$\widetilde{\chi}\left(  \Pm{G}\right)  =0$, we see that
$\widetilde{\chi}\left(  \Pm{G}\right)  =-c_{G}\left(
1,-1\right)  $. Hence, we have proved Lemma \ref{lem.PM.chitil-c} in the case
when $G$ has a useless arc.

\textit{Step 2:} Let us now prove Lemma \ref{lem.PM.chitil-c} in general.

We proceed by induction on the positive integer $\#E$ (this is a positive
integer, since $E\neq\varnothing$).

\textit{Base case:} We must show that Lemma \ref{lem.PM.chitil-c} holds when
$\#E=1$.

Indeed, assume that $\#E=1$. Thus, $E=\left\{  e\right\}  $ for some arc $e$
of $G$. Consider this arc $e$. If $G$ has a useless arc, then we already
know (from Step 1) that Lemma \ref{lem.PM.chitil-c} is true. Thus, we WLOG
assume that $G$ has no useless arcs. Hence, the arc $e$ is not useless. In
other words, $e$ is contained in an $s-t$-path of $G$. This $s-t$-path cannot
have any other arcs beyond $e$ (since $E=\left\{  e\right\}  $), and thus
must consist of the arc $e$ alone. Thus, the source of $e$ is $s$, whereas
the target of $e$ is $t$. Moreover, the arc $e$ cannot be a self-loop (since
it is contained in a path), and thus we have $s\neq t$. Thus, the digraph $G$
consists of the single arc $s\overset{e}{\longrightarrow}t$ and a (possibly
empty) set of other vertices but no other arcs (since $E=\left\{  e\right\}
$). Thus, we can describe $\Pm{G}$ explicitly: The
only subsets of $E$ are $\varnothing$ and $\left\{  e\right\}  $ (since
$E=\left\{  e\right\}  $), and we have $\varnothing\in\Pm{G}$
(since $E\setminus\varnothing=E$ contains an $s-t$-path) but
$\left\{  e\right\}  \notin\Pm{G}$ (since
$E\setminus\left\{  e\right\}  =\varnothing$ contains no $s-t$-path (because
$s\neq t$)). Hence, $\Pm{G}=\left\{  \varnothing
\right\}  $, so that $\widetilde{\chi}\left(  \Pm{G}
\right)  =\widetilde{\chi}\left(  \left\{  \varnothing\right\}  \right)  =-1$.
On the other hand, our description of $G$ shows that $G$ has no useless arcs
and no cycles, and has exactly $1$ arc (namely, $e$) and exactly $1$ nonsink
(namely, $s$). Definition \ref{def.cpol.cpol} thus shows that $c_{G}\left(
x,y\right)  =x^{1-1}y^{1-1}=x^{0}y^{0}=1$. Hence, $c_{G}\left(  1,-1\right)
=1$, so that $-c_{G}\left(  1,-1\right)  =-1$. Comparing this with
$\widetilde{\chi}\left(  \Pm{G}\right)  =-1$, we
obtain $\widetilde{\chi}\left(  \Pm{G}\right)
=-c_{G}\left(  1,-1\right)  $. This shows that Lemma \ref{lem.PM.chitil-c}
holds for our $G$. Thus, Lemma \ref{lem.PM.chitil-c} is proved when $\#E=1$.
This completes the base case.

\textit{Induction step:} We fix a directed graph $G=\left(  V,E,s,t\right)  $
with $\#E>1$, and we assume (as induction hypothesis) that Lemma
\ref{lem.PM.chitil-c} is true for all graphs with exactly $\#E-1$ arcs.
We must now prove Lemma \ref{lem.PM.chitil-c} for our graph $G$.

If $G$ has a useless arc, then we already know (from Step 1) that Lemma
\ref{lem.PM.chitil-c} is true. Thus, we WLOG assume that $G$ has no useless
arcs. However, $E\neq\varnothing$. Thus, there exists some $f\in E$. Consider
this $f$. The arc $f$ cannot be useless (since $G$ has no useless arcs), and
thus is contained in an $s-t$-path of $G$. This path has at least one arc
(since it contains $f$), and thus has a first arc. Let $e$ be this first
arc. Then, the arc $e\in E$ has source $s$ (since it is the first arc of an
$s-t$-path). Lemma \ref{lem.cpol.rec} thus yields%
\[
c_{G}\left(  x,y\right)  =xc_{G/e}\left(  x,y\right)  +yc_{G\backslash
e}\left(  x,y\right)  .
\]
Substituting $1$ and $-1$ for $x$ and $y$ in this equality, we find%
\begin{align}
c_{G}\left(  1,-1\right)    & =1c_{G/e}\left(  1,-1\right)  +\left(
-1\right)  c_{G\backslash e}\left(  1,-1\right)  \nonumber\\
& =c_{G/e}\left(  1,-1\right)  -c_{G\backslash e}\left(  1,-1\right)
.\label{pf.lem.PM.chitil-c.c-rec}%
\end{align}

However, $\#E>1$ entails $E\not \subseteq \left\{  e\right\}  $, thus
$E\setminus\left\{  e\right\}  \neq\varnothing$. The graph $G/e$ has arc set
$E\setminus\left\{  e\right\}  $, whose size is $\#\left(  E\setminus\left\{
e\right\}  \right)  =\#E-1$. Thus, by our induction hypothesis, Lemma
\ref{lem.PM.chitil-c} is true for $G/e$ instead of $G$. In other words, we
have%
\begin{equation}
\widetilde{\chi}\left(  \Pm{G/e}\right)  =-c_{G/e}\left(
1,-1\right)  .\label{pf.lem.PM.chitil-c.IH1}%
\end{equation}
The same argument (applied to $G\backslash e$ instead of $G/e$) shows that%
\begin{equation}
\widetilde{\chi}\left(  \Pm{G\backslash e}\right)  =-c_{G\backslash
e}\left(  1,-1\right)  .\label{pf.lem.PM.chitil-c.IH2}%
\end{equation}

But Lemma~\ref{lem.chitil.rec} (applied to $W=E$, $\Delta=\Pm{G}$
and $w=e$) yields
\begin{align*}
\widetilde{\chi}\left(  \Pm{G}\right)   &  =\widetilde{\chi}\left(
\dl_{\Pm{G}}\left(  e\right)  \right)  -\widetilde{\chi}\left(
\lk_{\Pm{G}}\left(  e\right)  \right)  \\
&  =\widetilde{\chi}\left(  \Pm{G/e}  \right)
-\widetilde{\chi}\left(  \Pm{G \backslash e}  \right)
\end{align*}
(since $\dl_{\Pm{G}}\left(  e\right)  = \Pm{G/e}$
by Lemma~\ref{lem:delcon} (b), and since
$\lk_{\Pm{G}}\left(  e\right)  = \Pm{G \backslash e}  $ by
Lemma~\ref{lem:delcon} (a)). In view of \eqref{pf.lem.PM.chitil-c.IH1} and
\eqref{pf.lem.PM.chitil-c.IH2}, we can rewrite this as
\begin{align*}
\widetilde{\chi}\left(  \Pm{G}\right)    & =\left(  -c_{G/e}\left(
1,-1\right)  \right)  -\left(  -c_{G\backslash e}\left(  1,-1\right)  \right)
\\
& =-\underbrace{\left(  c_{G/e}\left(  1,-1\right)  -c_{G\backslash e}\left(
1,-1\right)  \right)  }_{=c_{G}\left(  1,-1\right)  \quad \text{(by
\eqref{pf.lem.PM.chitil-c.c-rec})}}=-c_{G}\left(  1,-1\right)  .
\end{align*}
This shows that Lemma \ref{lem.PM.chitil-c} holds for our $G$. This completes
the induction step. Thus, Lemma \ref{lem.PM.chitil-c} is proved by induction.
\end{proof}

We can now prove Theorem~\ref{thm.PM.chitil} in its original form:

\begin{statement} \textbf{Theorem~\ref{thm.PM.chitil}:}
\begin{enumerate}
    \item[(a)] If $G$ contains a useless arc or a cycle or satisfies $\tup{E = \varnothing \text{ and } s \neq t}$, then $\chitil\tup{\Pm{G}} = 0$.
    \item[(b)] Otherwise, $\chitil\tup{\Pm{G}} = \tup{-1}^{\# E - \# \NS + 1}$.
\end{enumerate}
\end{statement}

\begin{proof}[Proof of Theorem \ref{thm.PM.chitil}.]
We are in one of the following four cases:

\textit{Case 1:} We have $E=\varnothing$ and $s=t$.

\textit{Case 2:} We have $E=\varnothing$ and $s\neq t$.

\textit{Case 3:} We have $E\neq\varnothing$, and the graph $G$ has a useless
arc or a cycle.

\textit{Case 4:} We have $E\neq\varnothing$, and the graph $G$ has no useless
arcs and no cycles.

Let us first consider Case 1. In this case, we have $E=\varnothing$ and $s=t$.
Thus, the graph $G$ has no arcs (since $E=\varnothing$). Moreover, $s=t$
shows that the set $\varnothing\setminus\varnothing$ contains an $s-t$-path
(namely, the trivial path, with no arcs at all). Hence, $\varnothing
\in\Pm{G}$, so that $\Pm{G}
=\left\{  \varnothing\right\}  $. Therefore, $\widetilde{\chi}\left(
\Pm{G}\right)  =\widetilde{\chi}\left(  \left\{
\varnothing\right\}  \right)  =-1$. On the other hand, $G$ has no nonsinks
(since $G$ has no arcs); thus, $\NS=\varnothing$ and therefore
$\#\NS=0$. Combined with $\#E=0$ (since $E=\varnothing$), this yields
$\left(  -1\right)  ^{\#E-\#\NS+1}=\left(  -1\right)  ^{0-0+1}=-1$.
Comparing this with $\widetilde{\chi}\left(  \Pm{G}
\right)  =-1$, we obtain $\widetilde{\chi}\left(  \Pm{G}  \right)
=\left(  -1\right)  ^{\#E-\#\NS+1}$, which is
precisely the value that Theorem \ref{thm.PM.chitil} (b) predicts for
$\widetilde{\chi}\left(  \Pm{G} \right)  $. Thus,
Theorem \ref{thm.PM.chitil} is proved in Case 1.

Let us next consider Case 2. In this case, we have $E=\varnothing$ and $s\neq
t$. Thus, the graph $G$ has no arcs (since $E=\varnothing$). Moreover, $s\neq
t$ shows that any $s-t$-path must contain at least one arc. Hence, $G$ has no
$s-t$-path (since $G$ has no arcs). Therefore, $\Pm{G}
=\varnothing$, so that $\widetilde{\chi}\left(  \Pm{G}
\right)  =\widetilde{\chi}\left(  \varnothing\right)  =0$. But this is
precisely the value that Theorem \ref{thm.PM.chitil} (a) predicts for
$\widetilde{\chi}\left(  \Pm{G} \right)  $ (since $E=\varnothing$ and $s\neq
t$). Thus, Theorem \ref{thm.PM.chitil} is proved in Case 2.

Next, let us consider Case 3. In this case, we have $E\neq\varnothing$, and
the graph $G$ has a useless arc or a cycle. Hence, Definition
\ref{def.cpol.cpol} yields $c_{G}\left(  x,y\right)  =0$. Substituting $1$ and
$-1$ for $x$ and $y$ in this equality, we find $c_{G}\left(  1,-1\right)  =0$.
However, Lemma \ref{lem.PM.chitil-c} yields%
\[
\widetilde{\chi}\left(  \Pm{G}\right)
= - c_{G}\left(  1,-1\right)
=0
\qquad \left(\text{since }c_{G}\left(  1,-1\right) = 0\right).
\]
But this is precisely the value that Theorem \ref{thm.PM.chitil} (a)
predicts for $\widetilde{\chi}\left(  \Pm{G}\right)  $. Thus, Theorem
\ref{thm.PM.chitil} is proved in Case 3.

Finally, let us consider Case 4. In this case, we have $E\neq\varnothing$, and
the graph $G$ has no useless arcs and no cycles. Hence, Definition
\ref{def.cpol.cpol} yields $c_{G}\left(  x,y\right)  =x^{\#\NS%
-1}y^{\#E-\#\NS}$. Substituting $1$ and $-1$ for $x$ and $y$ in this
equality, we find $c_{G}\left(  1,-1\right)
=\left(  -1\right)  ^{\#E-\#\NS}$.
However, Lemma \ref{lem.PM.chitil-c} yields%
\[
\widetilde{\chi}\left(  \Pm{G}\right)
=-\underbrace{c_{G}\left(  1,-1\right)  }_{=\left(  -1\right)
^{\#E-\#\NS}}=-\left(  -1\right)  ^{\#E-\#\NS}=\left(
-1\right)  ^{\#E-\#\NS -1}.
\]
But this is precisely the value that Theorem \ref{thm.PM.chitil} (b)
predicts for $\widetilde{\chi}\left(  \Pm{G}\right)  $. Thus, Theorem
\ref{thm.PM.chitil} is proved in Case 4.

We have now proved Theorem \ref{thm.PM.chitil} in all four cases.
\end{proof}

It is even easier to prove Theorem~\ref{thm.PF.chitil}:

\begin{statement} \textbf{Theorem~\ref{thm.PF.chitil}:}
If $E \neq \varnothing$, then:
\begin{enumerate}
    \item[(a)] If $G$ contains a useless arc or a cycle, then $\chitil\tup{\Pf{G}} = 0$.
    \item[(b)] Otherwise, $\chitil\tup{\Pf{G}} = \tup{-1}^{\# \NS}$.
\end{enumerate}
If $E = \varnothing$, then $\chitil\tup{\Pf{G}}$ equals $0$ if $s = t$ and $-1$ otherwise.
\end{statement}

\begin{proof}[Proof of Theorem~\ref{thm.PF.chitil}.]
The case $E = \varnothing$ is straightforward and left to the reader;
so we WLOG assume that $E \neq \varnothing$.
Lemma~\ref{lem.duals} yields
$\Pf{G} = \left(  \Pm{G} \right)  ^{\vee}$. Hence,
$\widetilde{\chi}\left(  \Pf{G} \right)  = \widetilde{\chi}\left(
\left(  \Pm{G} \right)  ^{\vee}\right)  = \left(  -1 \right)  ^{\# E
- 1} \widetilde{\chi}\left(  \Pm{G} \right)  $ (by
Lemma~\ref{lem.chitil.dual}, applied to $W = E$ and $\Delta= \Pm{G}$).
Substituting the expression for $\widetilde{\chi}\left( \Pm{G}
\right)  $ given in Theorem~\ref{thm.PM.chitil} into this equation, we
find an expression for $\widetilde{\chi}\left(  \Pf{G} \right)  $.
This proves Theorem~\ref{thm.PF.chitil}.
\end{proof}

The two corollaries about the parities of the sizes of our two complexes are now easily obtained:

\begin{statement} \textbf{Corollary~\ref{cor.PM.parity}:}
\begin{enumerate}
    \item[(a)] If $G$ contains a useless arc or a cycle or satisfies $\tup{E = \varnothing \text{ and } s \neq t}$, then $\# \tup{\Pm{G}}$ is even.
    \item[(b)] Otherwise, $\# \tup{\Pm{G}}$ is odd.
\end{enumerate}
\end{statement}

\begin{proof}[Proof of Corollary~\ref{cor.PM.parity}.]
If $\Delta$ is any simplicial complex, then
\begin{align}
\# \Delta \equiv \chitil\tup{\Delta} \mod 2 ,
\label{pf.cor.PM.parity.1}
\end{align}
because the definition of the reduced Euler characteristic (specifically, the right hand side of \eqref{eq.chitil.def}) shows that
\[
\chitil\tup{\Delta}
= \sum_{I \in \Delta} \underbrace{\tup{-1}^{\# I - 1}}_{\equiv 1 \mod 2}
\equiv \sum_{I \in \Delta} 1 = \# \Delta \mod 2 .
\]
Applying this to $\Delta = \Pm{G}$, we obtain
$\# \tup{\Pm{G}} \equiv \chitil\tup{\Pm{G}} \mod 2$.
Hence, Corollary~\ref{cor.PM.parity} follows from
Theorem~\ref{thm.PM.chitil}.
\end{proof}

\begin{statement} \textbf{Corollary~\ref{cor.PF.parity}:}
\begin{enumerate}
    \item[(a)] If $G$ contains a useless arc or a cycle or satisfies $\tup{E = \varnothing \text{ and } s = t}$, then $\# \tup{\Pf{G}}$ is even.
    \item[(b)] Otherwise, $\# \tup{\Pf{G}}$ is odd.
\end{enumerate}
\end{statement}

\begin{proof}[Proof of Corollary~\ref{cor.PF.parity}.]
This is analogous to the proof of Corollary~\ref{cor.PM.parity}
(but relies on Theorem~\ref{thm.PF.chitil} instead of
Theorem~\ref{thm.PM.chitil}).
\end{proof}

\section{\label{sect.dmt}Discrete Morse theory}

\subsection{Definitions and topological meaning}

In this section, we shall recall the basics of Forman's \emph{discrete Morse
theory} (foreshadowed by Brown's \cite{Brown92}). We refer to
\cite{Forman-user} and \cite{Kozlov20} for deeper-going expositions of this
subject. Here we shall only recall the basics that we need. We will follow the
modern terminology of \textquotedblleft acyclic matchings\textquotedblright,
as in Kozlov's \cite{Kozlov20}.

First, we introduce a basic set-theoretic notation.

\begin{definition}
\label{def.succ}
Let $A$ and $B$ be two sets. Then, we write $A\prec B$ if
there exists some $b\in B\setminus A$ such that $B=A\cup\left\{  b\right\}  $.
Equivalently, we write $B\succ A$ in this case.
\end{definition}

Clearly, if $A$ and $B$ are two finite sets, then we have the equivalences
\begin{align*}
\left(  A\prec B\right)  \  &  \Longleftrightarrow\ \left(  B\succ A\right)
\ \Longleftrightarrow\ \left(  A\subseteq B\text{ and }\#\left(  B\setminus
A\right)  =1\right) \\
&  \Longleftrightarrow\ \left(  A\subseteq B\text{ and }\#B=\#A+1\right)  .
\end{align*}
For instance, $\left\{  2,5\right\}  \prec\left\{  2,3,5\right\}  $ but not
$\left\{  2,5\right\}  \prec\left\{  2,3,4,5\right\}  $. The binary relations
$\prec$ and $\succ$ are called \textquotedblleft is covered
by\textquotedblright\ and \textquotedblleft covers\textquotedblright.

Next, we define the notion of a matching (following \cite[Definition
10.6]{Kozlov20}\footnote{In \cite[Definition 10.6]{Kozlov20}, Kozlov works in
a more general setting, replacing a simplicial complex $\Delta$ by an
arbitrary poset. We do not need this generality here; the only posets we will
be using are simplicial complexes $\Delta$, ordered by inclusion (so that the
relation $\prec$ introduced in Definition \ref{def.succ} is precisely the
covering relation of this poset).}).

\begin{definition}
\label{def.parmat}
Let $\Delta$ be a simplicial complex with ground set $W$. A
\emph{partial matching} (or \emph{matching} for short) on $\Delta$ shall
mean a pair $\left(  M,\mu\right)  $, where $M$ is a subset of $\Delta$ (that
is, a set of faces of $\Delta$), and where $\mu:M\rightarrow M$ is an
involution (that is, a map satisfying $\mu\circ\mu=\operatorname*{id}$) with
the property that each $A\in M$ satisfies%
\[
\text{either }\mu\left(  A\right)  \prec A\text{ or }\mu\left(  A\right)
\succ A.
\]

Note that $M$ is uniquely determined by $\mu$ (namely, as the domain of $\mu
$), so that we will refer to $\mu$ alone as a matching.

Given a matching $\left(  M,\mu\right)  $, we shall refer to the faces $A\in M
$ as the \emph{matched} faces of this matching, while the faces $A\in
\Delta\setminus M$ will be called the \emph{unmatched} faces of this matching.
\end{definition}

\begin{example}
\label{exa.parmat.0}
Let $W=\left\{  1,2,3\right\}  $. Let $\Delta$ be the
simplicial complex with ground set $W$ that contains all $8$ subsets of $W$ as
faces. Consider the matching $\left(  M,\mu\right)  $ given by $M=\Delta
\setminus\left\{  \varnothing,W\right\}  $ and%
\begin{align*}
\mu\left(  \left\{  1\right\}  \right)   &  =\left\{  1,2\right\}  ,\qquad
\mu\left(  \left\{  2\right\}  \right)  =\left\{  2,3\right\}  ,\qquad
\mu\left(  \left\{  3\right\}  \right)  =\left\{  3,1\right\}  ,\\
\mu\left(  \left\{  1,2\right\}  \right)   &  =\left\{  1\right\}  ,\qquad
\mu\left(  \left\{  2,3\right\}  \right)  =\left\{  2\right\}  ,\qquad
\mu\left(  \left\{  3,1\right\}  \right)  =\left\{  3\right\}  .
\end{align*}
The unmatched faces of this matching are $\varnothing$ and $W$.
\end{example}

\begin{example}
\label{exa.parmat.1}Let $n$ and $k$ be two nonnegative integers. Let $W$ be an
$n$-element set. Let $\Delta$ be the simplicial complex consisting of all
subsets $A$ of $W$ having size $\#A\leq k$. (This is called the $\left(
k-1\right)  $\emph{-skeleton} of the simplex on $W$.) Pick any element $w\in
W$. Let $M\subseteq\Delta$ be the set of all subsets $A$ of $W$ that satisfy
$\#\left(  A\setminus\left\{  w\right\}  \right)  <k$ (or, equivalently, that
satisfy $\#A<k$ or $\left(  \#A=k\text{ and }w\in A\right)  $). We can then
define a map $\mu:M\rightarrow M$ by setting%
\[
\mu\left(  A\right)  =%
\begin{cases}
A\cup\left\{  w\right\}  , & \text{if }w\notin A;\\
A\setminus\left\{  w\right\}  , & \text{if }w\in A
\end{cases}
\qquad\text{for each }A\in M.
\]
(That is, the map $\mu$ inserts the element $w$ into any face that does not
contain $w$, and removes it from any face that does.) It is easy to see that
this map $\mu$ is well-defined\footnote{Indeed, each $A\in M$ satisfies
$\left(  A\cup\left\{  w\right\}  \right)  \setminus\left\{  w\right\}
=A\setminus\left\{  w\right\}  $, so that $\#\left(  \left(  A\cup\left\{
w\right\}  \right)  \setminus\left\{  w\right\}  \right)  =\#\left(
A\setminus\left\{  w\right\}  \right)  <k$ and therefore $A\cup\left\{
w\right\}  \in M$, and similarly $A\setminus\left\{  w\right\}  \in M$.} and
is a matching on $\Delta$ (or, more precisely, the pair $\left(  M,\mu\right)
$ is).

The unmatched faces of this matching are precisely the $k$-element subsets of
$W$ that do not contain $w$. Their number is $\dbinom{n-1}{k}$.
\end{example}

Discrete Morse theory is interested in matchings with a special property,
defined in terms of cycles (\cite[Definition 10.7]{Kozlov20}),
that we describe now.

\begin{definition}
Let $\Delta$ be a simplicial complex with ground set $W$. Let $\left(
M,\mu\right)  $ be a matching on $\Delta$.

\begin{enumerate}
\item[(a)] A \emph{cycle} of $\mu$ means an $n$-tuple $\left(  F_{1}%
,F_{2},\ldots,F_{n}\right)  $ of distinct faces in $M$ such that $n\geq2$ and%
\[
F_{1}\succ\mu\left(  F_{1}\right)  \prec F_{2}\succ\mu\left(  F_{2}\right)
\prec F_{3}\succ\cdots\prec F_{n}\succ\mu\left(  F_{n}\right)  \prec F_{1}%
\]
(that is, such that each $i\in\left\{  1,2,\ldots,n\right\}  $ satisfies
$F_{i}\succ\mu\left(  F_{i}\right)  \prec F_{i+1}$, where $F_{n+1}:=F_{1}$).

\item[(b)] The matching $\mu$ is said to be \emph{acyclic} if it has no cycle.
\end{enumerate}
\end{definition}

\begin{example}
\ \ 

\begin{enumerate}
\item[(a)] The matching $\mu$ constructed in Example \ref{exa.parmat.0} is not
acyclic. Indeed, the $3$-tuple $\left(  \left\{  1,2\right\}  ,\ \left\{
3,1\right\}  ,\ \left\{  2,3\right\}  \right)  $ is a cycle of $\mu$, since%
\[
\left\{  1,2\right\}  \succ\underbrace{\left\{  1\right\}  }_{=\mu\left(
\left\{  1,2\right\}  \right)  }\prec\left\{  3,1\right\}  \succ
\underbrace{\left\{  3\right\}  }_{=\mu\left(  \left\{  3,1\right\}  \right)
}\prec\left\{  2,3\right\}  \succ\underbrace{\left\{  2\right\}  }%
_{=\mu\left(  \left\{  2,3\right\}  \right)  }\prec\left\{  1,2\right\}  .
\]

\item[(b)] The matching $\mu$ in Example \ref{exa.parmat.1} is acyclic.
Indeed, the faces $F\in M$ satisfying $F\succ\mu\left(  F\right)  $ are
precisely the faces $F\in M$ that contain $w$; therefore, a cycle $\left(
F_{1},F_{2},\ldots,F_{n}\right)  $ of $\mu$ would have to satisfy $w\in F_{1}$
and $w\in F_{2}$, which would easily yield $F_{1}=\mu\left(  F_{1}\right)
\cup\left\{  w\right\}  =F_{2}$, contradicting the distinctness of
$F_{1},F_{2},\ldots,F_{n}$.
\end{enumerate}
\end{example}

Acyclic matchings are the main objects of discrete Morse theory, although
different texts give different definitions whose equivalence is not always
immediate\footnote{The closest notion in Forman's original work is that of a
gradient vector field in \cite[\S 3]{Forman-user}. Indeed, our partial
matchings $\mu$ correspond to Forman's \textquotedblleft discrete vector
fields\textquotedblright\ as defined in \cite[Definition 3.3]{Forman-user}
(specifically, if $\left(  M,\mu\right)  $ is a partial matching, then the set
$\left\{  \left(  A,\mu\left(  A\right)  \right)  \ \mid\ A\in M\text{ and
}A\prec\mu\left(  A\right)  \right\}  $ is a discrete vector field); our
cycles are more or less Forman's \textquotedblleft closed $V$%
-paths\textquotedblright\ (at least those that cannot be broken up into
shorter ones); thus, our acyclic partial matchings correspond to Forman's
\textquotedblleft gradient vector fields of discrete Morse
functions\textquotedblright\ (according to \cite[Theorem 3.5]{Forman-user},
which is proved in \cite[Proposition 14.11]{Kozlov20}). The notion of a
discrete Morse function was regarded as fundamental when Forman originally
conceived discrete Morse theory in 1995, but is nowadays considered as a
refinement whose use is entirely optional.
\par
We note that unmatched faces of a partial matching are called
\textquotedblleft critical simplices\textquotedblright\ in \cite{Forman-user}%
.
\par
There is a subtle but confusing notational disagreement in
the literature. Our conventions follow \cite{Kozlov20}. Meanwhile, more
topologically inclined sources such as \cite{Forman-user} and
\cite{Bjo} have their complexes consist of \textbf{nonempty} subsets of
their ground sets $W$. Accordingly, their matchings slightly differ from ours:
Where our matchings $\mu$ will often match the empty face $\varnothing$
with a singleton face $\set{v}$, their matchings would instead leave
$\set{v}$ unmatched.
}. They give highly useful information on the homotopy type of a complex, as
the following theorem (\cite[Theorem 2.5]{Forman-user}, \cite[Theorem
11.2]{Kozlov20}, actually a particular case of \cite[Proposition 1]{Brown92})
shows.\footnote{Note the requirement that the empty face $\varnothing
\in \Delta$ is unmatched in $\mu$. This is to make our setting compatible
with the notion of CW complexes, which have no $\left(-1\right)$-cells.
Of course, if $\varnothing$ is matched in $\mu$, then we can always make
$\varnothing$ unmatched by removing $\varnothing$ and its partner
(which is a $0$-dimensional face) from the domain (and target) of $\mu$;
so this requirement does not significantly restrict the applicability of
the theorem.}

\begin{theorem}
\label{thm.dmt.cw}
Let $\mu$ be an acyclic matching on a simplicial complex $\Delta$.
Assume that the empty face $\varnothing \in \Delta$ is unmatched in $\mu$.
Then, $\Delta$ is homotopy-equivalent to a CW complex $X$ with the property
that for each $d \geq 0$, the number of $d$-cells in $X$ equals the number
of unmatched $d$-dimensional faces of $\Delta$.
\end{theorem}

Thus, acyclic matchings can be used as a combinatorial proxy for (certain
kinds of) homotopy equivalences, particularly when they have few unmatched
faces. It is via these proxies that we will prove Theorems
\ref{thm.PM.homotopy} and \ref{thm.PF.homotopy}.

\subsection{The unmatched f-polynomial}

First, we shall show some basic properties of acyclic matchings. We will use
the following polynomial fingerprint of matchings.

\begin{definition}
Let $\left(  M,\mu\right)  $ be a matching on a simplicial complex $\Delta$.
Then, we define the \emph{unmatched f-polynomial} of $\mu$ to be the
polynomial
\[
u_{\mu}\left(  x\right)
:= \sum_{I\in\Delta\setminus M}x^{\#I}
\in \mathbb{Z}\left[  x\right]  .
\]
Note that the sum here ranges over all unmatched faces of $\mu$.
\end{definition}

For instance, the matching $\mu$ in Example \ref{exa.parmat.0} has unmatched
f-polynomial $u_{\mu}\left(  x\right)  =x^{0}+x^{3}$, whereas the one in
Example \ref{exa.parmat.1} has unmatched f-polynomial $u_{\mu}\left(
x\right)  =\dbinom{n-1}{k}x^{k}$. A trivial example is the empty matching
$\left(  \varnothing,\varnothing\right)  $, which exists for every simplicial
complex $\Delta$, and whose unmatched f-polynomial $u_{\varnothing}\left(
x\right)  $ is just the usual f-polynomial $f_{\Delta}\left(  x\right)  $
(since $\Delta\setminus\varnothing=\Delta$).

A matching $\mu$ whose unmatched f-polynomial $u_{\mu}\left(  x\right)  $ is a
single monomial $x^{m}$ is one that has only one unmatched face (which has
size $m$, that is, dimension $m-1$). For such matchings, Theorem
\ref{thm.dmt.cw} has the following consequence.\footnote{To recover
Corollary~\ref{cor.dmt.hteq} from Theorem~\ref{thm.dmt.cw}, make sure to
first delete the empty face $\varnothing$ and its partner from the matching
$\mu$, so that $\varnothing$ becomes unmatched and Theorem~\ref{thm.dmt.cw}
applies.}

\begin{corollary}
\label{cor.dmt.hteq}Let $\mu$ be an acyclic matching on a simplicial complex
$\Delta$.

\begin{enumerate}
\item[(a)] If $u_{\mu}\left(  x\right)  =0$, then $\Delta$ is contractible.

\item[(b)] If $u_{\mu}\left(  x\right)  =x^{m}$ for some $m\in\mathbb{N}$,
then $\Delta$ is homotopy-equivalent to a sphere of dimension $m-1$.
\end{enumerate}
\end{corollary}

We note that the condition $u_{\mu}\left(  x\right)  =0$ in Corollary
\ref{cor.dmt.hteq} (a) means that all faces of $\Delta$ are matched; a
simplicial complex $\Delta$ with such a matching $\mu$ is said to be
\emph{collapsible} (see \cite[Theorem 6.4]{Forman-user}).

\subsection{Two reduction lemmas}

To construct acyclic matchings with simple unmatched f-polynomials, we shall
use two basic lemmas. The first one guarantees the collapsibility of any cone
(\cite[Proposition 10.11]{Kozlov20}).

\begin{lemma}
\label{lem.dmt.cone}
Let $\Delta$ be a simplicial complex that is a cone. Then,
$\Delta$ has an acyclic matching $\left(  M,\mu\right)  $ satisfying $u_{\mu
}\left(  x\right)  =0$.
\end{lemma}

To keep this paper self-contained, we shall give a proof of this lemma
in Appendix~\ref{sect.apx-morse}.


The second lemma allows the recursive construction of acyclic matchings, based
on deletions and links. In its essence, it appears to go back to Forman, and
similar facts are found across the literature (e.g., \cite[Proposition
10.13]{Kozlov20}, \cite[Lemma 2.4]{Engstrom08}, \cite[Theorem 11.10]%
{Kozlov-CAT}), but we have not been able to locate the following
version.\footnote{Actually, Lemma \ref{lem.dmt.rec} can be derived from
\cite[Theorem 11.10]{Kozlov-CAT} (= \cite[Theorem 16.8]{Kozlov20})
by taking $P=\Delta$ and $Q=\left\{
0<1\right\}  $ and letting $\varphi:P\rightarrow Q$ be the map that sends each
face $F\in\Delta$ to $1$ if $w\in F$ and to $0$ if $w\notin F$. The acyclic
matching $\left(  M_-,\mu_-\right)  $ on $\dl_{\Delta}\left(  w\right)  $
then becomes an acyclic matching of
$\varphi^{-1}\left(  0\right)  $, whereas the acyclic matching $\left(
M_+,\mu_+\right)  $ on $\lk_{\Delta}\left(
w\right)  $ can be converted into an acyclic matching of $\varphi^{-1}\left(
1\right)  $ (by inserting $w$ into each face).}

\begin{lemma}
\label{lem.dmt.rec}
Let $\Delta$ be a simplicial complex with ground set $W$.
Let $w\in W$.

Let $\left(  M_-,\mu_-\right)  $ be an acyclic matching on
$\dl_{\Delta}\left(  w\right)  $. Let $\left(
M_+,\mu_+\right)  $ be an acyclic matching on
$\lk_{\Delta}\left(  w\right)  $. Then, $\Delta$ has an acyclic matching
$\left(  M,\mu\right)  $ satisfying%
\[
u_{\mu}\left(  x\right)  =u_{\mu_-}\left(  x\right)  +xu_{\mu_+}\left(
x\right)  .
\]

\end{lemma}

This lemma, too, will be proved in Appendix~\ref{sect.apx-morse}.


\section{\label{sect.homtype-proof}Proof of the homotopy types}

Our final goal is to establish the homotopy types of $\Pf{G}$ and
$\Pm{G}$ (Theorems \ref{thm.PM.homotopy} and \ref{thm.PF.homotopy}).
In view of Corollary \ref{cor.dmt.hteq}, it will suffice to prove the
following two Morse-theoretic results.

\begin{theorem}
\label{thm.PM.morse} Assume that $E\neq\varnothing$. Then:

\begin{enumerate}
\item[(a)] If $G$ contains a useless arc or a cycle, then the complex
$\Pm{G}$ has an acyclic matching $\left(  M,\mu\right)  $ satisfying
$u_{\mu}\left(  x\right)  =0$.

\item[(b)] Otherwise, $\Pm{G}$ has an acyclic matching $\left(
M,\mu\right)  $ satisfying $u_{\mu}\left(  x\right)  =x^{\#E-\#\NS}$.
\end{enumerate}
\end{theorem}

\begin{theorem}
\label{thm.PF.morse} Assume that $E\neq\varnothing$. Then:

\begin{enumerate}
\item[(a)] If $G$ contains a useless arc or a cycle, then the complex
$\Pf{G}$ has an acyclic matching $\left(  M,\mu\right)  $ satisfying
$u_{\mu}\left(  x\right)  =0$.

\item[(b)] Otherwise, $\Pf{G}$ has an acyclic matching $\left(
M,\mu\right)  $ satisfying $u_{\mu}\left(  x\right)  =x^{\#\NS -1}$.
\end{enumerate}
\end{theorem}

We will derive both of these from the following lemmas (stated in terms of the
c-polynomial $c_{G}\left(  x,y\right)  $ from Definition \ref{def.cpol.cpol}).

\begin{lemma}
\label{lem.PM.morse-c}Assume that $E\neq\varnothing$. Then, the complex
$\Pm{G}$ has an acyclic matching $\left(
M,\mu\right)  $ satisfying $u_{\mu}\left(  x\right)  =c_{G}\left(  1,x\right)
$.
\end{lemma}

\begin{lemma}
\label{lem.PF.morse-c}Assume that $E\neq\varnothing$. Then, the complex
$\Pf{G}$ has an acyclic matching $\left(
M,\mu\right)  $ satisfying $u_{\mu}\left(  x\right)  =c_{G}\left(  x,1\right)
$.
\end{lemma}

\begin{proof}
[Proof of Lemma \ref{lem.PM.morse-c}.]Our proof is structurally
similar to the proof of Lemma \ref{lem.PM.chitil-c}.
We proceed in two steps:

\textit{Step 1:} We claim that Lemma \ref{lem.PM.morse-c} is true when $G$ has
a useless arc.

Indeed, assume that $G$ has a useless arc. Then, Definition
\ref{def.cpol.cpol} yields $c_{G}\left(  x,y\right)  =0$. By specializing $x$
and $y$ to $1$ and $x$ here, we obtain $c_{G}\left(  1,x\right)  =0$.
Meanwhile,
Lemma~\ref{lem:dumbarc} shows that $\Pm{G}$ is a cone (since $G$ has
a useless arc). Hence, Lemma \ref{lem.dmt.cone} (applied to $W=E$ and
$\Delta=\Pm{G}$) shows that $\Pm{G}$ has an acyclic matching
$\left(  M,\mu\right)  $ satisfying
$u_{\mu}\left(  x\right)  =0$. In other words, $\Pm{G}
$ has an acyclic matching $\left(  M,\mu\right)  $ satisfying $u_{\mu}\left(
x\right)  =c_{G}\left(  1,x\right)  $ (since $c_{G}\left(  1,x\right)  =0$).
Hence, we have proved Lemma \ref{lem.PM.morse-c} in the case when $G$ has a
useless arc.

\textit{Step 2:} Let us now prove Lemma \ref{lem.PM.morse-c} in general.

We proceed by induction on the positive integer $\#E$ (this is a positive
integer, since $E\neq\varnothing$).

\textit{Base case:} We must show that Lemma \ref{lem.PM.morse-c} holds when
$\#E=1$.

Indeed, assume that $\#E=1$. Thus, $E=\left\{  e\right\}  $ for some arc $e$
of $G$. Consider this arc $e$. If $G$ has a useless arc, then we already
know (from Step 1) that Lemma \ref{lem.PM.morse-c} is true. Thus, we WLOG
assume that $G$ has no useless arcs. As in the proof of
Lemma~\ref{lem.PM.chitil-c}, we can now see that
$\Pm{G}=\left\{  \varnothing\right\}  $ and $c_{G}\left(  x,y\right)  =1$.
Substituting $1$ and $x$ for $x$ and $y$ in the latter equality, we obtain
$c_{G}\left(  1,x\right)  =1$. However, $\Pm{G}
=\left\{  \varnothing\right\}  $ shows that the complex $\Pm{G}$ has an acyclic matching $\left(  M,\mu\right)  $ satisfying
$u_{\mu}\left(  x\right)  =1$ (namely, the empty matching $\left(
\varnothing,\varnothing\right)  $). In other words, $\Pm{G}$ has an acyclic matching $\left(  M,\mu\right)  $ satisfying
$u_{\mu}\left(  x\right)  =c_{G}\left(  1,x\right)  $ (since $c_{G}\left(
1,x\right)  =1$). This shows that Lemma \ref{lem.PM.morse-c} holds for our
$G$. Thus, Lemma \ref{lem.PM.morse-c} is proved when $\#E=1$. This completes
the base case.

\textit{Induction step:} We fix a directed graph $G=\left(  V,E,s,t\right)  $
with $\#E>1$, and we assume (as induction hypothesis) that Lemma
\ref{lem.PM.morse-c} is true for all graphs with exactly $\#E-1$ arcs.
We must now prove Lemma \ref{lem.PM.morse-c} for our graph $G$.

If $G$ has a useless arc, then we already know (from Step 1) that Lemma
\ref{lem.PM.morse-c} is true. Thus, we WLOG assume that $G$ has no useless
arcs. As in our above proof of Lemma \ref{lem.PM.chitil-c}, we can thus find
an arc $e\in E$ with source $s$. Consider this arc $e$. Lemma
\ref{lem.cpol.rec} yields%
\[
c_{G}\left(  x,y\right)  =xc_{G/e}\left(  x,y\right)  +yc_{G\backslash
e}\left(  x,y\right)  .
\]
Substituting $1$ and $x$ for $x$ and $y$ in this equality, we find%
\begin{align}
c_{G}\left(  1,x\right)
& =c_{G/e}\left(  1,x\right)  +xc_{G\backslash e}\left(  1,x\right)
.\label{pf.lem.PM.morse-c.c-rec}%
\end{align}

However, $\#E>1$ entails $E\not \subseteq \left\{  e\right\}  $, thus
$E\setminus\left\{  e\right\}  \neq\varnothing$. The graph $G/e$ has arc set
$E\setminus\left\{  e\right\}  $, whose size is $\#\left(  E\setminus\left\{
e\right\}  \right)  =\#E-1$. Thus, by our induction hypothesis, Lemma
\ref{lem.PM.morse-c} is true for $G/e$ instead of $G$. In other words, the
complex $\Pm{G/e}  $ has an acyclic matching $\left(
M_{/},\mu_{/}\right)  $ satisfying $u_{\mu_{/}}\left(  x\right)
=c_{G/e}\left(  1,x\right)  $. The same argument (applied to $G\backslash e$
instead of $G/e$) shows that the complex $\Pm{G\backslash e}$ has an
acyclic matching $\left(  M_{\backslash},\mu_{\backslash
}\right)  $ satisfying $u_{\mu_{\backslash}}\left(  x\right)  =c_{G\backslash
e}\left(  1,x\right)  $. Consider these two matchings $\left(  M_{/},\mu
_{/}\right)  $ and $\left(  M_{\backslash},\mu_{\backslash}\right)  $. Thus,
$\left(  M_{/},\mu_{/}\right)  $ is an acyclic matching of the complex
$\Pm{G/e}  =\dl_{\Pm{G}}\left(  e\right)  $ (by
Lemma~\ref{lem:delcon} (b)), whereas $\left(  M_{\backslash},\mu_{\backslash
}\right)  $ is an acyclic matching of the complex $\Pm{G\backslash e}
=\lk_{\Pm{G}}\left(  e\right)  $ (by
Lemma~\ref{lem:delcon} (a)). Hence, Lemma \ref{lem.dmt.rec} (applied to $W=E$
and $\Delta=\Pm{G}$ and $w=e$ and $\left(  M_-%
,\mu_-\right)  =\left(  M_{/},\mu_{/}\right)  $ and $\left(  M_+,\mu
_+\right)  =\left(  M_{\backslash},\mu_{\backslash}\right)  $) shows that
$\Pm{G}$ has an acyclic matching $\left(
M,\mu\right)  $ satisfying%
\[
u_{\mu}\left(  x\right)  =\underbrace{u_{\mu_{/}}\left(  x\right)  }%
_{=c_{G/e}\left(  1,x\right)  }+\,x\underbrace{u_{\mu_{\backslash}}\left(
x\right)  }_{=c_{G\backslash e}\left(  1,x\right)  }=c_{G/e}\left(
1,x\right)  +xc_{G\backslash e}\left(  1,x\right)  =c_{G}\left(  1,x\right)
\]
(by \eqref{pf.lem.PM.morse-c.c-rec}). This shows that Lemma
\ref{lem.PM.morse-c} holds for our $G$. This completes the induction step.
Thus, Lemma \ref{lem.PM.morse-c} is proved by induction.
\end{proof}

\begin{proof}
[Proof of Lemma \ref{lem.PF.morse-c}.]This is analogous to our above proof of
Lemma \ref{lem.PM.morse-c}. The main difference is in the induction step: The
matching $\left(  M_{/},\mu_{/}\right)  $ is now an acyclic matching of the
complex
$\Pf{G/e}  =\lk_{\Pf{G}}\left(  e\right)  $
(by Lemma~\ref{lem:delcon_pf}
(b)), whereas the matching $\left(  M_{\backslash},\mu_{\backslash}\right)  $
is now an acyclic matching of the complex $\Pf{G\backslash e}
=\dl_{\Pf{G}}\left(  e\right)  $
(by Lemma~\ref{lem:delcon_pf} (a)). Hence, Lemma \ref{lem.dmt.rec} (applied to
$W=E$ and $\Delta=\Pf{G}$ and $w=e$ and $\left(
M_-,\mu_-\right)  =\left(  M_{\backslash},\mu_{\backslash}\right)  $ and
$\left(  M_+,\mu_+\right)  =\left(  M_{/},\mu_{/}\right)  $) shows that
$\Pf{G}$ has an acyclic matching $\left(
M,\mu\right)  $ satisfying%
\begin{align*}
u_{\mu}\left(  x\right)
=\underbrace{u_{\mu_{\backslash}}\left(
x\right)  }_{=c_{G\backslash e}\left(  x,1\right)  }+\,x\underbrace{u_{\mu
_{/}}\left(  x\right)  }_{=c_{G/e}\left(  x,1\right)  }
&= c_{G\backslash e}\left(  x,1\right)  +xc_{G/e}\left(  x,1\right)  \\
& =xc_{G/e}\left(  x,1\right)  +c_{G\backslash e}\left(  x,1\right)
=c_{G}\left(  x,1\right)
\end{align*}
(again by a specialization of Lemma \ref{lem.cpol.rec}).
\end{proof}

\begin{proof}
[Proof of Theorem \ref{thm.PM.morse}.] (a) Assume that $G$ has a useless arc
or a cycle. Thus, Definition \ref{def.cpol.cpol} yields $c_{G}\left(
x,y\right)  =0$. Hence, $c_{G}\left(  1,x\right)  =0$.

Also, Lemma \ref{lem.PM.morse-c} shows that the complex
$\Pm{G}$ has an acyclic matching $\left(
M,\mu\right)  $ satisfying $u_{\mu}\left(  x\right)  =c_{G}\left(  1,x\right)
$. In other words, $\Pm{G}$ has an acyclic matching
$\left(  M,\mu\right)  $ satisfying $u_{\mu}\left(  x\right)  =0$ (since
$c_{G}\left(  1,x\right)  =0$). This proves Theorem \ref{thm.PM.morse} (a).

(b) Assume that $G$ has no useless arcs and no cycles.



Hence, Definition \ref{def.cpol.cpol} yields
$c_{G}\left(  x,y\right)  =x^{\#\NS -1}y^{\#E-\#\NS}$ (since
$E \neq \varnothing$). Substituting $1$ and $x$ for $x$ and $y$
in this equality, we find $c_{G}\left(  1,x\right)  =1^{\#\NS%
-1}x^{\#E-\#\NS}=x^{\#E-\#\NS}$.

But Lemma \ref{lem.PM.morse-c} shows that the complex $\Pm{G}$ has an acyclic matching $\left(  M,\mu\right)  $ satisfying
$u_{\mu}\left(  x\right)  =c_{G}\left(  1,x\right)  $. In other words, the
complex $\Pm{G}$ has an acyclic matching $\left(
M,\mu\right)  $ satisfying $u_{\mu}\left(  x\right)  =x^{\#E-\#\NS}$
(since $c_{G}\left(  1,x\right)  =x^{\#E-\#\NS}$). This proves Theorem
\ref{thm.PM.morse} (b).
\end{proof}

\begin{proof}
[Proof of Theorem \ref{thm.PF.morse}.] This follows from Lemma
\ref{lem.PF.morse-c} in the same way as Theorem \ref{thm.PM.morse} was derived
from Lemma \ref{lem.PM.morse-c}.
\end{proof}

As explained above, Theorem \ref{thm.PM.homotopy} and Theorem
\ref{thm.PF.homotopy} follow (respectively) from Theorem \ref{thm.PM.morse}
and Theorem \ref{thm.PF.morse} using Corollary \ref{cor.dmt.hteq}.
(In the proof of Theorem \ref{thm.PM.homotopy}, we need to handle
the $E = \varnothing$ cases separately, but this is easy%
\footnote{Namely, let us assume that $E = \varnothing$.
Then, $\Pm{G} = \varnothing$ if $s \neq t$, and
$\Pm{G} = \set{\varnothing}$ if $s = t$.}.)


\section{\label{sect.further}Further directions}

\subsection{Combinatorial grapes}

Marietti and Testa, in \cite[Definition 3.2]{MariettiTesta}, introduced a certain well-behaved class of simplicial complexes: the \emph{combinatorial grapes}.
We recall their definition (rewritten using our notations) next.

\begin{definition}
The \emph{combinatorial grapes} are a class of simplicial complexes defined recursively:
\begin{itemize}
\item Any simplicial complex $\tup{W, \Delta}$ with $\# W \leq 1$ is a combinatorial grape.
\item Let $\tup{W, \Delta}$ be a simplicial complex. If there exists an $a \in W$ such that both $\lk_\Delta (a)$ and $\dl_\Delta (a)$ are combinatorial grapes, and such that there is a cone $\tup{W \setminus \set{a}, \Gamma}$ satisfying $\lk_\Delta (a) \subseteq \Gamma \subseteq \dl_\Delta (a)$, then $\tup{W, \Delta}$ is a combinatorial grape.\footnote{Here, the statement $\lk_\Delta (a) \subseteq \Gamma \subseteq \dl_\Delta (a)$ means that each face of $\lk_\Delta (a)$ is a face of $\Gamma$, and that each face of $\Gamma$ is a face of $\dl_\Delta (a)$.}
\end{itemize}
\end{definition}

It is shown in \cite[Proposition 3.3]{MariettiTesta} that combinatorial grapes have a rather simple homotopy type (viz., they are disjoint unions of points or wedges of spheres).

An even more restrictive notion is that of a \emph{strong grape}, which we define as follows.

\begin{definition}
The \emph{strong grapes} are a class of simplicial complexes defined recursively:
\begin{itemize}
\item Any simplicial complex $\tup{W, \Delta}$ with $\# W \leq 1$ is a strong grape.
\item Let $\tup{W, \Delta}$ be a simplicial complex. If there exists an $a \in W$ such that both $\lk_\Delta (a)$ and $\dl_\Delta (a)$ are strong grapes, and such that at least one of $\lk_\Delta (a)$ and $\dl_\Delta (a)$ is a cone, then $\tup{W, \Delta}$ is a strong grape.
\end{itemize}
\end{definition}

Clearly, every strong grape is a combinatorial grape.
We shall now show that both complexes $\Pf{G}$ and $\Pm{G}$ for a graph $G$ are strong grapes.

\begin{proposition}
\label{prop.strong-grapes.PfPm}
Both $\Pf{G}$ and $\Pm{G}$ are strong grapes.
\end{proposition}

\begin{proof}
We proceed by strong induction on $\# E$.
Thus, we assume (as induction hypothesis) that both $\Pf{H}$ and $\Pm{H}$ are strong grapes whenever $H$ is a graph with fewer arcs than $G$.
We must now show that both $\Pf{G}$ and $\Pm{G}$ are strong grapes.
This is obvious if $\# E \leq 1$, so we WLOG assume that $\# E > 1$.

The case when $s = t$ is easy\footnote{%
Indeed, in this case, it is clear that any subset of $E$
contains an $s-t$-path, and thus we have
$\Pf{G} = \varnothing$ and $\Pm{G} = 2^E$.
Thus, all that needs to be proved is that both $\tup{E, 2^E}$
and $\tup{E, \varnothing}$ are strong grapes.
But this is straightforward to do by induction on $\# E$
(taking any $a \in E$ in the induction step).}.
Thus, we WLOG assume that $s \neq t$.

The case when $G$ has no $s-t$-path is easy\footnote{%
Indeed, in this case, it can easily be checked that
$\Pf{G} = 2^E$ and $\Pm{G} = \varnothing$.
Thus, all that needs to be proved is that both $\tup{E, 2^E}$
and $\tup{E, \varnothing}$ are strong grapes.
But this was done in the previous footnote.}.
Thus, we WLOG assume that $G$ has at least one $s-t$-path.
This $s-t$-path has at least one arc (since $s \neq t$).
Thus, its first arc is well-defined.
This first arc must have source $s$ and is not useless (since it
belongs to an $s-t$-path).
Hence, there exists an arc $e$ with source $s$ that is not useless.
Consider this arc $e$.
Let $s'$ be its target.

The graph $G / e$ has fewer arcs than $G$.
Hence, by the induction hypothesis, both
$\Pf{G / e}$ and $\Pm{G / e}$ are strong grapes.
Likewise, both
$\Pf{G \backslash e}$ and $\Pm{G \backslash e}$ are strong grapes.

Lemma~\ref{lem:delcon_pf} (a) yields
that $\dl_{\Pf{G}}(e) = \Pf{G \backslash e}$,
whereas
Lemma~\ref{lem:delcon_pf} (b) yields
that $\lk_{\Pf{G}}(e) = \Pf{G / e}$.
Likewise, Lemma~\ref{lem:delcon} (a) yields
that $\lk_{\Pm{G}}(e) = \Pm{G \backslash e}$,
whereas
Lemma~\ref{lem:delcon} (b) yields
that $\dl_{\Pm{G}}(e) = \Pm{G / e}$.

Recall that $\Pf{G / e}$ and $\Pf{G \backslash e}$ are
strong grapes.
In other words, $\lk_{\Pf{G}}(e)$ and $\dl_{\Pf{G}}(e)$
are strong grapes
(since $\dl_{\Pf{G}}(e) = \Pf{G \backslash e}$ and
$\lk_{\Pf{G}}(e) = \Pf{G / e}$).
Thus, if we can show that at least one of
$\lk_{\Pf{G}}(e)$ and $\dl_{\Pf{G}}(e)$ is a cone,
then we can conclude that $\Pf{G}$ is a strong grape
(by the definition of a strong grape).

Likewise, if we can show that at least one of
$\lk_{\Pm{G}}(e)$ and $\dl_{\Pm{G}}(e)$ is a cone,
then we can conclude that $\Pm{G}$ is a strong grape
(since $\lk_{\Pm{G}}(e) = \Pm{G \backslash e}$ and
$\dl_{\Pm{G}}(e) = \Pm{G / e}$ are strong grapes).

We are in one of the following two cases:

\textit{Case 1:} The arc $e$ is not the only arc with target $s'$.

\textit{Case 2:} The arc $e$ is the only arc with target $s'$.

Let us first consider Case 1.
In this case, the arc $e$ is not the only arc with target $s'$.
Hence, Lemma~\ref{lem:induction_step_1} shows that
$G / e$ has a useless arc.
Let $e'$ be this useless arc.
Thus, Lemma~\ref{lem:dumbarc} (applied to $G / e$ and $e'$ instead
of $G$ and $e$) shows that
both $\Pf{G / e}$ and $\Pm{G / e}$ are cones with apex $e'$.
In other words,
both $\lk_{\Pf{G}}(e)$ and $\dl_{\Pm{G}}(e)$ are cones
(since $\lk_{\Pf{G}}(e) = \Pf{G / e}$
and $\dl_{\Pm{G}}(e) = \Pm{G / e}$).
In particular, at least one of
$\lk_{\Pf{G}}(e)$ and $\dl_{\Pf{G}}(e)$ is a cone
(namely, $\lk_{\Pf{G}}(e)$).
As we have seen above, this entails that $\Pf{G}$ is a strong grape.
Furthermore, at least one of
$\lk_{\Pm{G}}(e)$ and $\dl_{\Pm{G}}(e)$ is a cone
(namely, $\dl_{\Pm{G}}(e)$).
As we have seen above, this entails that $\Pm{G}$ is a strong grape.

We thus have shown that
both $\Pf{G}$ and $\Pm{G}$ are strong grapes.
This completes the induction step in Case 1.

Let us now consider Case 2.
In this case, the arc $e$ is the only arc with target $s'$.
Hence, Lemma~\ref{lem:induction_step_2b} yields that
$G \backslash e$ has a useless arc.
Let $e'$ be this useless arc.
Thus, Lemma~\ref{lem:dumbarc} (applied to $G \backslash e$ and $e'$ instead
of $G$ and $e$) shows that
both $\Pf{G \backslash e}$ and $\Pm{G \backslash e}$ are cones with apex $e'$.
In other words,
both $\dl_{\Pf{G}}(e)$ and $\lk_{\Pm{G}}(e)$ are cones
(since $\lk_{\Pm{G}}(e) = \Pm{G \backslash e}$
and $\dl_{\Pf{G}}(e) = \Pf{G \backslash e}$).
In particular, at least one of
$\lk_{\Pf{G}}(e)$ and $\dl_{\Pf{G}}(e)$ is a cone
(namely, $\dl_{\Pf{G}}(e)$).
As we have seen above, this entails that $\Pf{G}$ is a strong grape.
Furthermore, at least one of
$\lk_{\Pm{G}}(e)$ and $\dl_{\Pm{G}}(e)$ is a cone
(namely, $\lk_{\Pm{G}}(e)$).
As we have seen above, this entails that $\Pm{G}$ is a strong grape.

We thus have shown that
both $\Pf{G}$ and $\Pm{G}$ are strong grapes.
This completes the induction step in Case 2.

We have now completed the induction step in each of
the two Cases 1 and 2.
Hence, the induction step is complete,
and Proposition~\ref{prop.strong-grapes.PfPm} is proved.
\end{proof}

\begin{remark}
It is easy to see (recursively, using Lemma~\ref{lem.dmt.rec})
that any strong grape 
has an acyclic matching $\tup{M, \mu}$ satisfying
$u_\mu\tup{x} = 0$ or $u_\mu\tup{x} = x^m$ for some
$m \in \NN$.
However, it does not seem easy to describe which of
these two alternatives holds, nor what $m$ is.
Thus, this does not yield alternative proofs of
Theorems~\ref{thm.PM.morse} and~\ref{thm.PF.morse}.
\end{remark}

\subsection{Open questions}

While the homotopy types of $\Pm{G}$ and $\Pf{G}$ answer
many natural questions about these complexes, some further
questions remain.

\begin{question}
Let $k$ be a nonnegative integer, and assume that $G$ has at least $k$ disjoint quasi-cycles.
What can we say about the remainders of the polynomials $f_{\Pm{G}}\tup{x}$ and $f_{\Pf{G}}\tup{x}$ modulo $\tup{1+x}^{k+1}$ ?
\end{question}

\begin{question}
We proved our claims about $\Pm{G}$ and $\Pf{G}$ by induction.
Can they be proved more directly? In particular, can the
acyclic matchings we recursively constructed on $\Pm{G}$ and
$\Pf{G}$ be described directly?
\end{question}

\begin{question}
We can generalize $\Pm{G}$ and $\Pf{G}$ as follows:
Fix a positive integer $r$.
Define two complexes
    \begin{align*}
        \Pf{G, r} &:= \set{ F \subseteq E \with F \text{ contains no $r$ arc-disjoint $s-t$-paths}} , \\
        \Pm{G, r} &:= \set{ F \subseteq E \with E\setminus F \text{ contains $r$ arc-disjoint $s-t$-paths}}
    \end{align*}
on the ground set $E$.
What can we say about these two complexes?
The Euler characteristics no longer restrict themselves to the values $0, 1, -1$
(for example, if $G$ consists of two vertices $s$ and $t$ and $k$ arcs from $s$ to $t$,
then $\chitil\tup{\Pf{G, r}} = \tup{-1}^r \dbinom{k-1}{r-1}$
and  $\chitil\tup{\Pm{G, r}} = \tup{-1}^{k+r-1} \dbinom{k-1}{r-1}$).
Still, can anything be said about these complexes?  E.g., are they bouquets of spheres? combinatorial grapes?
\end{question}

The latter question is of particular interest in that
any interesting homotopical properties of the complexes
$\Pf{G, r}$ and $\Pm{G, r}$ would suggest a topological
undercurrent in the theory of network flows.
(Indeed, the existence of $r$ arc-disjoint $s-t$-paths
is equivalent to the existence of an $s-t$-flow of
value $r$; see \cite[Theorem 4.2 and paragraph thereafter]{ForFul62}.)

\appendix


\section{\label{sect.apx-morse}Proofs of Morse theory basics}

In this appendix, we shall give proofs for two lemmas left unproved
in Section~\ref{sect.dmt}.

\begin{proof}[Proof of Lemma \ref{lem.dmt.cone}.]
Let $W$ be the ground set of $\Delta$. Let $w\in W$ be such that $\Delta$ is a
cone with apex $w$. (Such a $w$ exists by assumption.)

For each subset $A$ of $W$ that contains $w$, we set $A^-:=A\setminus
\left\{  w\right\}  $. If $A\in\Delta$ contains $w$, then $A^-\in\Delta$
(since $A^-=A\setminus\left\{  w\right\}  \subseteq A$) and $w\notin A^-$
(by definition).

For each subset $A$ of $W$ that does not contain $w$, we set $A^+%
:=A\cup\left\{  w\right\}  $. If $A\in\Delta$ does not contain $w$, then
$A^+\in\Delta$ (since $\Delta$ is a cone with apex $w$, so that $A\in\Delta$
entails $A\cup\left\{  w\right\}  \in\Delta$) and $w\in A^+$ (by definition).

Define a map $\mu:\Delta\rightarrow\Delta$ by setting%
\[
\mu\left(  A\right)  =%
\begin{cases}
A^+, & \text{if }w\notin A;\\
A^-, & \text{if }w\in A
\end{cases}
\qquad\text{for each }A\in\Delta.
\]
This map $\mu$ is well-defined (by the previous two paragraphs) and is an
involution (since each subset $A$ of $W$ satisfies $\left(  A^-\right)
^+=A$ if it contains $w$, and $\left(  A^+\right)  ^-=A$ if it does
not). Moreover, each $A\in\Delta$ satisfies either $\mu\left(  A\right)  \prec
A$ or $\mu\left(  A\right)  \succ A$ (indeed, if $w\in A$, then $\mu\left(
A\right)  =A^-=A\setminus\left\{  w\right\}  \prec A$, whereas otherwise we
have $\mu\left(  A\right)  =A^+=A\cup\left\{  w\right\}  \succ A$). Thus,
$\left(  \Delta,\mu\right)  $ is a matching on $\Delta$. It has no unmatched
faces. Thus, by the definition of $u_{\mu}\left(  x\right)  $, we have%
\[
u_{\mu}\left(  x\right)  =\sum_{I\in\Delta\setminus\Delta}x^{\#I}=\left(
\text{empty sum}\right)  =0.
\]

It remains to show that this matching $\left(  \Delta,\mu\right)  $ is acyclic.

To show this, we let $\left(  F_{1},F_{2},\ldots,F_{n}\right)  $ be a cycle of
$\mu$. Thus, $F_{1},F_{2},\ldots,F_{n}$ are distinct faces of $\Delta$,
satisfying $n\geq2$ and%
\[
F_{1}\succ\mu\left(  F_{1}\right)  \prec F_{2}\succ\mu\left(  F_{2}\right)
\prec F_{3}\succ\cdots\prec F_{n}\succ\mu\left(  F_{n}\right)  \prec F_{1}.
\]
Our matching $\mu$ has the property that the only faces $A\in\Delta$ that
satisfy $A\succ\mu\left(  A\right)  $ are the faces that contain $w$. Thus,
from $F_{1}\succ\mu\left(  F_{1}\right)  $, we conclude that $F_{1}$ contains
$w$. Similarly, $F_{2}$ contains $w$. The definition of $\mu$ yields
$\mu\left(  F_{1}\right)  =F_{1}^-$ (since $F_{1}$ contains $w$), so that
$\mu\left(  F_{1}\right)  $ does not contain $w$.

However, $\mu\left(  F_{1}\right)  \prec F_{2}$, so that the set difference
$F_{2}\setminus\mu\left(  F_{1}\right)  $ has exactly $1$ element. This
element must be $w$ (since $F_{2}$ contains $w$, but $\mu\left(  F_{1}\right)
$ does not). Therefore, $F_{2}\setminus\mu\left(  F_{1}\right)  =\left\{
w\right\}  $. Since $\mu\left(  F_{1}\right)  \subseteq F_{2}$ (because
$\mu\left(  F_{1}\right)  \prec F_{2}$), this entails $F_{2}=\mu\left(
F_{1}\right)  \cup\left\{  w\right\}  $. On the other hand, $F_{1}=\mu\left(
F_{1}\right)  \cup\left\{  w\right\}  $ (since $\mu\left(  F_{1}\right)
=F_{1}^-=F_{1}\setminus\left\{  w\right\}  $). Comparing these two
equalities, we obtain $F_{1}=F_{2}$. This contradicts the fact that the faces
$F_{1},F_{2},\ldots,F_{n}$ are distinct.

Forget that we fixed $\left(  F_{1},F_{2},\ldots,F_{n}\right)  $. We thus have
found a contradiction for every cycle $\left(  F_{1},F_{2},\ldots
,F_{n}\right)  $ of $\mu$. Hence, the matching $\mu$ has no cycle, i.e., is
acyclic. This completes the proof of Lemma \ref{lem.dmt.cone}.
\end{proof}

\begin{proof}[Proof of Lemma \ref{lem.dmt.rec}.]
For each subset $A$ of $W$ that contains $w$, we set $A^-:=A\setminus
\left\{  w\right\}  $. This set $A^-$ always satisfies $w\notin A^-$ and
$A^-\prec A$ and $\#\left(  A^-\right)  =\#A-1$. Moreover, if $A\in\Delta$
contains $w$, then $A^-\in\Delta$ (since $A^-=A\setminus\left\{
w\right\}  \subseteq A$).

For each subset $A$ of $W$ that does not contain $w$, we set $A^+%
:=A\cup\left\{  w\right\}  $. This set $A^+$ always satisfies $w\in A^+$
and $A^+\succ A$ and $\#\left(  A^+\right)  =\#A+1$. However, if
$A\in\Delta$ does not contain $w$, then we don't always have $A^+\in\Delta$.

Note that each subset $A$ of $W$ satisfies
\begin{equation}
\left(  A^-\right)  ^+=A\ \ \ \ \ \ \ \ \ \ \text{if }w\in A,
\label{pf.lem.dmt.rec.-+}%
\end{equation}
and satisfies%
\begin{equation}
\left(  A^+\right)  ^-=A\ \ \ \ \ \ \ \ \ \ \text{if }w\notin A.
\label{pf.lem.dmt.rec.+-}%
\end{equation}

Moreover, the operations $A\mapsto A^-$ and $A\mapsto A^+$ preserve
covering relations: i.e.,

\begin{itemize}
\item If $A$ and $B$ are two subsets of $W$ that contain $w$, then%
\begin{equation}
A\prec B \quad \text{implies} \quad A^-\prec B^-. \label{pf.lem.dmt.rec.prec-}%
\end{equation}
[\textit{Proof:} Let $A$ and $B$ be two subsets of $W$ that contain $w$.
Assume that $A \prec B$. Now, $A\prec B$ means that
$B = A \cup \set{b}$ for some $b\in B\setminus A$. Consider
this $b$. From $b \in B \setminus A$, we obtain $b \notin A$ and thus
$b \neq w$ (because $b \notin A$ and $w \in A$).
Combining $b \in B$ with $b \neq w$, we obtain
$b \in B \setminus \set{w} = B^-$. Moreover, $b \notin A$ (since
$b \in B \setminus A$) and thus $b \notin A \setminus \set{w} = A^-$.
Combining $b \in B^-$ with $b \notin A^-$, we obtain
$b \in B^- \setminus A^-$.
Now, $B^- = B \setminus \set{w} = \tup{A \cup \set{b}} \setminus \set{w}$
(since $B = A \cup \set{b}$), so that
$B^- = \tup{A \cup \set{b}} \setminus \set{w} = \tup{A \setminus \set{w}}
\cup \set{b}$ (since $b \neq w$). In view of $A \setminus \set{w} = A^-$,
we can rewrite this as $B^- = A^- \cup \set{b}$.
Since $b \in B^- \setminus A^-$, this equality shows that
$A^- \prec B^-$. This proves \eqref{pf.lem.dmt.rec.prec-}.] \\

\item If $A$ and $B$ are two subsets of $W$ that don't contain $w$, then%
\begin{equation}
A\prec B \quad \text{implies} \quad A^+\prec B^+. \label{pf.lem.dmt.rec.prec+}%
\end{equation}
[\textit{Proof:} Let $A$ and $B$ be two subsets of $W$ that don't
contain $w$. Assume that $A \prec B$. Now, $A\prec B$ means that
$B = A \cup \set{b}$ for some $b\in B\setminus A$. Consider
this $b$. From $b \in B \setminus A$, we obtain $b \in B$ and thus
$b \neq w$ (because $b \in B$ and $w \notin B$).
Combining $b \notin A$ with $b \neq w$, we obtain
$b \notin A \cup \set{w} = A^+$. Moreover, $b \in B \setminus A
\subseteq B \subseteq B \cup \set{w} = B^+$.
Combining $b \in B^+$ with $b \notin A^+$, we obtain
$b \in B^+ \setminus A^+$.
Now, $B^+ = B \cup \set{w} = \tup{A \cup \set{b}} \cup \set{w}$
(since $B = A \cup \set{b}$), so that
$B^+ = \tup{A \cup \set{b}} \cup \set{w} = \tup{A \cup \set{w}}
\cup \set{b} = A^+ \cup \set{b}$
(since $A \cup \set{w} = A^+$).
Since $b \in B^+ \setminus A^+$, this equality shows that
$A^+ \prec B^+$. This proves \eqref{pf.lem.dmt.rec.prec+}.]
\end{itemize}


The set $\Delta\setminus\dl_{\Delta}\left(  w\right)
$ consists of all faces $A\in\Delta$ that contain $w$. Removing $w$ from such
a face yields a face of $\lk_{\Delta}\left(  w\right)
$. In other words, $A^-\in\lk_{\Delta}\left(
w\right)  $ for each $A\in\Delta\setminus\dl_{\Delta
}\left(  w\right)  $. Conversely, inserting $w$ into a face of
$\lk_{\Delta}\left(  w\right)  $ yields a face
$A\in\Delta$ that contains $w$, that is, a face in $\Delta\setminus
\dl_{\Delta}\left(  w\right)  $. In other words,
$A^+\in\Delta\setminus\dl_{\Delta}\left(  w\right)
$ for each $A\in\lk_{\Delta}\left(  w\right)  $. Thus,
we have two mutually inverse bijections%
\begin{align}
\Delta\setminus\dl_{\Delta}\left(  w\right)   &
\rightarrow\lk_{\Delta}\left(  w\right)  ,\nonumber\\
A  &  \mapsto A^-=A\setminus\left\{  w\right\}  \label{pf.lem.dmt.rec.bij-}%
\end{align}
and%
\begin{align}
\lk_{\Delta}\left(  w\right)   &  \rightarrow
\Delta\setminus\dl_{\Delta}\left(  w\right)
,\nonumber\\
A  &  \mapsto A^+=A\cup\left\{  w\right\}  . \label{pf.lem.dmt.rec.bij+}%
\end{align}

Recall that $M_+\subseteq\lk_{\Delta}\left(
w\right)  \subseteq 2^{W\setminus\left\{  w\right\}  }$, so that each face
$B\in M_+$ is a subset of $W\setminus\left\{  w\right\}  $. In other words,
each face $B\in M_+$ is a subset of $W$ that does not contain $w$. Hence, we
can set%
\[
M_+^+:=\left\{  B^+\ \mid\ B\in M_+\right\}  .
\]
Then, $M_+^+\subseteq\Delta$ (since each $B\in M_+$ satisfies $B\in
M_+\subseteq\lk_{\Delta}\left(  w\right)  $ and thus
$B\cup\left\{  w\right\}  \in\Delta$, so that $B^+=B\cup\left\{  w\right\}
\in\Delta$). Moreover, the set $M_+^+$ only contains faces that contain
$w$ (since $B^+$ contains $w$ for each $B\in M_+$). Thus, every face $A\in
M_+^+$ satisfies $w\in A$ and therefore
$A\notin \dl_{\Delta}\left(  w\right)  $ (since
$A\in\dl_{\Delta}\left(  w\right)  $ would mean that $w\notin A$), so that
$A\in\Delta\setminus\dl_{\Delta}\left(  w\right)  $
(since $A\in M_+^+\subseteq\Delta$). In other words, $M_+^+%
\subseteq\Delta\setminus\dl_{\Delta}\left(  w\right)
$.

Furthermore, $M_-\subseteq\dl_{\Delta}\left(
w\right)  \subseteq2^{W\setminus\left\{  w\right\}  }$. Hence, each face $B\in
M_-$ is a subset of $W\setminus\left\{  w\right\}  $, and thus satisfies
$w\notin B$. In other words, $M_-$ only contains faces that don't contain
$w$.

We now define a set
\begin{equation}
M:=M_-\cup M_+^+. \label{pf.lem.dmt.rec.M=}%
\end{equation}
Note that $M\subseteq\Delta$ (since
$M_-\subseteq\dl_{\Delta}\left(  w\right)  \subseteq\Delta$ and
$M_+^+ \subseteq\Delta$). Moreover, the union on the right hand side of
\eqref{pf.lem.dmt.rec.M=} is a disjoint union (since the set $M_-$ only
contains faces that don't contain $w$, whereas the set $M_+^+$ only
contains faces that contain $w$).

We observe the following:

\begin{statement}
\textit{Claim 1:} Let $A\in M$ be such that $w\notin A$. Then, $\mu_-\left(
A\right)  \in M$.
\end{statement}

\begin{proof}
[Proof of Claim 1.]The face $A$ does not contain $w$ (since $w\notin A$).
Hence, $A\notin M_+^+$ (since the set $M_+^+$ only contains faces that
contain $w$). Combining this with $A\in M$, we obtain $A\in M\setminus
M_+^+\subseteq M_-$ (since $M=M_-\cup M_+^+$). Thus, $\mu
_-\left(  A\right)  $ is well-defined and belongs to $M_-$ (since $\mu
_-$ is a map $M_-\rightarrow M_-$). Therefore, $\mu_-\left(  A\right)
\in M_-\subseteq M$ (since $M=M_-\cup M_+^+$). This proves Claim 1.
\end{proof}

\begin{statement}
\textit{Claim 2:} Let $A\in M$ be such that $w\in A$. Then, $\left(  \mu
_+\left(  A^-\right)  \right)  ^+\in M$.
\end{statement}

\begin{proof}
[Proof of Claim 2.]The face $A$ contains $w$ (since $w\in A$). Hence, $A\notin
M_-$ (since the set $M_-$ only contains faces that don't contain $w$).
Combining this with $A\in M$, we obtain $A\in M\setminus M_-\subseteq
M_+^+$ (since $M=M_-\cup M_+^+$). Thus, $A\in M_+^+=\left\{
B^+\ \mid\ B\in M_+\right\}  $. In other words, $A=B^+$ for some $B\in
M_+$. Consider this $B$. From $A=B^+$, we obtain $A^-=\left(
B^+\right)  ^-=B\in M_+$. Thus $\mu_+\left(  A^-\right)  $ is
well-defined and belongs to $M_+$ (since $\mu_+$ is a map $M_+%
\rightarrow M_+$). Hence, $\mu_+\left(  A^-\right)  \in M_+$, so that
$\left(  \mu_+\left(  A^-\right)  \right)  ^+\in M_+^+$ (by the
definition of $M_+^+$), and therefore $\left(  \mu_+\left(
A^-\right)  \right)  ^+\in M_+^+\subseteq M$ (since $M=M_-\cup
M_+^+$). This proves Claim 2.
\end{proof}

Combining Claim 1 with Claim 2, we see that
\[%
\begin{cases}
\mu_-\left(  A\right)  , & \text{if }w\notin A;\\
\left(  \mu_+\left(  A^-\right)  \right)  ^+, & \text{if }w\in A
\end{cases}
\ \ \in M\ \ \ \ \ \ \ \ \ \ \text{for each }A\in M.
\]
This lets us construct a map%
\begin{align*}
\mu:M  &  \rightarrow M,\\
A  &  \mapsto%
\begin{cases}
\mu_-\left(  A\right)  , & \text{if }w\notin A;\\
\left(  \mu_+\left(  A^-\right)  \right)  ^+, & \text{if }w\in A.
\end{cases}
\end{align*}
We shall show that this map $\mu$ is an acyclic matching on $\Delta$. First,
we argue that it is an involution:

\begin{statement}
\textit{Claim 3:} We have $\mu\circ\mu=\operatorname*{id}$.
\end{statement}

\begin{proof}
[Proof of Claim 3.]It is enough to show that $\mu\left(  \mu\left(  A\right)
\right)  =A$ for each $A\in M$. So let us fix $A\in M$. We are in one of the
following two cases:

\textit{Case 1:} We have $w\in A$.

\textit{Case 2:} We have $w\notin A$.

First consider Case 1. In this case, we have $w\in A$. The definition of $\mu$
thus yields $\mu\left(  A\right)  =\left(  \mu_+\left(  A^-\right)
\right)  ^+$. Set $C:=\mu\left(  A\right)  $. Thus, $C=\mu\left(  A\right)
=\left(  \mu_+\left(  A^-\right)  \right)  ^+$, so that $C^-=\left(
\left(  \mu_+\left(  A^-\right)  \right)  ^+\right)  ^-=\mu_+\left(
A^-\right)  $ (by \eqref{pf.lem.dmt.rec.+-}). Hence, $\mu_+\left(
C^-\right)  =\mu_+\left(  \mu_+\left(  A^-\right)  \right)  =A^-$
(since $\mu_+\circ\mu_+=\operatorname*{id}$ (because $\mu_+$ is a
matching)) and therefore $\left(  \mu_+\left(  C^-\right)  \right)
^+=\left(  A^-\right)  ^+=A$ (by \eqref{pf.lem.dmt.rec.-+}).

Obviously, $w\in\left(  \mu_+\left(  A^-\right)  \right)  ^+$ (since
$w\in B^+$ for any $B$). In other words, $w\in C$ (since $C=\left(  \mu
_+\left(  A^-\right)  \right)  ^+$). Hence, the definition of $\mu$
yields $\mu\left(  C\right)  =\left(  \mu_+\left(  C^-\right)  \right)
^+=A$. Since $\mu\left(  A\right)  =C$, we now have $\mu\left(  \mu\left(
A\right)  \right)  =\mu\left(  C\right)  =A$. Thus, $\mu\left(  \mu\left(
A\right)  \right)  =A$ is proved in Case 1.

Now, consider Case 2. Here, we have $w\notin A$. The definition of $\mu$ thus
yields $\mu\left(  A\right)  =\mu_-\left(  A\right)  \in M_-$ (since
$\mu_-$ is a map from $M_-$ to $M_-$). Thus, the face $\mu\left(
A\right)  $ does not contain $w$ (since the set $M_-$ only contains faces
that don't contain $w$). In other words, $w\notin\mu\left(  A\right)  $.
Hence, the definition of $\mu$ yields $\mu\left(  \mu\left(  A\right)
\right)  =\mu_-\left(  \mu\left(  A\right)  \right)  =\mu_-\left(  \mu
_-\left(  A\right)  \right)  $ (since $\mu\left(  A\right)  =\mu_-\left(
A\right)  $). But $\mu_-\left(  \mu_-\left(  A\right)  \right)  =A$ (since
$\mu_-\circ\mu_-=\operatorname*{id}$ (because $\mu_-$ is a matching)).
Thus, $\mu\left(  \mu\left(  A\right)  \right)  =\mu_-\left(  \mu_-\left(
A\right)  \right)  =A$. Hence, $\mu\left(  \mu\left(  A\right)  \right)  =A$
is proved in Case 2.

We have now proved $\mu\left(  \mu\left(  A\right)  \right)  =A$ in both Cases
1 and 2. Hence, $\mu\left(  \mu\left(  A\right)  \right)  =A$ always holds,
and Claim 3 is proved.
\end{proof}

Next, we argue that $\mu$ is a matching:

\begin{statement}
\textit{Claim 4:} Each $A\in M$ satisfies either $\mu\left(  A\right)  \prec
A$ or $\mu\left(  A\right)  \succ A$.
\end{statement}

\begin{proof}
[Proof of Claim 4.]Let $A\in M$. We are in one of the following two cases:

\textit{Case 1:} We have $w\in A$.

\textit{Case 2:} We have $w\notin A$.

First consider Case 1. In this case, we have $w\in A$. The definition of $\mu$
thus yields $\mu\left(  A\right)  =\left(  \mu_+\left(  A^-\right)
\right)  ^+$. However, since $\mu_+$ is a matching, we must have either
$\mu_+\left(  A^-\right)  \prec A^-$ or $\mu_+\left(  A^-\right)
\succ A^-$. Since both $\mu_+\left(  A^-\right)  $ and $A^-$ are
subsets of $W$ that don't contain $w$, we thus conclude using
\eqref{pf.lem.dmt.rec.prec+} that we have either $\left(  \mu_+\left(
A^-\right)  \right)  ^+\prec\left(  A^-\right)  ^+$ or $\left(
\mu_+\left(  A^-\right)  \right)  ^+\succ\left(  A^-\right)  ^+$. In
view of $\left(  \mu_+\left(  A^-\right)  \right)  ^+=\mu\left(
A\right)  $ and $\left(  A^-\right)  ^+=A$, we can rewrite this as
follows: We have either $\mu\left(  A\right)  \prec A$ or $\mu\left(
A\right)  \succ A$. Thus, Claim 4 is proved in Case 1.

Now, consider Case 2. In this case, we have $w\notin A$. The definition of
$\mu$ thus yields $\mu\left(  A\right)  =\mu_-\left(  A\right)  $. However,
since $\mu_-$ is a matching, we must have either $\mu_-\left(  A\right)
\prec A$ or $\mu_-\left(  A\right)  \succ A$. In view of $\mu_-\left(
A\right)  =\mu\left(  A\right)  $, we can rewrite this as follows: We have
either $\mu\left(  A\right)  \prec A$ or $\mu\left(  A\right)  \succ A$. Thus,
Claim 4 is proved in Case 2.

We have now proved Claim 4 in both Cases 1 and 2; thus, Claim 4 always holds.
\end{proof}

Claim 3 and Claim 4 reveal that $\left(  M,\mu\right)  $ is a matching on
$\Delta$. Next, we claim:

\begin{statement}
\textit{Claim 5:} This matching $\left(  M,\mu\right)  $ is acyclic.
\end{statement}

\begin{proof}
[Proof of Claim 5.]Let us first observe that any cycle of $\mu$ can be
rotated: If $\left(  F_{1},F_{2},\ldots,F_{n}\right)  $ is a cycle of $\mu$,
then $\left(  F_{i},F_{i+1},\ldots,F_{n},F_{1},F_{2},\ldots,F_{i-1}\right)  $
is also a cycle of $\mu$ for each $i\in\left\{  1,2,\ldots,n\right\}  $
(because the definition of a cycle has rotational symmetry).

Let $\left(  F_{1},F_{2},\ldots,F_{n}\right)  $ be a cycle of $\mu$. Thus,
$F_{1},F_{2},\ldots,F_{n}$ are distinct faces in $M$, satisfying $n\geq2$ and%
\begin{equation}
F_{1}\succ\mu\left(  F_{1}\right)  \prec F_{2}\succ\mu\left(  F_{2}\right)
\prec F_{3}\succ\cdots\prec F_{n}\succ\mu\left(  F_{n}\right)  \prec F_{1}.
\label{pf.lem.dmt.rec.c5.pf.chain}%
\end{equation}
We will derive a contradiction.

Let us set $F_{n+1}:=F_{1}$, so that each $j\in\left\{  1,2,\ldots,n\right\}
$ satisfies%
\begin{equation}
F_{j}\succ\mu\left(  F_{j}\right)  \prec F_{j+1}
\label{pf.lem.dmt.rec.c5.pf.1}%
\end{equation}
(by \eqref{pf.lem.dmt.rec.c5.pf.chain}). We distinguish between the following
two cases:

\textit{Case 1:} Some $i\in\left\{  1,2,\ldots,n\right\}  $ satisfies $w\in
F_{i}$.

\textit{Case 2:} No $i\in\left\{  1,2,\ldots,n\right\}  $ satisfies $w\in
F_{i}$.

We consider Case 1 first. In this case, some $i\in\left\{  1,2,\ldots
,n\right\}  $ satisfies $w\in F_{i}$. By rotating the cycle $\left(
F_{1},F_{2},\ldots,F_{n}\right)  $ to $\left(  F_{i},F_{i+1},\ldots
,F_{n},F_{1},F_{2},\ldots,F_{i-1}\right)  $, we can ensure that this $i$
becomes $1$. Thus, we WLOG assume that $i=1$.

Hence, $w\in F_{1}$ (since $w\in F_{i}$). Therefore, the definition of $\mu$
yields $\mu\left(  F_{1}\right)  =\left(  \mu_+\left(  F_{1}^-\right)
\right)  ^+$. But $w\in\left(  \mu_+\left(  F_{1}^-\right)  \right)
^+$ (since $w\in B^+$ for any set $B$). In other words, $w\in\mu\left(
F_{1}\right)  $ (since $\mu\left(  F_{1}\right)  =\left(  \mu_+\left(
F_{1}^-\right)  \right)  ^+$). However, \eqref{pf.lem.dmt.rec.c5.pf.chain}
also yields $\mu\left(  F_{1}\right)  \prec F_{2}$, so that $\mu\left(
F_{1}\right)  \subseteq F_{2}$ and therefore $w\in\mu\left(  F_{1}\right)
\subseteq F_{2}$.

We have thus proved $w\in F_{2}$ using \eqref{pf.lem.dmt.rec.c5.pf.chain} and
$w\in F_{1}$. Likewise, we can prove $w\in F_{3}$ (using
\eqref{pf.lem.dmt.rec.c5.pf.chain} and $w\in F_{2}$) and $w\in F_{4}$ (using
\eqref{pf.lem.dmt.rec.c5.pf.chain} and $w\in F_{3}$) and so on. This line of
reasoning shows eventually (i.e., by induction on $j$) that $w\in F_{j}$ for
each $j\in\left\{  1,2,\ldots,n\right\}  $.

Now, let $j\in\left\{  1,2,\ldots,n\right\}  $. As we just showed, we have
$w\in F_{j}$, so that $\mu\left(  F_{j}\right)  =\left(  \mu_+\left(
F_{j}^-\right)  \right)  ^+$ (by the definition of $\mu$). Thus, $\left(
\mu\left(  F_{j}\right)  \right)  ^-=\left(  \left(  \mu_+\left(
F_{j}^-\right)  \right)  ^+\right)  ^-=\mu_+\left(  F_{j}^-\right)
$ (by \eqref{pf.lem.dmt.rec.+-}).

But $w\in\left(  \mu_+\left(  F_{j}^-\right)  \right)  ^+$ (since $w\in
B^+$ for any set $B$). In other words, $w\in\mu\left(  F_{j}\right)  $
(since $\mu\left(  F_{j}\right)  =\left(  \mu_+\left(  F_{j}^-\right)
\right)  ^+$). Hence, $w\in\mu\left(  F_{j}\right)  \subseteq F_{j+1}$
(since \eqref{pf.lem.dmt.rec.c5.pf.1} says that $\mu\left(  F_{j}\right)
\prec F_{j+1}$). This shows that $F_{j+1}^-$ is well-defined.

From \eqref{pf.lem.dmt.rec.c5.pf.1}, we know that $F_{j}\succ\mu\left(
F_{j}\right)  \prec F_{j+1}$. Since $F_{j}$, $\mu\left(  F_{j}\right)  $ and
$F_{j+1}$ are subsets of $W$ that contain $w$ (because $w\in F_{j}$ and
$w\in\mu\left(  F_{j}\right)  $ and $w\in F_{j+1}$), we thus conclude using
\eqref{pf.lem.dmt.rec.prec-} that $F_{j}^-\succ\left(  \mu\left(
F_{j}\right)  \right)  ^-\prec F_{j+1}^-$. In other words, $F_{j}^-%
\succ\mu_+\left(  F_{j}^-\right)  \prec F_{j+1}^-$ (since $\left(
\mu\left(  F_{j}\right)  \right)  ^-=\mu_+\left(  F_{j}^-\right)  $).

Forget that we fixed $j$. We thus have shown that $F_{j}^-\succ\mu
_+\left(  F_{j}^-\right)  \prec F_{j+1}^-$ for each $j\in\left\{
1,2,\ldots,n\right\}  $. In other words,
\begin{equation}
      F_{1}^- \succ \mu_+\left(  F_{1}^-\right)
\prec F_{2}^- \succ \mu_+\left(  F_{2}^-\right)
\prec F_{3}^- \succ \cdots
\prec F_{n}^- \succ \mu_+\left(  F_{n}^-\right)
\prec F_{1}^-
\label{pf.lem.dmt.rec.c5.pf.c1.chain}
\end{equation}
(since $F_{n+1}=F_{1}$). Moreover, the $n$ sets $F_{1}^-,F_{2}^-,
\ldots,F_{n}^-$ are distinct\footnote{\textit{Proof.} Assume the contrary.
Thus, some $i<j$ in $\left\{  1,2,\ldots,n\right\}  $ satisfy $F_{i}^-
=F_{j}^-$. Consider these $i<j$. Then, \eqref{pf.lem.dmt.rec.-+} yields
$\left(  F_{i}^-\right)  ^+=F_{i}$ and $\left(  F_{j}^-\right)
^+=F_{j}$. Thus, $F_{i}=\left(  F_{i}^-\right)  ^+=\left(  F_{j}%
^-\right)  ^+$ (since $F_{i}^-=F_{j}^-$), so that $F_{i}=\left(  F_{j}%
^-\right)  ^+=F_{j}$. But this contradicts the distinctness of
$F_{1},F_{2},\ldots,F_{n}$. This contradiction shows that our assumption was
false.} and belong to $M_+$\ \ \ \ \footnote{\textit{Proof.} Let
$j\in\left\{  1,2,\ldots,n\right\}  $. We must show that $F_{j}^-\in M_+$.
\par
We have $F_{j}\in M=M_-\cup M_+^+$. However, the set $F_{j}$ contains
$w$ (since $F_{j}^-$ is defined in the first place), whereas the set $M_-$
only contains faces that don't contain $w$. Thus, $F_{j}\notin M_-$.
Combining this with $F_{j}\in M_-\cup M_+^+$, we obtain $F_{j}\in\left(
M_-\cup M_+^+\right)  \setminus M_-\subseteq M_+^+=\left\{
B^+\ \mid\ B\in M_+\right\}  $. In other words, $F_{j}=B^+$ for some
$B\in M_+$. Consider this $B$.
\par
From $F_{j}=B^+$, we obtain $F_{j}^-=\left(  B^+\right)  ^-=B$ (by
\eqref{pf.lem.dmt.rec.+-}), so that $F_{j}^-=B\in M_+$, qed.}. Thus, they
are $n$ distinct faces of $M_+$. Since they satisfy
\eqref{pf.lem.dmt.rec.c5.pf.c1.chain} (and since $n\geq2$), we thus conclude
that $\left(  F_{1}^-,F_{2}^-,\ldots,F_{n}^-\right)  $ is a cycle of the
matching $\left(  M_+,\mu_+\right)  $. Thus, $\left(  M_+,\mu
_+\right)  $ has a cycle, i.e., is not acyclic. But this contradicts the
assumption that $\left(  M_+,\mu_+\right)  $ is acyclic. Thus, we have
found a contradiction in Case 1.

Let us now consider Case 2. In this case, no $i\in\left\{  1,2,\ldots
,n\right\}  $ satisfies $w\in F_{i}$.

In other words, none of the $n$ faces $F_{1},F_{2},\ldots,F_{n}$ contains $w$.
Thus, these $n$ faces $F_{1},F_{2},\ldots,F_{n}$ belong to $M_-%
$\ \ \ \ \footnote{\textit{Proof.} Let $j\in\left\{  1,2,\ldots,n\right\}  $.
We must show that $F_{j}\in M_-$.
\par
We have $F_{j}\in M=M_-\cup M_+^+$. However, the set $F_{j}$ does not
contain $w$ (since none of the $n$ faces $F_{1},F_{2},\ldots,F_{n}$ contains
$w$), whereas the set $M_+^+$ only contains faces that contain $w$. Thus,
$F_{j}\notin M_+^+$. Combining this with $F_{j}\in M_-\cup M_+^+$,
we obtain $F_{j}\in\left(  M_-\cup M_+^+\right)  \setminus M_+%
^+\subseteq M_-$, qed.}.

For each $j\in\left\{  1,2,\ldots,n\right\}  $, we have $w\notin F_{j}$ (since
none of the $n$ faces $F_{1},F_{2},\ldots,F_{n}$ contains $w$) and therefore
$\mu\left(  F_{j}\right)  =\mu_-\left(  F_{j}\right)  $ (by the definition
of $\mu$). Hence, we can rewrite the chain \eqref{pf.lem.dmt.rec.c5.pf.chain}
as follows:%
\[
F_{1}\succ\mu_-\left(  F_{1}\right)  \prec F_{2}\succ\mu_-\left(
F_{2}\right)  \prec F_{3}\succ\cdots\prec F_{n}\succ\mu_-\left(
F_{n}\right)  \prec F_{1}.
\]
This chain (combined with the fact that $F_{1},F_{2},\ldots,F_{n}$ are $n$
distinct faces in $M_-$, and the fact that $n\geq2$) shows that $\left(
F_{1},F_{2},\ldots,F_{n}\right)  $ is a cycle of the matching $\left(
M_-,\mu_-\right)  $. Thus, $\left(  M_-,\mu_-\right)  $ has a cycle,
i.e., is not acyclic. But this contradicts the assumption that $\left(
M_-,\mu_-\right)  $ is acyclic. Thus, we have found a contradiction in
Case 2.

We have now found a contradiction in each case.

Forget that we fixed $\left(  F_{1},F_{2},\ldots,F_{n}\right)  $. We thus have
found a contradiction for each cycle $\left(  F_{1},F_{2},\ldots,F_{n}\right)
$ of $\mu$. Hence, $\mu$ has no cycle. In other words, $\mu$ is acyclic. Claim
5 is thus proved.
\end{proof}

In preparation for our next step, we prove two more simple claims:

\begin{statement}
\textit{Claim 6:} The map%
\begin{align*}
\lk_{\Delta}\left(  w\right)  \setminus M_+  &
\rightarrow\left\{  I\in\Delta\setminus M\ :\ w\in I\right\}  ,\\
A  &  \mapsto A^+%
\end{align*}
is well-defined and is a bijection.
\end{statement}

\begin{proof}
[Proof of Claim 6.]We shall prove that this map is well-defined, injective and surjective.

\begin{proof}
[Well-definedness:]Let us first prove that this map is well-defined. To this
purpose, we must show that $A^+\in\left\{  I\in\Delta\setminus M\ :\ w\in
I\right\}  $ for each $A\in\lk_{\Delta}\left(  w\right)  \setminus M_+$.

Let us fix $A\in\lk_{\Delta}\left(  w\right)
\setminus M_+$. Then, $A\in\lk_{\Delta}\left(
w\right)  $ and $A\notin M_+$.

From $A\in\lk_{\Delta}\left(  w\right)  $, we obtain
$A\subseteq W\setminus\left\{  w\right\}  $ and $A\cup\left\{  w\right\}
\in\Delta$ (by the definition of $\lk_{\Delta}\left(
w\right)  $). The definition of $A^+$ yields $A^+=A\cup\left\{  w\right\}
\in\Delta$ and $w\in A^+$. Moreover, it is easy to see that $A^+\notin
M_-$\ \ \ \ \footnote{\textit{Proof:} Assume the contrary. Thus, $A^+\in
M_-\subseteq\dl_{\Delta}\left(  w\right)
\subseteq2^{W\setminus\left\{  w\right\}  }$. Hence, $A^+\subseteq
W\setminus\left\{  w\right\}  $, so that $w\notin A^+$. But this contradicts
$w\in A^+$. This contradiction shows that our assumption was false.} and
$A^+\notin M_+^+$\ \ \ \ \footnote{\textit{Proof.} Assume the contrary.
Thus, $A^+\in M_+^+=\left\{  B^+\ \mid\ B\in M_+\right\}  $. In
other words $A^+=B^+$ for some $B\in M_+$. Consider this $B$. Now,
\eqref{pf.lem.dmt.rec.+-} yields $A=\left(  \underbrace{A^+}_{=B^+%
}\right)  ^-=\left(  B^+\right)  ^-=B$ (by \eqref{pf.lem.dmt.rec.+-}).
Thus, $A=B\in M_+$, but this contradicts $A\notin M_+$. This contradiction
shows that our assumption was false.}. Combining these two facts, we obtain
$A^+\notin M_-\cup M_+^+=M$ (by \eqref{pf.lem.dmt.rec.M=}). Combining
this with $A^+\in\Delta$, we find $A^+\in\Delta\setminus M$. Thus, $A^+$
is a face $I\in\Delta\setminus M$ satisfying $w\in I$ (since $w\in A^+$). In
other words, $A^+\in\left\{  I\in\Delta\setminus M\ :\ w\in I\right\}  $.
This completes the proof that our map is well-defined.
\end{proof}

\begin{proof}
[Injectivity:]Let us next prove that our map is injective.

For this purpose, we must show that any two faces
$A,B\in\lk_{\Delta}\left(  w\right)  \setminus M_+$ satisfying
$A^+=B^+$ must satisfy $A=B$.

But this is easy: If $A,B\in\lk_{\Delta}\left(
w\right)  \setminus M_+$ are two faces satisfying $A^+=B^+$, then
\eqref{pf.lem.dmt.rec.+-} yields $A=\left(  \underbrace{A^+}_{=B^+%
}\right)  ^-=\left(  B^+\right)  ^-=B$ (by \eqref{pf.lem.dmt.rec.+-}).
Thus, our map is injective.
\end{proof}

\begin{proof}
[Surjectivity:]Let us now prove that our map is surjective.

Let $C\in\left\{  I\in\Delta\setminus M\ :\ w\in I\right\}  $. We must prove
that $C=A^+$ for some $A\in\lk_{\Delta}\left(
w\right)  \setminus M_+$.

Indeed, $C\in\left\{  I\in\Delta\setminus M\ :\ w\in I\right\}  $ shows that
$C\in\Delta\setminus M$ and $w\in C$. From $C\in\Delta\setminus M$, we obtain
$C\in\Delta$ and $C\notin M$. From $w\in C$, we conclude that $C^-$ is
well-defined. Moreover, \eqref{pf.lem.dmt.rec.-+} yields $\left(
C^-\right)  ^+=C\in\Delta$. Now, we have $C^-\subseteq W\setminus
\left\{  w\right\}  $ and $C^-\cup\left\{  w\right\}  =\left(  C^-\right)
^+\in\Delta$, so that $C^-\in\lk_{\Delta}\left(
w\right)  $. If we had $C^-\in M_+$, then we would have $\left(
C^-\right)  ^+\in M_+^+$ (by the definition of $M_+^+$), which
would entail $C=\left(  C^-\right)  ^+\in M_+^+\subseteq M_-\cup
M_+^+=M$ (by \eqref{pf.lem.dmt.rec.M=}); but this would contradict
$C\notin M$. Thus, we cannot have $C^-\in M_+$. Hence, we obtain
$C^-\notin M_+$. Combining this with
$C^-\in\lk_{\Delta}\left(  w\right)  $, we obtain
$C^-\in\lk_{\Delta}\left(  w\right)  \setminus M_+$. Since $C=\left(
C^-\right)  ^+$, we thus have $C=A^+$ for some
$A\in\lk_{\Delta}\left(  w\right)  \setminus M_+$ (namely, for $A=C^-$).

Forget that we fixed $C$. We thus have shown that each $C\in\left\{
I\in\Delta\setminus M\ :\ w\in I\right\}  $ can be written as $C=A^+$ for
some $A\in\lk_{\Delta}\left(  w\right)  \setminus
M_+$. In other words, our map is surjective.
\end{proof}

We have now proved that the map
\begin{align*}
\lk_{\Delta}\left(  w\right)  \setminus M_+  &
\rightarrow\left\{  I\in\Delta\setminus M\ :\ w\in I\right\}  ,\\
A  &  \mapsto A^+%
\end{align*}
is well-defined, injective and surjective. Hence, it is a bijection. Claim 6
is thus proved.{}
\end{proof}

\begin{statement}
\textit{Claim 7:} We have%
\[
\left\{  I\in\Delta\setminus M\ :\ w\notin I\right\}
= \dl_{\Delta}\left(  w\right)  \setminus M_-.
\]

\end{statement}

\begin{proof}
[Proof of Claim 7.]The sets $\dl_{\Delta}\left(
w\right)  $ and $M_+^+$ are disjoint (since every face $A\in
\dl_{\Delta}\left(  w\right)  $ satisfies $w\notin A$,
whereas every face $A\in M_+^+$ satisfies $w\in A$). Hence,
$\dl_{\Delta}\left(  w\right)  \setminus M_+%
^+=\dl_{\Delta}\left(  w\right)  $.

However,
\begin{align*}
\left\{  I\in\Delta\setminus M\ \mid\ w\notin I\right\}   &  =\left\{
I\in\Delta\ \mid\ I\notin M\text{ and }w\notin I\right\}
=\underbrace{\left\{  I\in\Delta\ \mid\ w\notin I\right\}  }%
_{=\dl_{\Delta}\left(  w\right)  }\setminus\,M\\
&  =\dl_{\Delta}\left(  w\right)  \setminus
\underbrace{M}_{\substack{=M_-\cup M_+^+\\=M_+^+\cup M_-%
}}=\dl_{\Delta}\left(  w\right)  \setminus\left(
M_+^+\cup M_-\right) \\
&  =\underbrace{\left(  \dl_{\Delta}\left(  w\right)
\setminus M_+^+\right)  }_{=\dl_{\Delta}\left(
w\right)  }\setminus\,M_-=\dl_{\Delta}\left(
w\right)  \setminus M_-.
\end{align*}
This proves Claim 7.
\end{proof}

Finally, we compute the unmatched f-polynomial of the acyclic matching $\mu$.
The definition of the unmatched f-polynomial yields%
\begin{equation}
u_{\mu_-}\left(  x\right)  =\sum_{I\in\dl_{\Delta
}\left(  w\right)  \setminus M_-}x^{\#I} \label{pf.lem.dmt.rec.u-}%
\end{equation}
and%
\begin{equation}
u_{\mu_+}\left(  x\right)  =\sum_{I\in\lk_{\Delta
}\left(  w\right)  \setminus M_+}x^{\#I} \label{pf.lem.dmt.rec.u+}%
\end{equation}
and%
\begin{equation}
u_{\mu}\left(  x\right)  =\sum_{I\in\Delta\setminus M}x^{\#I}=\sum
_{\substack{I\in\Delta\setminus M;\\w\in I}}x^{\#I}+\sum_{\substack{I\in
\Delta\setminus M;\\w\notin I}}x^{\#I}. \label{pf.lem.dmt.rec.u}%
\end{equation}

But Claim 6 yields that the map%
\begin{align*}
\lk_{\Delta}\left(  w\right)  \setminus M_+  &
\rightarrow\left\{  I\in\Delta\setminus M\ :\ w\in I\right\}  ,\\
A  &  \mapsto A^+%
\end{align*}
is a bijection. Hence, we can substitute $A^+$ for $I$ in the sum
$\sum\limits_{\substack{I\in\Delta\setminus M;\\w\in I}}x^{\#I}$. Thus, we
obtain
\begin{align*}
\sum_{\substack{I\in\Delta\setminus M;\\w\in I}}x^{\#I}  &  =\sum
_{A\in\lk_{\Delta}\left(  w\right)  \setminus M_+%
}\underbrace{x^{\#\left(  A^+\right)  }}_{\substack{=x^{\#A+1}\\\text{(since
}\#\left(  A^+\right)  =\#A+1\text{)}}}
=\sum_{A\in\lk_{\Delta}\left(  w\right)  \setminus M_+}x^{\#A+1}\\
&  =\sum_{I\in\lk_{\Delta}\left(  w\right)  \setminus
M_+}\underbrace{x^{\#I+1}}_{=xx^{\#I}}=x\underbrace{\sum_{I\in
\lk_{\Delta}\left(  w\right)  \setminus M_+}x^{\#I}%
}_{\substack{=u_{\mu_+}\left(  x\right)  \\\text{(by
\eqref{pf.lem.dmt.rec.u+})}}}=xu_{\mu_+}\left(  x\right)  .
\end{align*}
Moreover, the summation sign $\sum\limits_{\substack{I\in\Delta\setminus
M;\\w\notin I}}$ is equivalent to
$\sum\limits_{I\in\dl_{\Delta}\left(  w\right)  \setminus M_-}$
(by Claim 7). Thus,
\[
\sum_{\substack{I\in\Delta\setminus M;\\w\notin I}}x^{\#I}
= \sum\limits_{I\in \dl_{\Delta}\left(  w\right)  \setminus M_-}
x^{\#I}
=u_{\mu_-}\left(  x\right)  \ \ \ \ \ \ \ \ \ \ \left(  \text{by
\eqref{pf.lem.dmt.rec.u-}}\right)  .
\]
Thus, \eqref{pf.lem.dmt.rec.u} becomes
\[
u_{\mu}\left(  x\right)  =\underbrace{\sum_{\substack{I\in\Delta\setminus
M;\\w\in I}}x^{\#I}}_{=xu_{\mu_+}\left(  x\right)  }+\underbrace{\sum
_{\substack{I\in\Delta\setminus M;\\w\notin I}}x^{\#I}}_{=u_{\mu_-}\left(
x\right)  }=xu_{\mu_+}\left(  x\right)  +u_{\mu_-}\left(  x\right)
=u_{\mu_-}\left(  x\right)  +xu_{\mu_+}\left(  x\right)  .
\]
The proof of Lemma \ref{lem.dmt.rec} is thus complete.
\end{proof}

\printbibliography

\end{document}